\documentclass[11pt]{amsart}
\usepackage{amssymb,amsmath,amsfonts,amsthm,mathrsfs,mathtools}
\usepackage{tikz-cd,soul,graphics,enumerate}
\usetikzlibrary{positioning,arrows,scopes,calc,decorations.markings}
\usepackage{geometry}       % See geometry.pdf to learn the layout options. There are lots.
\definecolor{cadmiumgreen}{rgb}{0.0, 0.42, 0.24}
\usepackage[
colorlinks, citecolor=cadmiumgreen,
pagebackref,
pdfauthor={David Harry Richman}, 
pdfstartview ={FitV},
]{hyperref}
\hypersetup{
pdftitle={The distribution of Weierstrass points on a tropical curve},
}

\mathtoolsset{showonlyrefs}

\newtheorem{thm}{Theorem}[section]
\newtheorem*{thm*}{Theorem}
\newtheorem{innercustomthm}{Theorem}
\newenvironment{customthm}[1]
  {\renewcommand\theinnercustomthm{#1}\innercustomthm}
  {\endinnercustomthm}
\newtheorem{prop}[thm]{Proposition}
\newtheorem{lem}[thm]{Lemma}
\newtheorem{cor}[thm]{Corollary}

\theoremstyle{definition}
\newtheorem{dfn}[thm]{Definition}

\newtheorem{eg}[thm]{Example}
\newtheorem{rmk}[thm]{Remark}
\newtheorem{conj}[thm]{Conjecture}

\newcommand{\CC}{\mathbb{C}}

\newcommand{\KK}{\mathbb{K}}
\newcommand{\RR}{\mathbb{R}}

\newcommand{\ZZ}{\mathbb{Z}}

\newcommand{\PP}{\mathbb{P}}

\DeclareMathOperator{\val}{val}

\DeclareMathOperator{\Sym}{Sym}
\DeclareMathOperator{\Div}{Div}
\DeclareMathOperator{\Eff}{Eff}

\DeclareMathOperator{\Pic}{Pic}
\DeclareMathOperator{\Jac}{Jac}

\DeclareMathOperator{\coker}{coker}

\DeclareMathOperator{\PL}{PL}
\DeclareMathOperator{\rS}{S}
\DeclareMathOperator{\Divisor}{\Delta}

\DeclareMathOperator{\red}{red}
\DeclareMathOperator{\stred}{st.red}
\DeclareMathOperator{\zeros}{\Divisor^{+}}%{zeros}
\DeclareMathOperator{\poles}{\Divisor^{--}}%{poles}

\DeclareMathOperator{\br}{br} % break-divisor-rep function
\DeclareMathOperator{\Br}{Br}
\DeclareMathOperator{\cone}{cone} % positive cone in R^n
\newcommand{\Xan}{X^{\mathrm{an}}}
\newcommand{\Ztwo}{Z^{(2)}} % codimension 2 skeleton
\newcommand{\Utwo}{U^{(2)}} % open complement of Z^2

\newcommand{\reddiv}[2]{\red_{#1} [#2]}
\newcommand{\reddivst}[2]{\red^\mathrm{st}_{#1} [#2]}
\newcommand{\streddiv}[2]{\stred_{#1} [#2]}
\newcommand{\potent}[2]{j_{#1}^{#2}}

\newcommand{\divfiber}[1]{{#1}^{-1}_{\Divisor}}
\newcommand{\Wstab}{W^\mathrm{st}}

\newcommand{\startv}{s}
\newcommand{\tailv}{t}
\newcommand{\en}{n}
\newcommand{\length}[1]{\ell({#1})}
\newcommand{\lengtheff}[1]{\ell_\text{eff}({#1})}

\begin{document}

\title[Tropical Weierstrass points]{The distribution of Weierstrass points on a tropical curve}
\author{David Harry Richman}

\thanks{This work was partially supported by NSF grant DMS-1600223}

% 14T15: Combinatorial aspects of tropical varieties
% 05C22: Signed and weighted graphs
% 14H55: Riemann surfaces; Weierstrass points; gap sequences
% 57M12: Low-dimensional topology of special (e.g. branched) coverings
% 60B10: Convergence of probability measures
% 94C05: Analytic circuit theory
%%% unused
% 05C10: Planar graphs; geometric and topological aspects of graph theory

\date{\today}

\begin{abstract}
We show that on a metric graph of genus $g$,
a divisor of degree $\en$ generically has $g(n-g+1)$ Weierstrass points.
For a sequence of generic divisors on a metric graph whose degrees grow to infinity, 
we show that the associated Weierstrass points become distributed according to the Zhang canonical measure.
In other words, the limiting distribution is determined by effective resistances on the metric graph.
This distribution result has an analogue for complex algebraic curves, due to Neeman,
and for curves over non-Archimedean fields, due to Amini.
%However, the results in this paper are 
%purely combinatorial statements which are
%proved using combinatorial arguments
%rather than algebraic or analytic geometry.
\end{abstract}
%\vspace*{-0.4in}
\maketitle

\setcounter{tocdepth}{1}
\tableofcontents

\section{Introduction}
For any divisor class on an algebraic curve,
there is an associated finite set of Weierstrass points.
In this paper we study the set of   
(tropical) Weierstrass points 
associated to a divisor class
on an abstract tropical curve.
In particular, we ask 
\begin{itemize}
	\item[(A)] When is the set of Weierstrass points finite? If so, how many are there?
\end{itemize}
and
\begin{itemize}
	\item[(B)] How are these points distributed as the degree approaches infinity?
\end{itemize}
We show that, for any abstract tropical curve $\Gamma$, 
the Weierstrass locus is finite for a generic divisor class. 
Generically, the number of Weierstrass points depends only 
on the divisor degree and the genus of the underlying tropical curve.
We further prove that, for any degree-increasing sequence of such generic divisors, 
the Weierstrass points become distributed on $\Gamma$ according to the Zhang canonical measure.
This measure can be described in terms of effective resistances when interpreting $\Gamma$ as an electrical network of resistors.

We also define a stable Weierstrass locus which is finite for
an arbitrary divisor class, and compute its cardinality for 
a generic divisor class,
which depends only on the degree and genus.
The contents of this paper first appeared in the author's PhD thesis \cite{R-thesis}.

\subsection{Statement of results}
Given a compact, connected metric graph $\Gamma$ and a divisor $D$ of rank $r = r(D)$,
we define the {\em Weierstrass locus} ${ W(D) }$ as 
\[ 
	W(D) = \{ x\in \Gamma : \reddiv{x}{D} \geq (r+1)x \},
\]
% \[ W(D) = \{ x\in \Gamma : D \sim (r+1)x + E \text{ for some }E \geq 0\},\]
where $\reddiv{x}{D}$ denotes the reduced divisor in $[D]$ with respect to $x$
%where $\sim$ denotes linear equivalence of divisors
%and $r(D)$ is the Baker--Norine rank
(see Section~\ref{sec:trop-curves} for definitions).
The set $W(D)$ may fail to be finite.
This fact was first observed by Baker and Norine~\cite[Example 4.5]{Bak}, 
see also Example~\ref{eg:3loop-chain}.

For a divisor of degree $\en\geq g$, we define the {\em stable Weierstrass locus} of $D$ as
\[ \Wstab(D) = \{ x\in\Gamma : \streddiv{x}{D} \geq (n-g+1)x \}\]
% \[ \Wstab(D) = \{ x\in\Gamma : \br[D-(\en-g)x] = x + E \text{ for some }E \geq 0\}\]
where $\streddiv{x}{D}$ denotes the stable reduced divisor in $[D]$ with respect to $x$.
%where $\br[D]$ denotes the unique break divisor representative of a degree $g$ divisor $D$.
The stable Weierstrass locus is  finite for any divisor.
If $D$ has rank $r(D) = \en - g$, 
i.e. $D$ is nonspecial, 
then the stable Weierstrass locus is contained in $W(D)$.
In particular, this containment holds when the degree satisfies $\en \geq 2g - 1$.
See Section~\ref{sec:trop-curves} for definitions of 
rank, reduced divisor, and stable reduced divisor.

Our first result addresses the question of 
counting the number of Weierstrass points.
Here ``generic'' means %the condition holds 
on a dense open subset of the space of divisor classes.

\begin{customthm}{A}
\label{thm:w-finite}
Let $\Gamma$ be a compact, connected metric graph of genus $g$.
\begin{enumerate}[(a)]
\item 
For a generic divisor class of degree $n\geq g$,
the  Weierstrass locus $W(D)$ is finite with cardinality
\[ \# W(D) = g(n-g+1) .\]
For a generic divisor class of degree $n < g$,  $W(D)$ is empty.

\item
For an arbitrary divisor class of degree $\en\geq g$,
the stable Weierstrass locus $\Wstab(D)$ is a finite set with cardinality
\[ \# \Wstab(D) \leq g(\en-g+1),\]
and equality holds for a generic divisor class.
\end{enumerate}
\end{customthm}
\noindent Parts (a) and (b) 
of Theorem~\ref{thm:w-finite}
are connected by showing that $W(D) = \Wstab(D)$ for a generic divisor class.

The next main theorem  of our paper
describes the distribution of tropical Weierstrass points.
Here, note that the condition 
``$W_\en = W(D_\en)$ is a finite set''
is satisfied for a generic divisor class $[D_\en] \in \Pic^\en(\Gamma)$ by Theorem~\ref{thm:w-finite}.
\begin{customthm}{B}
\label{thm:equidistribution}
Let $\Gamma$ be a metric graph of genus $g$,
and let $\{ D_{\en} : \en\geq 1\} $ be a sequence of divisors on $\Gamma$
 with $\deg D_{\en} = \en$.
Let $W_{\en}$ be the Weierstrass locus of $D_{\en}$. 
Suppose each $W_\en$ is a finite set, and let
\[ \delta_{\en}  = \frac1{\en} \sum_{x\in W_{\en}} \delta_x\]
denote the normalized discrete measure on $\Gamma$ associated to $W_{\en}$,
	where $\delta_x$ is the Dirac measure at $x$. 
Then as $\en\to \infty$, the measures $\delta_{\en}$ converge weakly to the Zhang canonical measure $\mu$ on $\Gamma$.
\end{customthm}
The Zhang canonical measure is defined in Section \ref{sec:canonical}.
This measure has a succinct description in terms of effective resistances; see Corollary~\ref{cor:canonical-via-resistance}.
(Warning: we use a different normalization for $\mu$ than previous authors;
namely we have total measure $\mu(\Gamma) = g$ rather than $\mu(\Gamma) = 1$.)
We also obtain a quantitative version of this distribution 
result which specifies a bound on the  rate of convergence.
\begin{customthm}{C}
\label{thm:quant-equidistribution}
Let $\Gamma$ be a metric graph of genus $g$,
let $D_{\en}$ be a divisor of degree $\en > g$, 
and let $W_{\en} = W(D_\en)$ denote the  Weierstrass locus. 
Suppose $W_\en$ is finite.
Let $\mu$ denote the Zhang canonical measure on $\Gamma$.
\begin{enumerate}[(a)]
\item For any segment $e$ in $\Gamma$,
\begin{equation*}
\en \mu(e) - 2 g \leq \# (W_\en \cap e) \leq \en \mu(e) + g + 2 .
\end{equation*}
%or,
%\begin{equation*}
% (\en -g) \mu(e) - g \leq \# (W_{\en}\cap e) \leq (\en -g)\mu(e) + g + 2.
%\en \mu(e) - g(1+\mu(e)) \leq \# (W_\en \cap e) \leq \en \mu(e) + g(1-\mu(e)) + 2 .
%\end{equation*}

\item If $e$ is a segment of $\Gamma$ with  $\mu(e) > {2g}/{\en}$,
then $e$ contains at least one Weierstrass point of $D_n$.

\item For a fixed continuous function $f: \Gamma \to \RR$, as $n\to \infty$
\begin{equation*}
\frac1{\en}\sum_{x \in W_{\en}} f(x) = \int_\Gamma f(x) \mu(dx) + O\left(\frac1{\en}\right).
\end{equation*}
(The big-$O$ constant depends on $f$, but is independent of 
the divisors $D_n$.)

\end{enumerate}
\end{customthm}
It is likely that the bounds in part (a) can be improved.
%(For the technical definition of a ``segment'' of $\Gamma$ see Section \ref{subsec:metric-graphs}.)

\subsection{Previous work}
The set of ordinary Weierstrass points on a complex algebraic curve of genus $g\geq 2$ 
has been a classical object of study (see e.g. \cite{Mir}).
This is a set of $g^3 - g$ points on $X$ (counting with multiplicity) which are intrinsic to $X$ as an abstract curve, without reference to any (non-canonical) embedding of $X$ into an ambient space.
They form a useful tool, e.g. for proving that the automorphism group of such a curve is finite.
This notion naturally extends to higher Weierstrass points (or higher-order Weierstrass points), 
which is a finite set of points on $X$ associated to a choice of divisor class on $X$.
%(With multiplicity, the number of points associated to $(X,[D])$ is $g(r+1)^2$ where $r+1 = h^0(D)$.)
The number of higher Weierstrass points (counted with multiplicity) 
grows quadratically as a function of the degree of $[D]$.
In this paper, we focus on this more general notion of higher Weierstrass points.
We refer to higher Weierstrass points of $D$ simply as Weierstrass points of $D$.

The following useful intuition is given by Mumford \cite{Mum}:
the Weierstrass points associated to a divisor of degree $\en$
form a higher-genus analogue of the set of $\en$-torsion points on an elliptic curve.
Just as choosing a different origin for the group law on a genus one curve 
leads to a different set of torsion points,
choosing different degree $\en$ divisors will give you different sets of Weierstrass points.
The fact that $\en$-torsion points on a complex elliptic curve 
become ``evenly distributed'' as $\en$ grows large
leads one to ask whether the same phenomenon holds for Weierstrass points
on other curves. 
%(on a curve of higher genus).

An answer %for complex algebraic curves 
was given by Neeman \cite{N}, 
% (a student of Mumford),
who showed that for any complex curve (i.e. Riemann surface) of genus $g\geq 2$,
when $\en\to \infty$ the Weierstrass points of degree $\en$ divisors
become distributed according to the Bergman measure.
\begin{thm}[Neeman \cite{N}]
\label{thm:neeman}
Let $X$ be a compact Riemann surface of genus $g \geq 2$,
and let $\{ D_{\en} : \en\geq 1\} $ be a sequence of divisors on $X$ with $\deg D_{\en} = \en$.
Let $W_{\en}$ denote the Weierstrass locus of the divisor $D_{\en}$, and let
$ \delta_{\en}  = \frac1{g \en^2} \sum_{x\in W_{\en}} \delta_x $
denote the normalized discrete measure on $X$ associated to $W_{\en}$
(where $\delta_x$ is the Dirac measure at $x$). 
Then as $\en\to \infty$, the  measures $  \delta_{\en}$ converge weakly 
to the Bergman measure  on $X$.
\end{thm}
\noindent Before Neeman's result, Olsen \cite{Ols} showed that
given a positive-degree divisor $D$ on a complex algebraic curve %(i.e. Riemann surface) 
$X$,
the union of the Weierstrass points of the multiples $nD$, 
over all $n\geq 1$, 
is dense in $X$ in the complex topology.

If one replaces the ground field $\CC$ with a non-Archimedean field, %with non-Archimedean valuation,
one may consider the same question of how Weierstrass points are distributed
inside the Berkovich analytification $\Xan$
of an algebraic curve,
say after retracting to a compact skeleton $\Gamma$.
This was addressed by Amini in  \cite{A-weier}.
Here the Weierstrass points are distributed
according to the {\em Zhang canonical admissible measure},
constructed by Zhang in \cite{Z}.
\begin{thm}[Amini \cite{A-weier}]
\label{thm:amini}
Let $X$ be a smooth proper curve of genus $g \geq 1$ over a complete, algebraically closed, non-Archimedean field $\KK$ with non-trivial valuation and residue characteristic $0$.
Let $\Gamma$ be a skeleton of the Berkovich analytification $\Xan$ with retraction map $\rho : \Xan \to \Gamma$.
Let $D$ be a positive-degree divisor on $X(K)$. 
Let $W_{\en}$ denote the Weierstrass locus of the divisor $\en D$, and let
$ \delta_{\en}  = \frac1{\# W_\en} \sum_{x\in W_{\en}} \delta_{\rho(x)} $
denote the normalized discrete measure on $\Gamma$ associated to $W_{\en}$
(where $\delta_x$ is the Dirac measure at $x$). 
Then as $\en\to \infty$, the  measures $  \delta_{\en}$ converge weakly to the Zhang canonical measure  on $\Gamma$,
up to a factor of $g$.
\end{thm}
Zhang's canonical measure does not have support on bridge edges, 
so it is independent of the choice of skeleton.
Zhang's construction was motivated 
 by Arakelov's pairing for divisors on a Riemann surface \cite{Ar},
%on an arithemtic surface, 
for the purpose of answering questions in arithmetic geometry.
Here we use a definition of $\mu$ along more elementary lines from
Chinburg--Rumely \cite{CR} and  Baker--Faber \cite{BF}, 
using the notions of current flow and electric potential in a network of resistors.

In \cite{A-weier}
Amini raised the question of whether 
the distribution of Weierstrass points is possibly intrinsic to the metric graph $\Gamma$,
without needing to identify $\Gamma$ with the skeleton of some Berkovich curve $\Xan$.
One major obstacle to this idea is that on a metric graph, 
the  Weierstrass locus for a divisor may fail to be a finite set of points.
Our approach is to sidestep this issue by showing that finiteness does hold for a {\em generic} choice of divisor class.
With this assumption of genericity, we are able to show that 
distribution of Weierstrass points is intrinsic to $\Gamma$.

Baker~\cite{Bak} was the first to define and study ordinary Weierstrass points on graphs and on metric graphs,
and mentions several applications of number-theoretic significance.
\cite[Example 4.5]{Bak}, attributed to Norine, describes a family of metric graphs whose ordinary Weierstrass locus contains infinitely many points.
These results concerned Weierstrass points associated to the canonical divisor, 
without considering higher Weierstrass points for general divisors.

%\todo{discuss number of tropical Weierstrass points vs algebraic Weierstrass points, suggestion of uniform multiplicity, relation to Eisenbud--Harris result on degenerating Weierstrass points}
Earlier, Eisenbud--Harris \cite{EH} studied the limiting behavior of ordinary Weierstrass points along a family of smooth curves degenerating to a nodal curve.
This is connected to the present work via the familiar tropical perspective that a family of curves degenerating to a nodal curve may be encoded by the dual graph of the nodal curve, with edge lengths that record the speed of degeneration at each node.
If $\Gamma$ is a metric graph and $D$ is a generic divisor on $\Gamma$ in the sense of Theorem~\ref{thm:w-finite}, 
then 
the conclusions of Theorem~\ref{thm:w-finite} apply to the limit Weierstrass points of $(\mathcal X_t, \mathcal D_t)$
for {\em any} family of smooth curves $\mathcal X_t$ degenerating to a totally degenerate nodal curve with dual graph $\Gamma$, and family of nonspecial divisors $\mathcal D_t$ tropicalizing to $D$.
The number of Weierstrass points on an algebraic curve is $g(n - g + 1)^2$, 
while the number of tropical Weierstrass points is $g(n - g + 1)$.
This seems to suggest that in this scenario, the specialization map is $(n - g + 1)$-to-$1$ on Weierstrass points.
(See Appendix~\ref{sec:tropicalization}.)

Technical note: our tropical curves $\Gamma$ have no ``hidden genus'' at vertices and no infinite legs,
i.e. we restrict our attention to $\Xan$ with totally degenerate reduction and no punctures.

\subsection{Outline}
In Section~\ref{sec:trop-curves} we review background material on metric graphs and their divisor theory.
In Section~\ref{sec:canonical} we review the interpretation of a metric graph as an electrical resistor network,
and define Zhang's canonical measure.
In Section~\ref{sec:weierstrass} we define the Weierstrass locus 
and stable Weierstrass locus for a divisor on a metric graph, 
and give examples.
In Section~\ref{sec:finite} we 
prove that $W(D)$ is generically finite 
and compute its cardinality (Theorem~\ref{thm:w-finite}).
In Section~\ref{sec:equidistribution}, we prove results on the distribution of Weierstrass points 
on a metric graph (Theorems~\ref{thm:equidistribution} and \ref{thm:quant-equidistribution}).

\subsection{Notation}
Here we collect some notation which will be used throughout the paper.

\begin{tabular}{ll}
$\Gamma$ & a compact, connected metric graph \\

$\PL_\RR(\Gamma)$ & continuous, piecewise linear functions on  $\Gamma$ \\

$\PL_\ZZ(\Gamma)$ & continuous, piecewise $\ZZ$-linear functions on $\Gamma$ \\

% \end{tabular}
% \begin{tabular}{ll}
$\rS(\Gamma)$ & ``well-behaved'' piecewise smooth  functions on $\Gamma$ \\
 % \harry{(or polynomial?)}

% \end{tabular} 
% \begin{tabular}{ll}
$\Divisor(f)$ & the principal divisor associated to a piecewise ($\ZZ$-)linear function $f$ \\

$D$ & a divisor on a metric graph or algebraic curve \\

$D_{\en}$  & a divisor of degree $\en$ \\

% \end{tabular}
% \begin{tabular}{ll}
$K = K_\Gamma$ & the canonical divisor on $\Gamma$ \\

$r(D)$ & the Baker--Norine rank of $D$ \\

$\Div(\Gamma)$ & divisors on $\Gamma$ (with $\ZZ$-coefficients) \\

$\Div_\RR(\Gamma)$ & divisors on $\Gamma$  with $\RR$-coefficients, 
i.e. $ \Div(\Gamma)\otimes_\ZZ \RR$\\

$\Div^d(\Gamma)$ & divisors of degree $d$ on $\Gamma$ \\

$\Pic^d(\Gamma)$ & divisor classes of degree $d$ on $\Gamma$ \\

$\Sym^d(\Gamma)$  & effective divisors of degree $d$ on $\Gamma$ \\

$\Eff^d(\Gamma)$  & effective divisor classes of degree $d$ on $\Gamma$ \\

%$W_r^d(\Gamma)$ & divisor classes of degree $d$ and rank $\geq r$ on $\Gamma$ \\

$[D]$ & a divisor class; the set of divisors linearly equivalent to $D$  \\
% (i.e. the linear equivalence class of the divisor $D$)

$|D|$ & the space of effective divisors linearly equivalent to $D$ \\
% (as a subspace of $\Sym^n(\Gamma)$)\\
% i.e. the complete linear system of $D$

$\reddiv{x}{D}$  & the $x$-reduced divisor equivalent to $D$, where $x\in \Gamma$ \\
% ( or $[D]_x$ or $\red(D,x)$?)

$\br[D]$ & the break divisor equivalent to $D$, where $D$ has degree $g$ \\

$\Br^g(\Gamma)$ & the space of break divisors on $\Gamma$ \\ 
% (as subspace of $\Sym^g(\Gamma)$) \\

$\mu = \mu_\Gamma$ &  the Zhang canonical measure on  $\Gamma$ \\

$G$ & a finite, connected graph with vertex set $V(G)$ and edge set $E(G)$ \\

$(G,\ell)$ & a combinatorial model for a metric graph, where $\ell : E(G) \to \RR_{\geq 0}$ \\

$\mathcal{T}(G)$ & the set of spanning trees of a graph $G$
\end{tabular}

\section{Abstract tropical curves}
\label{sec:trop-curves}
In this section we define metric graphs and linear equivalence of divisors on metric graphs. 
We use the terms ``metric graph'' and ``(abstract) tropical curve'' interchangeably.
We recall the Baker--Norine rank of a divisor, 
and state the Riemann--Roch theorem which is satisfied by this rank function.

\subsection{Metric graphs and divisors}
\label{subsec:metric-graphs}
A {\em metric graph} is a compact, connected metric space 
which comes from assigning positive real edge lengths to a finite connected combinatorial graph.
Namely, we construct a metric graph $\Gamma$  by taking 
a finite set of edges $E = \{e_i\}$,
each isometric to a real interval $e_i = [0,L_i]$ of length $L_i>0$,
gluing their endpoints to a finite set of vertices $V$,
and imposing the path metric.
The underlying combinatorial graph $G = (E,V)$ is called a {\em combinatorial model} for $\Gamma$.
We allow loops and parallel edges in a combinatorial graph $G$.
We say $e$ is a {\em segment} of $\Gamma$ if it is an edge in some combinatorial model. 

The {\em valence} $\val(x)$ of a point $x$ on a metric graph $\Gamma$ is defined to be the number 
of connected components of a sufficiently small punctured neighborhood of $x$.
Points in the interior of a segment of $\Gamma$ always have valence 2.
All points $x$ with $\val(x) \neq 2$ are contained in the vertex set of any combinatorial model.

The {\em genus} of a metric graph $\Gamma$ is its first Betti number as a topological space,
\[ g(\Gamma) = b_1(\Gamma) = \dim_\RR H_1(\Gamma,\RR).\]
%A metric graph of genus $g$ is homotopy equivalent to a wedge of $g$ loops, 
%and has Euler characteristic $\chi(\Gamm a) = 1 - g$.
If $G$ is a combinatorial model for $\Gamma$, the genus is equal to
$ g(\Gamma) = {\# E(G) - \# V(G) + 1}$.

A {\em divisor} on a metric graph $\Gamma$ is a finite formal sum of  
points of $\Gamma$ with integer coefficients.
%We express a divisor $D$ by
%\[ D = \sum_{x\in \Gamma} a_x x \qquad \text{where } a_x \in \ZZ,  
%\text{ and }a_x = 0 \text{ for almost all } x. \]
The {\em degree} of a divisor is the sum of its coefficients; 
i.e.
for the divisor $D = \sum_{x\in \Gamma} a_x x $,
we have
${ \deg( D ) = \sum_{x\in\Gamma} a_x }$.
We let $\Div(\Gamma)$ denote the set of all divisors on $\Gamma$,
and let $\Div^d(\Gamma)$ denote the divisors of degree $d$.
We say a divisor is {\em effective} if all of its coefficients are non-negative;
we write $D\geq 0$ to indicate that $D$ is effective.
More generally, we write $D \geq E$ to indicate that $D-E$ is an effective divisor.
We let $\Sym^d(\Gamma)$ denote the set of effective divisors of degree $d$ on $\Gamma$.
$\Sym^d (\Gamma)$ inherits from $\Gamma$ the structure of a polyhedral cell complex of dimension~$d$.

We let $\Div_\RR(\Gamma)$ denote the set of divisors on $\Gamma$ with coefficients in $\RR$.
In other words, $\Div_\RR(\Gamma) = \Div(\Gamma)\otimes_\ZZ \RR$.

\subsection{Principal divisors and linear equivalence}
We define linear equivalence for divisors on metric graphs, 
following Gathmann--Kerber \cite{GK} and Mikhalkin--Zharkov \cite{MZ}.
This notion is analogous to linear equivalence of divisors on an algebraic curve,
where rational functions are replaced with piecewise $\ZZ$-linear functions. 

A {\em piecewise linear function} on $\Gamma$ is a continuous function $f: \Gamma \to \RR$ 
such that there is 
% some decomposition of $\Gamma$ into a union of finitely many segments 
some combinatorial model for $\Gamma$ such that 
$f$ restricted to each edge is a linear function,
i.e. a function of the form
\[ f(x) = ax + b, \qquad a,b\in \RR ,\]
where $x$ is a length-preserving parameter on the edge.
We let $\PL_\RR(\Gamma)$ denote the set of all piecewise linear functions on $\Gamma$.
%Note that $\PL_\RR(\Gamma)$ naturally has the structure of an $\RR$-vector space.
%(It is moreover a tropical $\RR$-module, under operations addition and minimum.)

A {\em piecewise $\ZZ$-linear function} on $\Gamma$ is a piecewise linear function 
such that all its slopes are integers, 
i.e. $f$ restricted to each edge has the form
\[ f(x) = ax + b, \qquad a\in\ZZ,\, b\in \RR \]
(for some combinatorial model).
We let $\PL_\ZZ(\Gamma)$ denote the set of all piecewise $\ZZ$-linear functions on $\Gamma$.
The functions $\PL_\ZZ(\Gamma)$ are closed under the operations of addition,
multiplication by $\ZZ$,
and taking pairwise $\max$ and $\min$.
%(i.e. $\PL_\ZZ(\Gamma)$ naturally has the structure of a tropical $\ZZ$-module).

We let $UT_x\Gamma$  denote the 
unit tangent fan of $\Gamma$ at $x$,
which is the set of ``directions going away from $x$'' on $\Gamma$.
For $v \in UT_x\Gamma$,  
the symbol $\epsilon v$ for sufficiently small $\epsilon \geq 0$
means the point in $\Gamma$
 that is distance $\epsilon$ away from $x$ in the direction $v$.
%which is a finite set of size $\val(x)$.
For  $v\in UT_x\Gamma$
and a function $f: \Gamma \to \RR$   we let 
\[D_vf(x) = \lim_{\epsilon\to 0^+} \frac{f(x+\epsilon v)-f(x)}{\epsilon}\]
denote the slope of $f$ while travelling away from $x$ in the direction $v$
(if it exists).

Given %a piecewise $\ZZ$-linear function 
$f \in \PL_\ZZ(\Gamma)$, 
we define the {\em principal divisor} $\Divisor(f) \in \Div^0(\Gamma)$ by
\[ \Divisor(f) = \sum_{x\in \Gamma} a_x x \qquad\text{where}\qquad a_x = \sum_{v\in UT_x \Gamma} D_v f(x). \]
In words, the coefficient in $\Divisor(f)$ of a point $x$ is equal to the sum of the outgoing slopes of $f$ at $x$.
On a given segment, this divisor is supported on the finite set of points at which $f$ is not linear, 
sometimes called the ``break locus'' of $f$.
If $\Divisor(f) = D - E$ where $D, E$ are effective divisors with disjoint support, 
then we call $D = \zeros(f)$ the {\em divisor of zeros of $f$} 
and $E = \poles(f)$ the {\em divisor of poles} of $f$.

We say two divisors $D, E$ are {\em linearly equivalent}, denoted $D\sim E$, 
if there exists a piecewise $\ZZ$-linear function $f$ such that
\[ \Divisor(f) = D - E.\]
Note that linearly equivalent divisors must have the same degree.
We let $[D]$ denote the linear equivalence class of divisor $D$, i.e.
\[ [D] = \{ E \in \Div(\Gamma) : E \sim D\} 
= \{D + \Divisor(f) : f\in \PL_\ZZ(\Gamma)\} .\]
We say a divisor class $[D]$ is {\em effective}, or write $[D]\geq 0$, 
if there is an effective representative ${ E\sim D },\, E\geq 0$ in the equivalence class.

We let $|D|$ denote the {\em (complete) linear system} of $D$, 
which is the set of effective divisors linearly equivalent to $D$.
We have
\begin{align*}
 |D| &= \{ E \in \Div(\Gamma) : E \sim D,\, E\geq 0\} \\
 &= \{D + \Divisor(f) : f\in \PL_\ZZ(\Gamma),\, \Divisor(f) \geq -D\} .
\end{align*}
%The functions $\{f\in \PL_\ZZ(\Gamma) :  \Divisor(f) \geq -D \}$ are closed under 
%taking pairwise $\max$ and adding a constant function.
Unlike $[D]$, the linear system $|D|$ is naturally a compact polyhedral complex, with topology induced by the inclusion
$|D| \subset \Sym^d(\Gamma)$.

\begin{rmk}[Linear equivalence as chip firing]
We sometimes speak of a degree $\en$ effective divisor on $\Gamma$ as a collection of $\en$ 
``chips'' placed on $\Gamma$,
in the sense of playing a board game with board $\Gamma$.
Changing the divisor $D$ to a linearly equivalent divisor $D'$
can be achieved through a sequence of ``chip firing moves''
where we choose a simple cut\footnote{
A {\em simple cut} is a collection of segments of $\Gamma$ such that
removing the interiors of these segments
disconnects $\Gamma$ into exactly two components.} 
of $\Gamma$ consisting of
$m$ segments of length $\epsilon$,
and on each edge move a chip from one end to the other.
\begin{figure}[h]
\begin{minipage}{0.45\textwidth}
	\centering
\begin{tikzpicture}[scale=0.8]
\coordinate (A) at (0,0);
\coordinate (B) at (1.3,0);
\coordinate (C) at (0,1.5);
\coordinate (D) at (1.3,1.5);

\draw (A) -- (B);
\draw (C) -- (D);
\draw (A) -- (C);
\draw (B) -- (D);
\draw[rounded corners=10] (A) -- +(-1,0) -- +(-1,1.5) -- (C);
\draw[rounded corners=10] (B) -- +(1,0) -- +(1,1.5) -- (D);

\foreach \pos in {(0.3,1.5),(1.3,0.7),(2.3,0.7)} {
	\fill \pos circle (2.5pt);
}
\end{tikzpicture}
\end{minipage}
\begin{minipage}{0.45\textwidth}
	\centering
\begin{tikzpicture}[scale=0.8]
\coordinate (A) at (0,0);
\coordinate (B) at (1.3,0);
\coordinate (C) at (0,1.5);
\coordinate (D) at (1.3,1.5);

\draw (A) -- (B);
\draw (C) -- (D);
\draw (A) -- (C);
\draw (B) -- (D);
\draw[rounded corners=10] (A) -- +(-1,0) -- +(-1,1.5) -- (C);
\draw[rounded corners=10] (B) -- +(1,0) -- +(1,1.5) -- (D);

\foreach \pos in {(0.3,1.5),(1.3,0.7),(2.3,0.7)} {
	\fill[color=white] \pos circle (2.5pt);
	\draw[color=lightgray] \pos circle (2.5pt);
}
\foreach \pos in {(0.7,1.5),(1.3,1.1),(2.3,1.1)} {
	\fill[color=black] \pos circle (2.5pt);
}
\end{tikzpicture}
\end{minipage}
\caption{Chip firing across a simple cut.}
\end{figure}
The piecewise linear function associated to such a chip firing move 
has slope $0$ outside the cut segments, and slope $1$ on the cut segments.
For more discussion see \cite[Remark 2.2]{ABKS}, \cite[Section 1.5]{BN} and the references therein.
\end{rmk}

%\begin{rmk}[Linear interpolation along $f$]
Given a function $f \in \PL_\ZZ(\Gamma)$,
we may associate to $f$  
a 1-parameter family of effective divisors 
which ``linearly interpolate'' between the zeros $\zeros(f)$ and poles $\poles(f)$.
We can think of this construction as specifying a unique ``geodesic path'' 
between any two points in the complete linear system $|D|$.
This notion was previously studied by Luo \cite{Luo} under the name {\em t-path}.
Namely, we make the following definition.

\begin{dfn}
Suppose $f\in \PL_\ZZ(\Gamma)$ and $\lambda \in \RR$.
We let $\lambda \in \PL_\ZZ(\Gamma)$ also denote the constant function on $\Gamma$
by abuse of notation,
and we define the {\em tropical preimage} $\divfiber{f}(\lambda)$ 
by
\[ \divfiber{f}(\lambda) = \poles(f) + \Divisor(\max\{f ,\lambda\}). \] 
\end{dfn}
See Figure \ref{fig:interpolation} for an illustration.
The tropical preimage has the following equivalent description: consider the graph of $f$, which is a subset of $\Gamma \times \RR$, and attach to it half-infinite vertical rays pointing downwards (resp. upwards) at all zeros (resp. poles) of $f$.
Then $\divfiber{f}(\lambda)$ is obtained by intersecting this modified graph with the level set $\Gamma \times \{\lambda\}$,
and then projecting to $\Gamma$, when assigning appropriate multiplicities to the vertical rays and the resulting intersection.

Note that %according to this definition,
the tropical preimage is always an effective divisor, and that
${ \divfiber{f}(\lambda) = \poles(f) }$ for $\lambda$ sufficiently large and 
${ \divfiber{f}(\lambda) = \zeros(f) }$ for $\lambda$ sufficiently small.
%
%(We also have the more symmetric, but more complicated, expression
%\[ f^{-1}(\lambda) := D(f,\lambda) =  \zeros(f \vee \lambda) + \poles(f \wedge \lambda) 
%- \left( \cap \right) .\]
It is clear from definition that for any $\lambda$, 
the tropical preimage
$\divfiber{f}(\lambda)$ is %an (effective) divisor which is 
linearly equivalent to $\zeros(f)$ and to $\poles(f)$.
\begin{figure}[h]
\centering
\begin{tikzpicture}[scale=0.5]
	\coordinate (A) at (0,2.8);
	\coordinate (B) at ($(A) + (1.5,0)$);
	\coordinate (C) at ($(B) + (1.5,1.5)$);
	\coordinate (D) at ($(C) + (3,-3)$);
	\coordinate (E) at ($(D) + (2,0)$);
	\coordinate (F) at ($(E) + (0.5,-0.5)$);
	\coordinate (G) at ($(F) + (0.5,0.5)$);
	\coordinate (H) at ($(G) + (2,0)$);
	
	\coordinate (Ad) at (0,0);
	\coordinate (Bd) at (1.5,0);
	\coordinate (Dd) at (6,0);
	\coordinate (Ed) at (8,0);
	\coordinate (Fd) at (8.5,0);
	\coordinate (Gd) at (9,0);
	\coordinate (Hd) at (11,0);
	\coordinate (Au) at (0,5);
	\coordinate (Hu) at (11,5);
	
	\coordinate (Af) at ($(A) + (0,-1.1)$);
	\coordinate (Hf) at ($(Af) + (11,0)$);
	\coordinate (INTf) at ($(Hf) + (-5.4,0)$);
	\coordinate (INTd) at (5.6,0);
	
	\draw (A) -- (B) -- (C) -- (D) -- (E);
	\draw (E) -- (F) -- (G) -- (H);
	\draw[dotted] (B) -- (Bd);
	\draw[dotted] (C) -- +(0,0.5);
	\draw[dotted] (D) -- (Dd);
	\draw[dotted] (E) -- +(0,3.5);
	\draw[dotted] (F) -- (Fd);
	\draw[dotted] (G) -- +(0,3.5);
	
	\draw[line width=12pt, color=white] (Ad) -- (Hd);
	\draw (Ad) -- (Hd);
	% \draw[dotted] (Au) -- (Hu);
	\draw[dashed] (Af) -- (Hf);
	% graph of max(f, lambda)
	\draw[color=blue, line width=2pt, opacity=0.4] (A) -- (B) -- (C) 
		-- node[above right, opacity=0.9] {$\max\{f, \lambda\}$} (INTf) 
		-- (Hf);
	
	\node[left] at (A) {$f$};
	\node[left] at (Af) {$\lambda$};
	\node[left] at (Ad) {$\Gamma$};
	
	\foreach \coord in {Bd,INTd,Ed,Gd} {
		\fill (\coord) circle (4pt);
	}
	\foreach \coord in {Bd,INTd,Ed,Gd} {
		\fill[opacity=0.2] (\coord) + (0,1.7) circle (5pt);
	}
\end{tikzpicture}
\caption{Linear interpolation showing the divisor $\divfiber{f}(\lambda)$.}
\label{fig:interpolation}
\end{figure}
%\end{rmk}

\begin{dfn}
\label{dfn:geodesic}
A path $\gamma: [0,1] \to |D|$ is {\em geodesic}
if $\gamma$ is injective and its image is equal to a set of tropical preimages up to translation;
in other words, the image of $\gamma$ is
\[
\{D_0 +  \divfiber{f}(\lambda) : \lambda \in \RR \}
\]
where 
$D_0$ is a fixed divisor and 
$f\in \PL_\ZZ(\Gamma)$ is a function satisfying
$
\Divisor(f) = \gamma(0) - \gamma(1) .
$
(The function $f$ is determined up to an additive constant by $\gamma(0), \gamma(1)$.)
\end{dfn}
%%% ALTERNATE DEFINITION
%\begin{dfn}
%A path $\gamma : [0,1] \to |D|$ is {\em geodesic} if the following condition holds.
%Let $f_\gamma$ denote a piecewise linear function such that
%$\Divisor(f_\gamma) = \gamma(1) - \gamma(0)$.
%(Such $f_\gamma$ exists by definition of the linear equivalence class $|D|$.)
%Let 
%\[
%M = \max\{f_\gamma(x) : x \in \Gamma \}
%\qquad\text{and}\qquad
%m = \min\{f_\gamma(x) : x \in \Gamma \}
%\]
%Then we require $\gamma$ to satisfy the geodesic condition
%\begin{equation}
%\gamma(t) = \gamma(0) + \Divisor(\max\{f_\gamma, (1-t)m + tM \})
%\qquad\text{for all }t \in [0,1].
%\end{equation}
%\end{dfn}

\subsection{Reduced divisors}
\label{sec:reduced-prop}
%A divisor class $[D]$ is typically very large, 
%so it is convenient to have a method of choosing a (somewhat-)canonical representative divisor inside $[D]$.
%When $D$ has arbitrary degree,
%we can do so after fixing a basepoint $q$ on our metric graph $\Gamma$,
%using the $q$-reduced divisor construction.

Given a  point $q\in \Gamma$,
the {\em $q$-reduced divisor} $\reddiv{q}{D}$ is the unique divisor in $[D]$ which is effective away from $q$, 
and which minimizes a certain {energy function} among such representatives.
Intuitively, $\reddiv{q}{D}$ is the divisor in $[D]$ whose chips are
 ``as close as possible'' to the basepoint $q$.
We defer giving the full definition until Section \ref{sec:red-div},
following  \cite[Appendix A]{BS}. 
For now, we state these important properties of the reduced divisor:
\begin{enumerate}
%\item $\reddiv{q}{D}$ is a unique divisor inside $[D]$
%\item if $D$ has degree $d$, then $\reddiv{q}{D}$ has degree at least $d-g$ at $q$ 
%and at most $g$ away from $q$
%\item $\reddiv{q}{D}$ is effective away from $q$
\item[(RD1)] $[D] \geq 0$ if and only if $\reddiv{q}{D} \geq 0$
\item[(RD2)]
for any integer $m$, $\reddiv{q}{mq + D} = mq + \reddiv{q}{D}$
\item[(RD3)] the degree of $\reddiv{q}{D}$ away from $q$ is at most $g$, the genus of $\Gamma$
(follows from Riemann's inequality, Corollary~\ref{cor:riemann})
\item[(RD4)] for a fixed effective divisor $D$, the map $\Gamma \to |D|$ sending $q \mapsto \reddiv{q}{D}$ is continuous 
(due to  Amini \cite[Theorem 3]{A-reddiv}).

\end{enumerate}

\subsection{Picard group and Abel--Jacobi}
\label{sec:picard}

We let $\Pic(\Gamma)$ denote the {\em Picard group} of $\Gamma$,
which is the abelian group of all linear equivalence classes of divisors on $\Gamma$.
The addition operation on $\Pic(\Gamma)$ is induced from addition of divisors in $\Div(\Gamma)$.
In other words, $\Pic(\Gamma)$ is the cokernel of the map $\Divisor$
sending a piecewise $\ZZ$-linear function to its associated principal divisor:
\[ \PL_\ZZ(\Gamma) \xrightarrow{\Divisor} \Div(\Gamma) \to \Pic(\Gamma) \to 0 .\]
The kernel of $\Divisor$ is the set of constant functions on $\Gamma$.

Since the degree of a divisor class is well-defined, we have a disjoint union decomposition
\[ \Pic(\Gamma) = \bigsqcup_{d\in \ZZ}\Pic^d(\Gamma).\]
The degree-0 component $\Pic^0(\Gamma)$ is a compact abelian group, 
and each $\Pic^d(\Gamma)$ is a torsor for $\Pic^0(\Gamma)$.

\begin{thm}[Abel--Jacobi for metric graphs]
\label{thm:abel-jacobi}
Let $\Gamma$ be a metric graph of genus $g$.
Then for any degree $d$, there is a homeomorphism of topological spaces
\[ \Pic^d( \Gamma) \cong \RR^g / \ZZ^g .\]
% = (S^1)^{\times g} = \overbrace{S^1 \times \cdots \times S^1}^g.\]
When $d = 0$, this is an isomorphism of compact abelian topological groups.
\end{thm}
\begin{proof}
See Mikhalkin--Zharkov \cite{MZ}.
The proof follows the same idea as the classical Abel-Jacobi theorem,
to show that $\Pic^0(\Gamma) = H^1(\Gamma,\RR) / H_1(\Gamma,\ZZ)^{\vee} \cong \RR^g / \ZZ^g$.
\end{proof}
%We have the exact sequence
%\[ 0 \to \RR \xrightarrow{\text{const}} \PL_\ZZ(\Gamma) \xrightarrow{\Divisor} \Div(\Gamma) \to \ZZ \times (S^1)^g \to 0 .\]

We let $\Eff^d(\Gamma)$ denote the set of divisor classes on $\Gamma$ of degree $d$ which have an effective representative. 
In other words, $\Eff^d(\Gamma)$ is the image of $\Sym^d(\Gamma)$ 
under the cokernel map $\Div(\Gamma) \to \Pic(\Gamma)$ (or its degree-$d$ restriction):
\[ \begin{tikzcd}
\Sym^d(\Gamma) \arrow[r] \arrow[d] & \Div^d(\Gamma) \arrow[d,"\coker\Divisor"] \\
\Eff^d(\Gamma) \arrow[r] & \Pic^d(\Gamma).
\end{tikzcd} \]
The space $\Eff^d(\Gamma)$ is naturally a polyhedral complex of pure dimension $d$
when ${ 0 \leq d \leq g }$
(``pure'' due to Gross--Shokrieh--T\'{o}thm\'{e}r\'{e}sz~\cite{GST}).
As a particularly important case, 
the {\em theta divisor} $\Theta = \Theta(\Gamma)$ is the space $\Theta = \Eff^{g - 1}(\Gamma)$, 
which lives inside $\Pic^{g - 1}(\Gamma) $ as a codimension 1 polyhedral complex.

\begin{rmk}
\label{rmk:real-div}
The map $\Divisor:\PL_\ZZ(\Gamma) \to \Div(\Gamma)$ %sending $f \mapsto \Divisor(f)$ 
is also known as the {\em metric graph Laplacian} on $\Gamma$.
This comes from identifying $\Div(\Gamma)$ with the set of finite integer combinations of Dirac delta measures on $\Gamma$,	via
\[ D = \sum_{i=1}^\en a_i x_i 
\quad\longleftrightarrow\quad 
\delta = \sum_{i=1}^\en a_i \delta_{x_i} \]
so that $\Divisor(f)$ coincides with the distributional second derivative $-\frac{d^2}{dx^2}f(x)$,
at least for $x$ in the interior of an edge.
The definition of 
metric graph Laplacian naturally extends to piecewise linear functions on $\Gamma$ with arbitrary {\em real} slopes, 
if we also allow real-valued coefficients in the divisor $\Divisor(f)$.
This yields a map
\[ \PL_\RR(\Gamma) \xrightarrow{\Divisor} \Div_\RR(\Gamma) .\]
The cokernel of this map is less interesting (e.g. it does not tell us the genus of $\Gamma$); 
it is simply the degree function $\Div_\RR(\Gamma) \xrightarrow{\text{deg}} \RR$.
We will see why this is the cokernel in Section \ref{sec:voltage} on voltage functions.
% The kernel of this map is the constant functions on $\Gamma$.
This fits in the short exact sequence
\[ 0 \to \RR \xrightarrow{\text{const}} \PL_\RR(\Gamma)\xrightarrow{\Divisor} \Div_\RR(\Gamma) \xrightarrow{\text{deg}} \RR \to 0.\]
(Compare to the integral case
\[ 0 \to \RR \xrightarrow{\text{const}} \PL_\ZZ(\Gamma)\xrightarrow{\Divisor} \Div(\Gamma) 
   \xrightarrow{} \Pic(\Gamma)  \to 0 \]
where $\Pic(\Gamma) \cong \ZZ \times (S^1)^g$.)
\end{rmk}

\subsection{Rank and Riemann--Roch}
\label{sec:rank}
We recall the definition of the rank of a divisor on a metric graph,
which is due to Baker and Norine \cite{BN} (originally for divisors on a combinatorial graph)
and extended to metric graphs by  Gathmann--Kerber \cite{GK} and Mikhalkin--Zharkov \cite{MZ}.
The rank function is a natural way to extend  the important distinction 
between effective and non-effective divisor classes on a metric graph.
Divisor classes with larger rank are in a sense ``further away'' from the set of non-effective divisor classes,
where distance between divisors is given by adding or subtracting single points.

%A strong argument for why this notion of rank is ``morally correct'' is that Riemann-Roch is satisfied by divisors on a metric graph. 

The {\em rank} $r(D)$ of a divisor $D$ on $\Gamma$ is defined as
\[ r(D) = \max\{r \geq 0 : [D-E]\geq 0 \text{ for all }E\in \Sym^r(\Gamma) \}\]
if $[D]$ is effective, and $r(D) = -1$ otherwise.
Equivalently,
\[ r(D) = \begin{cases}
-1 & \text{if $[D]$ is not effective}, \\
1 + \displaystyle\min_{x\in \Gamma} \{ r(D - x) \} & \text{if $[D]$ is effective}.
\end{cases}\]
This second definition inductively gives the rank of a divisor in terms of divisors of smaller degree; 
the base case is the set of non-effective divisor classes.\footnote{
By Riemann's inequality, Corollary~\ref{cor:riemann},
 a non-effective divisor class has degree at most $g-1$.}
%for example, all divisors classes of degree $d<0$ have rank $-1$, and a divisor class of degree $0$ has rank $0$ if and only if it is effective.
Note that the rank of a divisor $D$ depends only on its linear equivalence class. 

The {\em canonical divisor} $K = K_\Gamma$ on a metric graph $\Gamma$ is defined as
\begin{equation} 
\label{eq:canonical-divisor}
K = \sum_{x\in \Gamma} (\val(x) - 2) \cdot x.
\end{equation}
The degree of the canonical divisor is $\deg K = 2g - 2$, which agrees with the degree of a canonical divisor on an algebraic curve.

\begin{thm}[Riemann--Roch for metric graphs]
\label{thm:riemann-roch}
Let $\Gamma$ be a metric graph of genus $g$, and let $K$ be the canonical divisor on $\Gamma$.
For any divisor $D$ on $\Gamma$,
\[ r(D) - r(K-D) = \deg(D) + 1 - g.\]
\end{thm}
\begin{proof}
See Gathmann--Kerber \cite[Proposition 3.1]{GK} and Mikhalkin--Zharkov \cite[Theorem 7.3]{MZ}, 
which both adapt the arguments of Baker--Norine \cite{BN} for the case of combinatorial graphs.
\end{proof}
\begin{cor}[Riemann's inequality for metric graphs]
\label{cor:riemann}
For a divisor $D$ on a metric graph of genus $g$,
\[ r(D) \geq \deg( D) - g.\]
\end{cor}
\begin{proof}
This follows from Riemann--Roch since $r(K-D) \geq -1$. % by definition of rank.
\end{proof}
By Riemann's inequality, any divisor $D$ satisfies $r(D) \geq \max\{ {\deg(D) - g}, -1\}$.
We say $D$ is {\em nonspecial} if $r(D)$ is equal to this lower bound.

\subsection{Break divisors and ABKS decomposition}
\label{sec:break-div}
When a divisor $D$ has degree $g$, 
there is a canonical representative of $[D]$ without any choice of basepoint,
using the concept of break divisor.
This notion was introduced by Mikhalkin--Zharkov~\cite{MZ} 
and studied extensively by An--Baker--Kuperberg--Shokrieh~\cite{ABKS}.
We review some of their results in this section.

A {\em break divisor} is an effective divisor of degree $g$
(the genus) which can be constructed in the following manner:
choose a combinatorial model $G = (V,E)$ for $\Gamma$
and choose a spanning tree $T$ of $G$, 
then place one chip on each edge in the complement $E \backslash E(T)$.
(Note that $E\backslash E(T)$ contains exactly $g$ edges.)
Placing a chip on the endpoint of an edge is allowed.

The set of break divisors does not depend on the choice of combinatorial model.
We use $\Br^g(\Gamma)$ 
to denote the set of all break divisors on $\Gamma$.
We may view $\Br^g(\Gamma)$ as a topological space,
using the topology induced from the inclusion in
$\Sym^g(\Gamma)$.

\begin{eg}
In Figure~\ref{fig:break}
we show three examples of break divisors, on the left, 
and three examples of non-break divisors, on the right,
on a genus $3$ metric graph.
\begin{figure}[h]
\begin{minipage}{0.40\textwidth}
	\centering
	\begin{tikzpicture}[scale=0.6]
		\coordinate (A) at (0,0);
		\coordinate (B) at (1.3,0);
		\coordinate (C) at (0,1.5);
		\coordinate (D) at (1.3,1.5);

		\draw (A) -- (B);
		\draw (C) -- (D);
		\draw (A) -- (C);
		\draw (B) -- (D);
		\draw[rounded corners=5] (A) -- +(-1,0) -- +(-1,1.5) -- (C);
		\draw[rounded corners=5] (B) -- +(1,0) -- +(1,1.5) -- (D);

		\foreach \pos in {(0,0.7),(-1,0.7),(1.3,0.9)} {
			\fill \pos circle (3pt);
		}
	\end{tikzpicture}
	\qquad
	\begin{tikzpicture}[scale=0.6]
		\coordinate (A) at (0,0);
		\coordinate (B) at (1.3,0);
		\coordinate (C) at (0,1.5);
		\coordinate (D) at (1.3,1.5);

		\draw (A) -- (B);
		\draw (C) -- (D);
		\draw (A) -- (C);
		\draw (B) -- (D);
		\draw[rounded corners=5] (A) -- +(-1,0) -- +(-1,1.5) -- (C);
		\draw[rounded corners=5] (B) -- +(1,0) -- +(1,1.5) -- (D);

		\foreach \pos in {(0.5,1.5),(1.3,0.7),(-1,0.7)} {
			\fill \pos circle (3pt);
		}
	\end{tikzpicture}\\
	\begin{tikzpicture}[scale=0.6]
		\coordinate (A) at (0,0);
		\coordinate (B) at (1.3,0);
		\coordinate (C) at (0,1.5);
		\coordinate (D) at (1.3,1.5);

		\draw (A) -- (B);
		\draw (C) -- (D);
		\draw (A) -- (C);
		\draw (B) -- (D);
		\draw[rounded corners=5] (A) -- +(-1,0) -- +(-1,1.5) -- (C);
		\draw[rounded corners=5] (B) -- +(1,0) -- +(1,1.5) -- (D);

		\foreach \pos in {(0,1.5),(1.3,1.5)} {
			\fill \pos circle (3pt);
		}
		\node[above] at (C) {$2$};
	\end{tikzpicture}
\end{minipage}
\begin{minipage}{0.1\textwidth}
\quad
\end{minipage}
\begin{minipage}{0.40\textwidth}
	\centering
	\begin{tikzpicture}[scale=0.6]
		\coordinate (A) at (0,0);
		\coordinate (B) at (1.3,0);
		\coordinate (C) at (0,1.5);
		\coordinate (D) at (1.3,1.5);

		\draw (A) -- (B);
		\draw (C) -- (D);
		\draw (A) -- (C);
		\draw (B) -- (D);
		\draw[rounded corners=5] (A) -- +(-1,0) -- +(-1,1.5) -- (C);
		\draw[rounded corners=5] (B) -- +(1,0) -- +(1,1.5) -- (D);

		\foreach \pos in {(0.5,0.0),(0.5,1.5),(1.3,0.7)} {
			\fill \pos circle (3pt);
		}
	\end{tikzpicture}
	\qquad
	\begin{tikzpicture}[scale=0.6]
		\coordinate (A) at (0,0);
		\coordinate (B) at (1.3,0);
		\coordinate (C) at (0,1.5);
		\coordinate (D) at (1.3,1.5);

		\draw (A) -- (B);
		\draw (C) -- (D);
		\draw (A) -- (C);
		\draw (B) -- (D);
		\draw[rounded corners=5] (A) -- +(-1,0) -- +(-1,1.5) -- (C);
		\draw[rounded corners=5] (B) -- +(1,0) -- +(1,1.5) -- (D);

		\foreach \pos in {(0.5,1.5),(1.3,0.7),(2.3,0.7)} {
			\fill \pos circle (3pt);
		}
	\end{tikzpicture} \\
	\begin{tikzpicture}[scale=0.6]
		\coordinate (A) at (0,0);
		\coordinate (B) at (1.3,0);
		\coordinate (C) at (0,1.5);
		\coordinate (D) at (1.3,1.5);

		\draw (A) -- (B);
		\draw (C) -- (D);
		\draw (A) -- (C);
		\draw (B) -- (D);
		\draw[rounded corners=5] (A) -- +(-1,0) -- +(-1,1.5) -- (C);
		\draw[rounded corners=5] (B) -- +(1,0) -- +(1,1.5) -- (D);

		\foreach \pos in {(1.3,1.5),(1.3,0.0)} {
			\fill \pos circle (3pt);
		}
		\node[above] at (D) {$2$};
	\end{tikzpicture}
\end{minipage}
\caption{Break divisors (left) and non-break divisors (right).}
\label{fig:break}
\end{figure}
\end{eg}

For a divisor class $[D]$ whose degree is $g$, the genus of the underlying curve,
there is a unique representative of $[D]$ which is a break divisor.
The following result combines results of
Mikhalkin--Zharkov~\cite[Corollary 6.6]{MZ}
and An--Baker--Kuperberg--Shokrieh~\cite[Theorem 1.1]{ABKS}.
\begin{thm}[Break divisors representatives]
\label{thm:break-div}
Let $\Gamma$ be a metric graph of genus~$g$. 
\begin{enumerate}[(a)]
\item Every divisor class $[D]\in \Pic^g(\Gamma)$ contains a unique break divisor,
which we denote $\br[D]$.

\item The map $\br : \Pic^g(\Gamma) \to \Sym^g(\Gamma)$
sending a divisor class to its break divisor representative
is continuous and injective.
Its image is the space of all break divisors $\Br^g(\Gamma)$. 

\item The map $\br : \Pic^g(\Gamma) \to \Sym^g(\Gamma)$
is the unique continuous section of the map 
$[-]: \Sym^g(\Gamma) \to \Pic^g(\Gamma)$ taking
an effective divisor to its linear equivalence class.
Namely, $\br(-)$ is the unique continuous map such that the composition
\begin{equation*}
\Pic^g(\Gamma) \xrightarrow{\br} \Sym^g(\Gamma) \xrightarrow{[-]} \Pic^g(\Gamma)
\end{equation*}
is the identity homeomorphism.
\end{enumerate}
\end{thm}

If we choose a combinatorial model $(G,\ell)$ for the metric graph $\Gamma$,
An--Baker--Kuperberg--Shokrieh \cite{ABKS} showed that
the theory of break divisors implies a nice combinatorial decomposition
of $\Pic^g(\Gamma)$.
($\Pic^g(\Gamma)$ is defined in Section~\ref{sec:picard}.)

\begin{thm}[ABKS decomposition, see {\cite[Section 3.2]{ABKS}}]
\label{thm:ABKS}
Suppose $\Gamma = (G,\ell)$ is a metric graph with a combinatorial model.
Let $\mathcal T(G)$ denote the set of spanning trees of $G$.
Then
\begin{equation*}
 \Pic^g(\Gamma) = \bigcup_{T \in \mathcal T(G)} C_T
\end{equation*}
where 
\[ C_T = \{ [x_1 + \cdots + x_g] :  E(G) \backslash E(T) = \{ e_1, \ldots, e_g\},\, x_i \in e_i \}\]
denotes the set of divisor classes represented by summing 
a point from each edge of $G$ not in $T$.
The cells $C_T$ have disjoint interiors, as $T\in \mathcal T(G)$ varies.

For fixed $T$, if we parametrize each edge $e_i\not\in E(T)$ as the closed real interval $[0, \ell(e_i)]$,
there is a natural surjective map $\prod_{i=1}^g [0, \ell(e_i)] \to C_T$.
This map always restricts to a homeomorphism on the respective interiors
$\prod_{i=1}^g (0,\ell(e_i)) \to C_T^\circ$,
but may be non-injective on the boundary.
\end{thm}
The proof is to combine Theorem~\ref{thm:break-div}
with the definition of break divisor,
using the auxiliary data of the spanning tree.
Since $\Pic^g(\Gamma)$ is canonically homeomorphic to $\Br^g(\Gamma)$,
we may view Theorem~\ref{thm:ABKS} as a decomposition of $\Br^g(\Gamma)$.

\begin{rmk}
If  we take the combinatorial model for $\Gamma$ to be sufficiently subdivided,
then   for each $T = G \backslash \{ e_1,\ldots, e_g\}$,
the surjection
$\prod_{i=1}^g [0,\ell(e_i)]\to C_T$ 
is a (global) homeomorphism.
In particular, for this to hold  it suffices that $G$ has girth~$> g$
(i.e.  every cycle contains more than $g$ edges). 
A necessary condition is that $G$ has no loops or parallel edges (if $g \geq 2$).
\end{rmk}

\begin{eg}
\label{eg:abks}
Consider the metric graph shown on the left side of Figure~\ref{fig:abks}.
Its minimal combinatorial model $\Gamma = (G,\ell)$ 
contains two vertices and three edges.
The associated ABKS decomposition of $\Pic^2(\Gamma)$ is shown on the right side of Figure~\ref{fig:abks};
segments on the boundary are glued to the parallel boundary segment.
There are three cells $C_T$, corresponding to the three spanning trees in $G$.

Here $\Pic^2(\Gamma)$ is homeomorphic to a %$2$-dimensional 
torus (cf. Theorem~\ref{thm:abel-jacobi}).
Each cell $C_T$ is homeomorphic to a rectangle with a pair of opposite vertices glued together.
\begin{figure}[h]
\begin{minipage}{0.35\textwidth}
	\centering
\begin{tikzpicture}[scale=1.0]
		\coordinate (A) at (0,0);
		\coordinate (C) at (0,2.0);
		\coordinate (E1) at (-1,1);
		\coordinate (E2) at (0,1);
		\coordinate (E3) at (1.5,1);
	
		\draw (A) -- (C);
		\draw[rounded corners=12] (A) -- ++(-1,0) -- ++(0,2.0) -- (C);
		\draw[rounded corners=12] (A) -- ++(1.5,0) -- ++(0,2.0) -- (C);
	\end{tikzpicture}
\end{minipage}
\begin{minipage}{0.55\textwidth}
	\centering
\begin{tikzpicture}[scale=1.0]
		\coordinate (O) at (0,0);
		\coordinate (Q) at (-1.3,1);
		\coordinate (R) at (1.3,2);
		\coordinate (S) at (0,3);
		\coordinate (Ol) at ($(O) + (-1,0)$);
		\coordinate (Ql) at ($(Q) + (-1,0)$);
		\coordinate (Sl) at ($(S) + (-1,0)$);
		\coordinate (Or) at ($(O) + (3,0)$);
		\coordinate (Rr) at ($(R) + (3,0)$);
		\coordinate (Sr) at ($(S) + (3,0)$);

		\draw (O) -- (R) -- (S) -- (Q) -- cycle;
				\draw (Ql) -- (Q);
		\draw (R) -- (Rr);
		\begin{scope}[decoration={
			markings,
			mark=at position 0.5 with {\arrow{Latex[scale width=2]}}
		}]
			\draw[postaction={decorate}] (Ol) -- (Or);
			\draw[postaction={decorate}] (Sl) -- (Sr);
\end{scope}
		\begin{scope}[decoration={
			markings,
			mark=at position 0.5 with {\arrow{Latex[open,scale width=2]}}
		}]
			\draw[dotted,postaction={decorate}] (Ql) -- (Sl);
			\draw[dotted,postaction={decorate}] (Or) -- (Rr);
		\end{scope}
		\begin{scope}[decoration={
			markings,
			mark=at position 0.5 with {\arrow{>>}}
		}]
			\draw[dotted,postaction={decorate}] (Ol) -- (Ql);
			\draw[dotted,postaction={decorate}] (Rr) -- (Sr);
		\end{scope}

		\begin{scope}[shift={(-0.1,1.1)}, scale=0.4]
			\coordinate (A) at (0,0);
			\coordinate (C) at (0,2.0);
			\coordinate (E1) at (-1,1);
			\coordinate (E2) at (0,1);
			\coordinate (E3) at (1.5,1);
		
			\draw (A) -- (C);
			\draw[rounded corners=6] (A) -- ++(-1,0) -- ++(0,2.0) -- (C);
			\draw[rounded corners=6] (A) -- ++(1.5,0) -- ++(0,2.0) -- (C);
		
			\foreach \P in {E1,E2} {
				\fill (\P) circle (4pt);
			}
			\end{scope}
			\begin{scope}[shift={(2.0,0.5)}, scale=0.4]
				\coordinate (A) at (0,0);
				\coordinate (C) at (0,2.0);
				\coordinate (E1) at (-1,1);
				\coordinate (E2) at (0,1);
				\coordinate (E3) at (1.5,1);
			
				\draw (A) -- (C);
				\draw[rounded corners=6] (A) -- ++(-1,0) -- ++(0,2.0) -- (C);
				\draw[rounded corners=6] (A) -- ++(1.5,0) -- ++(0,2.0) -- (C);
			
				\foreach \P in {E1,E3} {
					\fill (\P) circle (4pt);
				}
				\end{scope}
				\begin{scope}[shift={(2.0,2.1)}, scale=0.4]
					\coordinate (A) at (0,0);
					\coordinate (C) at (0,2.0);
					\coordinate (E1) at (-1,1);
					\coordinate (E2) at (0,1);
					\coordinate (E3) at (1.5,1);
				
					\draw (A) -- (C);
					\draw[rounded corners=6] (A) -- ++(-1,0) -- ++(0,2.0) -- (C);
					\draw[rounded corners=6] (A) -- ++(1.5,0) -- ++(0,2.0) -- (C);
				
					\foreach \P in {E2,E3} {
						\fill (\P) circle (4pt);
					}
					\end{scope}
				\end{tikzpicture}
\end{minipage}
\caption{Metric graph $\Gamma$, on left, and ABKS decomposition of $\Pic^2(\Gamma)$, on right.} % for a genus $2$ metric graph $\Gamma$.}
\label{fig:abks}
\end{figure}
\end{eg}

\begin{prop}
Let $q\in \Gamma$ be an arbitrary basepoint on a genus $g$ metric graph.
\begin{enumerate}[(a)]
\item 
For a generic divisor class $[D]$ of degree $g$,
the reduced divisor $\reddiv{q}{D} $ is equal to the break divisor
$ \br[D]$.

\item 
For a generic divisor class $[D]$ of degree $n$,
the reduced divisor $\reddiv{q}{D}$
is equal to
\[ \reddiv{q}{D} = (n-g)q + E\]
where $E = \br[D - (n-g)q]$ is a break divisor.
\end{enumerate}
\end{prop}
\begin{proof}
(a) The reduced divisor $\reddiv{q}{D}$ is an effective divisor 
%linearly equivalent to $D$, 
in the divisor class $[D]$, as is the break divisor $\br[D]$.
For a generic divisor class of $[D]$ of degree $g$, there is a unique effective divisor in the divisor class~\cite[Theorem 4.20]{ABKS}.

(b) This follows from (a) and property (RD2) in Section~\ref{sec:reduced-prop}.
\end{proof}

The following definition is used later to define the stable Weierstrass locus of a divisor, Definition~\ref{dfn:wstab}.
\begin{dfn}
\label{dfn:stable-reduced-div}
Given a divisor class $[D]$
on $\Gamma$ and a point $x \in \Gamma$,
the {\em stable reduced divisor} 
$\streddiv{x}{D}$
is
\begin{equation}
\label{eq:stable-reduced-div}
\streddiv{x}{D} = (n-g)x + \br[D-(n-g)x] .
\end{equation}
In other words, the stable reduced divisor map $\stred_x : \Pic(\Gamma) \to \Div(\Gamma)$
is determined by the two conditions:
\begin{enumerate}
\item[(SD1)] 
$\streddiv{x}{D} = \br[D]$
for every divisor class %$[D] \in \Pic^g(\Gamma)$
$[D]$ of degree $g$, and

\item[(SD2)] 
$\streddiv{x}{D + mx} = mx + \reddiv{x}{D}$ for every $[D]$ and every integer $m$.
\end{enumerate}
\end{dfn}

\section{Canonical measure and resistor networks}
\label{sec:canonical}
In this section we recall the definition of the Zhang canonical measure on a metric graph~\cite{Z} 
via the perspective of resistor networks following Baker--Faber \cite{BF}. 

%We may view this construction as a one-dimensional  analogue 
%of Gaussian curvature on a closed two-dimensional surface.

%The idea is to define, for a metric graph, an analogue of Gaussian curvature for a closed two-dimensional surface. 
%The curvature at a point on a surface is defined locally, only taking into account a small neighborhood of the point, 
%yet integrating curvature over the whole surface yields a global invariant, the genus.

%If we take a small neighborhood of a one-dimensional metric graph in the interior of an edge, with its metric data, 
% there is no way to assign a good analogue of curvature since all such neighborhoods look the same up to isometry.
%Thus we need a slightly more generous notion of ``small neighborhood''. 
%What works  is to look at a neighborhood equipped with the data of the effective resistance function, in addition to the metric data. 
%This effective resistance function is defined non-locally, so it is less surprising that this gives us a global invariant when integrating %over the whole graph.

\subsection{Voltage function}
\label{sec:voltage}
We  view a metric graph $\Gamma$ as a resistor network by interpreting an edge of length $L$ as a resistor of resistance $L$.
%Note that this is well-defined on a metric graph due to the series rule for combining resistances, 
%so we have compatibility with subdividing an edge into edges of shorter length.
%This interpretation is not only mathematically convenient, but physically 
%honest---the electrical resistance of a wire is directly proportional to its length, 
%a fact known as Pouillet's law.

On a resistor network we may send current from one point to another.
On a given segment, the voltage drop across the segment is equal to the resistance (i.e. length) of the segment 
multiplied by the amount of current passing through the segment---this is Ohm's law.
Under an externally-applied current, % flow, % or voltage differential,
the flow of current within the network is determined by Kirchhoff's circuit laws: 
the current law says that the sum of directed currents out of any point is equal to zero 
(accounting for external currents),
and the voltage law says that the sum of directed voltage differences around any closed loop is equal to zero.
It is a well-known mathematical result that Kirchhoff's circuit laws can be solved uniquely 
for any externally-applied current flow which satisfies conservation of current
(i.e. internal current flows are unique). %, voltage function is unique up to additive constant) 
To some, it is also a well-known empirical result.

Our convention is that current flows from higher voltage to lower voltage.

\begin{dfn}[physics version]
%Let $\Gamma$ be a metric graph.
Given points $y,z\in \Gamma$,
the {\em voltage function} (or unit potential function)
$  \potent{z}{y} : \Gamma \to \RR $
is defined by
\[ \potent{z}{y}(x) = \text{voltage at $x$ when sending one unit of current from $y$ to $z$,}\]
such that 
$\potent{z}{y}(z) = 0$, 
i.e. 
the network is ``grounded'' at $z$.
\end{dfn}

\begin{dfn}[math version] % definition--theorem]
%Let $\Gamma$ be a metric graph.
Given points $y,z\in \Gamma$,
the {\em unit potential function} (or voltage function)
$\potent{z}{y}$ is the unique function in $\PL_\RR(\Gamma)$
satisfying the conditions
\[ \Divisor(\potent{z}{y}) = z - y \in \Div_{\RR}^0(\Gamma) \quad\text{and}\quad \potent{z}{y}(z) = 0.\]
\end{dfn}
%\begin{proof}
For the existence and uniqueness of $\potent{z}{y}$, see Theorem 6 and Corollary 3 of Baker--Faber \cite{BF}.
Note that they use the notation $j_z(y,-)$ for $\potent{z}{y}(-)$.
%\end{proof}

The function $\potent{z}{y}$ satisfies the following properties:
\begin{enumerate}
\item[(V1)] for any $x\in \Gamma$, we have $0 =\potent{z}{y}(z) \leq \potent{z}{y}(x) \leq \potent{z}{y}(y)$;
\item[(V2)] $\potent{z}{y}(x)$ is piecewise linear in $x$;
\item[(V3)] $\potent{z}{y}{(x)}$ is continuous in $x$, $y$, and $z$.
%\item $\potent{z}{y}{x}$ is linear in $x$ in the interior of each edge, with $y$ and $z$ fixed, i.e.
%\item $\potent{z}{y} \in \PL_\RR\Gamma$.
%\item $\potent{z}{y}(x) = \potential{z}{x}{y}$ (* not obvious!?)
\end{enumerate}

\begin{prop}
\label{prop:voltage-sym}
The voltage function $\potent{z}{y}$ obeys the following symmetries.
\begin{enumerate}[(a)]
\item 
\label{it:3term}
For any three points $x,y,z\in \Gamma$, we have
%\begin{equation*}
%\label{eq:3term}
$\potent{z}{y}(x) = \potent{z}{x}(y)$. %(* not obvious!?)
%\end{equation*}

\item 
\label{it:4term}
For any four points $x,y,z,w \in \Gamma$, we have
%\begin{equation*}
%\label{eq:4term}
$
\potent{z}{y}(x) - \potent{z}{y}(w) = \potent{w}{x}(y) - \potent{w}{x}(z)
$.
%\end{equation*}
\end{enumerate}
\end{prop}
\begin{proof}
See Baker--Faber \cite[Theorem 8]{BF}; they refer to \eqref{it:4term} as the ``Magical Identity.''
Note that \eqref{it:3term} follows from \eqref{it:4term} by setting $z=w$.
\end{proof}

\begin{rmk}
We many interpret any function $f\in \PL_\RR(\Gamma)$
as a voltage function on $\Gamma$,
which results from the externally applied current $\Divisor(f)\in \Div_\RR(\Gamma)$.
In other words, the voltage $f$ results from sending current from $\poles(f)$ to $\zeros(f)$
in $\Gamma$.

The existence of $\potent{z}{y} \in \PL_\RR(\Gamma)$ for any $y,z\in \Gamma$ implies that 
the principal divisor map
$ \Divisor : \PL_\RR(\Gamma) \to \Div_\RR^0(\Gamma)$
is surjective.
This verifies the claim made in Remark~\ref{rmk:real-div} concerning the exactness of the sequence
\[ 0 \to \RR \xrightarrow{\text{const}} \PL_\RR(\Gamma)\xrightarrow{\Divisor} \Div_\RR(\Gamma) \xrightarrow{\text{deg}} \RR \to 0 .\]
\end{rmk}

\begin{prop}[Slope-current principle]
\label{prop:slope-current}
Suppose $f \in \PL_\RR(\Gamma)$ 
% is a piecewise linear function whose principal divisor
has zeros $\zeros(f)$ and poles $\poles(f)$ of degree $d\in \RR$.
Then the slope of $f$ is bounded by $d$, i.e.
\[ |f'(x)| \leq d \qquad \text{for any $x\in\Gamma$ where $f$ is linear}.\]
(This bound %is sharp; it 
is attained only on bridge edges,
and only when all zeros are on one side of the bridge
and all poles are on the other side.)
\end{prop}
\begin{proof}
Let $\lambda = f(x) $.
Then the ``tropical preimage''
\[ \divfiber{f}(\lambda) := \poles(f) + \Divisor(\max\{f, \lambda\})\]
has multiplicity $|f'(x)|$ at $x$, since the outgoing slopes of $\max\{f,\lambda\}$ 
at $x$ are $|f'(x)|$ and $0$.
(Note $x$  cannot be in $\poles(f)$ since $f$ is linear at $x$.)
Since the divisor $\divfiber{f}(\lambda)$ is effective of degree $d$, this implies 
$|f'(x)| \leq d$ as desired.
\end{proof}
\begin{rmk}
The above proposition is obvious from its ``physical interpretation'': 
$f$ gives the voltage in the resistor network $\Gamma$ when subjected to an external current
described by
$\poles(f)$ units flowing into the network
and $\zeros(f)$ units flowing out.
The slope $|f'(x)|$ is equal to the current flowing through the wire containing $x$, 
which must be no more than the total in-flowing (or out-flowing) current.
\end{rmk}

Next we address how the voltage function $\potent{z}{y} \in \PL_\RR(\Gamma)$
may be approximated by a sequence of functions in $\PL_\ZZ(\Gamma)$ (up to rescaling),
which depend on reduced divisors.
(We only use property (RD3) of reduced divisors in the following arguments.)

\begin{prop}[Quantitative version of unit potential approximation]
\label{prop:quant-voltage-approx}
Let $\Gamma$ be a metric graph of genus $g$,
and let $D_{\en}$ be a degree $\en$ divisor on $\Gamma$.
Fix two points $y,z \in\Gamma$, and let $f_{\en}$ be the unique function in $\PL_\ZZ(\Gamma)$ satisfying
\[ \Divisor(f_{\en}) = \reddiv{z}{D_{\en}} - \reddiv{y}{D_{\en}} 
\qquad\text{and}\qquad
 f_{\en}(z) = 0.\]
If $n > g$ then for any $x\in \Gamma$,
we have
$|(\frac1{\en - g} f_{\en} - \potent{z}{y})'(x)| \leq \frac{g}{\en - g}$.
\end{prop}
\begin{proof}
Let $\phi_{\en} = \frac1{\en-g} f_{\en} - \potent{z}{y}$.
%
%Note that  $\phi_{\en}$ is a  continuous, piecewise-differentiable function with $\phi_{\en}(z) = 0$,
%so for an arbitrary $x\in \Gamma$ we may calculate the value of $\phi_{\en}(x)$ by integrating
%the derivative of $\phi_{\en}$ along some path in $\Gamma$ from $z$ to $x$.
%The  length of such a path is bounded (independent of $x$) since $\Gamma$ is compact, 
%so to show that $\phi_{\en} \to 0$ uniformly %as $\en\to \infty$ 
%it suffices to show that 
%the magnitude of the derivative $|\phi_{\en}'|$ approaches $0$ uniformly.
%
To show that $|\phi_{\en}'(x)| \leq \frac{g}{\en - g}$ for any $x\in \Gamma$, we apply the slope-current principle, Proposition~\ref{prop:slope-current}.
By Riemann's inequality, the reduced divisors $\reddiv{y}{D_{\en}}$, $\reddiv{z}{D_{\en}}$ may be expressed as 
\[ \reddiv{y}{D_{\en}} = (\en - g) y + E_{y}
\qquad\text{and}\qquad
\reddiv{z}{D_{\en}} = (\en - g) z + E_{z}\]
for some  effective divisors $E_{y}$, $E_z$ of degree $g$.
Thus the principal divisor associated to $\frac1{\en-g} f_{\en}$ is
\[ 
\Divisor\left(\frac1{\en-g} f_{\en}\right) 
% = \frac1{\en} \reddiv{z}{D_{\en}} - \frac1{\en} \reddiv{y}{D_{\en}}
= z + \frac1{\en-g}E_{z} - y - \frac1{\en-g} E_{y} .
\]
Recall that $\Divisor(\potent{z}{y}) = z - y$;
it follows that the principal $\RR$-divisor associated to $\phi_{\en}$ is
\begin{equation*}
\Divisor(\phi_{\en}) = \Divisor\left(\frac1{\en-g} f_{\en} - \potent{z}{y} \right) 
= \frac1{\en-g} E_{z} - \frac1{\en-g} E_{y} .
\end{equation*}
In particular, $\Divisor(\phi_n)$ is a difference of effective $\RR$-divisors of degree $\frac{g}{\en-g}$, 
so the zeros $\zeros(\phi_{\en})$ and poles $\poles(\phi_{\en})$ each have degree at most $\frac{g}{\en-g}$.
By Proposition \ref{prop:slope-current}, this implies $|\phi_{\en}'(x)| \leq \frac{g}{\en - g}$
as claimed.
\end{proof}

\begin{prop}[Discrete approximation of unit potential function]
\label{prop:voltage-approx}
Let $\{D_{\en} : { \en\geq 1} \}$ be a sequence of divisors on $\Gamma$ with $\deg D_{\en} = \en$.
Fix two points $y, z \in \Gamma$.
Let $\reddiv{y}{D_{\en}}$ and $\reddiv{z}{D_{\en}}$ denote the $y$-- and $z$--reduced representatives 
in the divisor class $[D_{\en}]$,
and let $f_{\en}$ be the unique function in $\PL_\ZZ(\Gamma)$ satisfying
\[ \Divisor(f_{\en}) = \reddiv{z}{D_{\en}} - \reddiv{y}{D_{\en}} 
\qquad\text{and}\qquad 
f_{\en}(z) = 0.\]
Then the functions $\frac1{\en} f_{\en} $ converge uniformly to $ \potent{z}{y}$ as $\en \to \infty$.
\end{prop}
\begin{proof}
If the sequence $\frac1{\en} h_{\en}$ converges to a limit,
then the sequence $\frac1{\en +c} h_{\en}$  must also converge to the same limit as $n\to\infty$,
for any constant $c$.
Thus it suffices to show that the functions $\frac1{\en-g} f_{\en}$
converge uniformly to $\potent{z}{y}$.

Let $\phi_{\en} = \frac1{\en-g} f_{\en} - \potent{z}{y}$ for $n > g$.
We claim that the sequence of functions 
$\{ \phi_{\en} \in \PL_\RR(\Gamma): \en \geq g+1\}$
converges uniformly to $0$. 
Note that each  $\phi_{\en}$ is a  continuous, piecewise-differentiable function with $\phi_{\en}(z) = 0$,
so for an arbitrary $x\in \Gamma$ we may calculate the value of $\phi_{\en}(x)$ by integrating
the derivative of $\phi_{\en}$ along some path in $\Gamma$ from $z$ to $x$.
The  length of such a path is bounded (independent of $x$) since $\Gamma$ is compact, 
so to show that $\phi_{\en} \to 0$ uniformly %as $\en\to \infty$ 
it suffices to show that 
the magnitude of the derivative $|\phi_{\en}'|$ approaches $0$ uniformly as $\en\to\infty$.
This convergence is provided by Proposition \ref{prop:quant-voltage-approx}, which states that $|\phi_{\en}'(x)| \leq \frac{g}{\en-g}$.
\end{proof}

\begin{rmk}
We can interpret Proposition \ref{prop:voltage-approx} as follows:
the existence of the unit potential function $\potent{z}{y}: \Gamma \to \RR$ follows from 
(tropical) Riemann's inequality for divisors on $\Gamma$.
\end{rmk}

\subsection{Resistance function}
\label{subsec:resistance}
In this section we recall the definition of the 
(Arakelov--)Zhang canonical measure $\mu$
on a metric graph.

\begin{dfn}
Let $r : \Gamma \times \Gamma \to \RR$ denote the 
{\em effective resistance} function on the metric graph $\Gamma$. 
Namely, viewing $\Gamma$ as a resistor network,
%each edge is a resistor whose resistance is equal to its length, and
\begin{align*}
r(x,y) &= \text{effective resistance between $x$ and $y$} \\
&= \text{total voltage drop when sending 1 unit of current from $x$ to $y$}.
\end{align*}
If we wish to emphasize the underlying graph, we write $r(x,y; \Gamma)$ for $r(x, y)$.
\end{dfn}
In terms of the unit potential function (see Section \ref{sec:voltage}), $r(x,y) = \potent{y}{x}(x)$.
% It is straightforward to verify that 
The resistance function satisfies the following properties:
\begin{enumerate}
%\item $r(x,y) \geq 0$,
\item $r(x,x) = 0$;
\item 
$r(x,y) > 0$ if $x \neq y$;
\item $r(x,y) = r(y,x)$;
\item $r(x,y)$ is continuous with respect to $x$ and $y$.
\end{enumerate}

In contrast with the voltage function $\potent{z}{y}$, 
the function $x \mapsto r(x,y)$ is not piecewise linear. 
We will see that it is instead piecewise quadratic.

\begin{eg}
\label{eg:circle-resistance}
Let $\Gamma$ be a circle of circumference $L$. 
By choosing a basepoint which we denote as $0$, we may parametrize $\Gamma$ with the interval $[0,L]$.
Identifying points in this way, we have
\begin{align*}
r(x,0) &= \text{parallel combination of resistances $x$ and $L-x$} \\
&= \frac{x(L-x)}{x + (L-x)} = x - \frac1 L x^2.
\end{align*}
The effective resistance is maximized when $x = \frac12 L$, with maximum value $ \frac14 L$.
The effective resistance is minimized when $x= 0$ or $x = L$, with effective resistance $0$.
\end{eg}

There is a special case of effective resistance which will be particularly useful in the following sections.

\begin{dfn}
Given  a segment $e$ in a metric graph $\Gamma$, %with endpoints $\startv$ and $\tailv$,
the {\em deleted effective resistance} $\lengtheff{\Gamma \backslash e}$
is the effective resistance between endpoints of $e$ 
in the $e$-deleted subgraph; that is, if $s,t$ are the endpoints of $e$ then
\begin{equation*}
  \lengtheff{\Gamma \backslash e} = r(s,t;\Gamma \backslash e). 
\end{equation*}
\end{dfn}
Note that $\lengtheff{\Gamma\backslash e} = 0$ when $e$ is a loop,
and $\lengtheff{\Gamma \backslash e} = +\infty$ when $e$ is a bridge.
The rule for combining resistances in parallel implies that
for a segment $e$ with endpoints $\startv$ and $\tailv$,
\begin{equation*}
r(s,t; \Gamma)  
= \left( \frac{1}{\length{e}} + \frac{1}{\lengtheff{\Gamma\backslash e}} \right)^{-1}
= \frac{\length{e} \lengtheff{\Gamma\backslash e}}{\length{e} + \lengtheff{\Gamma\backslash e}} .
\end{equation*}

\subsection{Canonical measure}

\begin{dfn}
The {\em Zhang canonical measure} $\mu = \mu_\Gamma$ on a metric graph $\Gamma$ is 
the continuous measure defined by
\[ 
\mu % = \mu(dx) 
= -\frac12 \frac{d^2}{dx^2} r(x,y_0) \,dx.\]
where $x$ is a length-preserving parameter on a segment,
$dx$ is the Lebesgue measure,
and $y_0 \in \Gamma$ is fixed.
This defines $\mu$ on the open dense subset of $\Gamma$ where the second derivative exists;
at the finite set of points where $r(-,y_0)$ is not differentiable, 
or where the valence of $x$ differs from 2, 
we let $\mu = 0$.
\end{dfn}
For brevity, we will refer to the Zhang canonical measure simply as the canonical measure.

\begin{rmk}
The first derivative of a smooth function on $\Gamma$ is only well-defined up to a choice of sign, 
since there are two directions in which we could parametrize any segment.
The second derivative, however, is well-defined on each segment 
(without choosing an orientation). 
Either choice of direction yields the same second derivative
since $(\pm 1)^2 = 1$.
\end{rmk}

\begin{rmk}
The definition of the canonical measure is independent of the choice of basepoint $y_0$ 
because of the ``Magical Identity'' in Proposition \ref{prop:voltage-sym} (b).
%[see Baker--Faber Theorem 11, ``magical identity'']
Namely, for two basepoints $y_0, z_0$ we have
$\potent{y_0}{x}(x) - \potent{y_0}{x}(z_0) = \potent{z_0}{x}(x) - \potent{z_0}{x}(y_0)$ 
which implies
\begin{align*}
 r(x,y_0) - r(x,z_0) &= \potent{y_0}{x}{(x)} - \potent{z_0}{x}{(x)}  \\
 &= \potent{y_0}{x}(z_0) - \potent{z_0}{x}(y_0) = \potent{y_0}{z_0}(x) - \potent{z_0}{y_0}(x).
\end{align*}
Since the voltage functions $\potent{y_0}{z_0}, \potent{z_0}{y_0}$ are piecewise linear,
we have 
\[ \frac{d^2}{dx^2}(r(x,y_0) - r(x,z_0)) 
 = \frac{d^2}{dx^2}(\potent{y_0}{z_0}(x) - \potent{z_0}{y_0}(x)) = 0.\]
\end{rmk}

\begin{rmk}\label{rmk:canonical-differ}
The definition of canonical measure given here differs from those used by previous authors in the following ways.
	
	\begin{enumerate}[(a)]
	\item 
	The canonical measure given here % is equal to 
	differs by a multiplicative factor from Zhang's canonical measure 
\cite[Section 3, Theorem 3.2 c.f. Lemma 3.7]{Z} associated to the canonical divisor $D = K$.
	% up to a multiplicative factor.
Our canonical measure is normalized to satisfy $\mu(\Gamma) = g$ rather than $\mu(\Gamma) = 1$.

\item 
	The canonical measure given here differs from that used by Baker--Faber \cite{BF}, 
	in that our $\mu$ does not have a discrete part supported at the points of $\Gamma$ with valence different from $2$.
The canonical measure of Baker--Faber is equal to Zhang's canonical measure associated to $D = 0$.
\end{enumerate}
\end{rmk}

\begin{eg}[Canonical measure on circle]
\label{eg:circle-canonical}
If $\Gamma$ is a circle of circumference $L$, 
then by Example \ref{eg:circle-resistance} we have ${ r(x,0) = x - \frac1L x^2 }$ 
so the canonical measure is ${ \mu = \frac1 L dx }$.
The total measure on the metric graph is $\mu(\Gamma) = 1$.
\end{eg}

\begin{eg}[Canonical measure on theta graph]
\label{eg:genus2-canonical}
Consider the metric graph $\Gamma$ of genus 2 shown in Figure \ref{fig:genus2len}, with edge lengths $a,b,c$.

\begin{figure}[h]
\begin{minipage}{0.45\textwidth}
\centering
\begin{tikzpicture}[scale=1.0]
\coordinate (A) at (0,0);
\coordinate (B) at (0,2);
\coordinate (C) at (1.2,0);
\coordinate (D) at (1.2,2);
\coordinate (E) at (3,0);
\coordinate (F) at (3,2);

\draw[rounded corners=10] (C) -- (A) -- node[left] {$a$} (B) -- (D);
\draw[dotted] (C) -- node[right] {$b$} (D);
\draw[dotted,rounded corners=10] (C) -- (E) -- node[right] {$c$} (F) -- (D);

\foreach \c in {C,D} {
	\fill (\c) circle (2.5pt);
}
\end{tikzpicture}
\\
$\length{e} = a$
\end{minipage}
\begin{minipage}{0.45\textwidth}
\centering
\begin{tikzpicture}[scale=1.0]
\coordinate (A) at (0,0);
\coordinate (B) at (0,2);
\coordinate (C) at (1.2,0);
\coordinate (D) at (1.2,2);
\coordinate (E) at (3,0);
\coordinate (F) at (3,2);

\draw[dotted,rounded corners=10] (C) -- (A) -- node[left] {$a$} (B) -- (D);
\draw (C) -- node[right] {$b$} (D);
\draw[rounded corners=10] (C) -- (E) -- node[right] {$c$} (F) -- (D);

\foreach \c in {C,D} {
	\fill (\c) circle (2.5pt);
}
\end{tikzpicture}
\\
$\lengtheff{\Gamma\backslash e} = \frac{bc}{b+c}$ %bc/(b+c)$
\end{minipage}
\caption{Edge length and deleted effective resistance on a metric graph.}
\label{fig:genus2len}
\end{figure}
Let $e$ be the edge of length $a$;
on $e$ we have $\length{e} = a$ and $\lengtheff{\Gamma \backslash e} = \frac{bc}{b+c}$.
When measuring effective resistance between points in the interior of $e$, 
we can  think of $\Gamma$ as a circle of total length 
$\length{e} + \lengtheff{\Gamma \backslash e} = \frac{ab + ac + bc}{b + c}$.
Thus the canonical measure on this edge is $\mu = \frac{b+c}{ab+ac+bc}dx$,
by the computation for a circle in Example \ref{eg:circle-resistance}.
The total measure on this edge is $\mu(e) = \frac{ab+ac}{ab+ac+bc}$, 
and by symmetry the total measure on the metric graph is $\mu(\Gamma) = 2$.
\end{eg}

\begin{prop}[{Baker--Faber \cite[Theorem 12]{BF}}]
\label{prop:canonical-formula}
The  canonical measure $\mu$ on a metric graph $\Gamma$
is a piecewise-constant multiple of the Lebesgue measure,
which vanishes on all bridge segments.
On a non-bridge segment $e$ in $\Gamma$, 
\[ \mu|_e = \frac{1}{\length{e} + \lengtheff{\Gamma\backslash e}} dx\]
where $\length{e}$ denotes the length of $e$ and  $\lengtheff{\Gamma\backslash e}$ denotes the effective resistance between the endpoints of $e$ on the graph after removing the interior of $e$.
% For a bridge segment, $\mu|_e = 0$.
\end{prop}
The proof idea should be clear from Example \ref{eg:genus2-canonical}.
If a segment $e$ is subdivided into $e_1 \sqcup e_2$,
the expression $\mu|_e$ agrees with $\mu|_{e_1} = \mu|_{e_2}$.
Note that our $\mu$ is defined to be the continuous part of Baker--Faber's $\mu_{\text{can}}$ (see Remark~\ref{rmk:canonical-differ}).

\begin{cor}
\label{cor:canonical-bound}
Let $\Gamma$ be a metric graph with canonical measure $\mu$,
and let $e$ be a segment in $\Gamma$. 
% (i.e. $e$ is subspace isometric to a closed interval, whose interior points all have valence $2$ in $\Gamma$).
Then
\begin{enumerate}[(a)]
\item 
$ 0\leq \mu(e) \leq 1$;
\item 
$\mu(e) = 0$ if and only if $e$ is a bridge edge;
\item 
$\mu(e) = 1$ if and only if $e$ is a loop edge.
\end{enumerate}
\end{cor}
\begin{proof}
By Proposition \ref{prop:canonical-formula},
$\mu(e) = 0$ for bridges
and 
$\mu(e) = 
\frac{\length{e}}{\length{e} + \lengtheff{\Gamma\backslash e}} $
otherwise.
\end{proof}

The following corollary to Proposition~\ref{prop:canonical-formula} gives an explicit relation between the canonical measure and effective resistance.
\begin{cor}
\label{cor:canonical-via-resistance}
If $e$ is a segment of $\Gamma$ with endpoints $s$ and $t$,
and $e$ is not a bridge or loop, then
$\displaystyle
	\mu(e) = \frac{r(s, t; \Gamma)}{r(s, t; \Gamma \setminus e)}.
$
\end{cor}

\begin{prop}[Foster's Theorem]
Let $\Gamma$ be a metric graph of genus $g$, and let $\mu$ be the canonical measure on $\Gamma$.
Then the total measure on $\Gamma$ is
$\mu(\Gamma) = g$.
\end{prop}
\begin{proof}
See Baker--Faber \cite[Corollary 5 and  Corollary 6]{BF} and Foster \cite{F}.
\end{proof}

\subsection{Energy and reduced divisors}
\label{sec:red-div}
Here we give a  definition of $q$-reduced divisors on a metric graph.
We will only need to use $q$-reduced divisors for effective divisor classes,
so we restrict our discussion here to the effective case.

\begin{dfn}
Given a basepoint $q$ on $\Gamma$,
we define the {\em $q$-energy} $\mathcal E_q : \Gamma \to \RR$
by 
\[ \mathcal E_q(y) = \potent{q}{y}(y) = r(y,q). \]
Given an effective divisor $D = \sum_i y_i$, 
we define the {\em $q$-energy} $\mathcal E_q(D)$ by
\[ \mathcal E_q(D) = \sum_i \sum_j \potent{q}{y_i}(y_j)  .\]
%\[ \mathcal E_q(D) = \sum_i \potent{q}{D}(y_i)  
%\qquad\text{where}\qquad \potent{q}{D} = \sum_j \potent{q}{y_j} : \Gamma \to \RR.\]
\end{dfn}
Note that 
\begin{itemize}
\item $\mathcal E_q(D)\geq 0$, 
\item $\mathcal E_q(D)$ 
is strictly positive if $D$ has support outside of $q$,
\item  $\mathcal E_q(D) \geq \sum_{i} \mathcal E_q(y_i)$, and in general this inequality is strict.
\end{itemize}

\begin{thm}[Baker--Shokrieh {\cite[Theorem A.7]{BS}}] 
Fix a basepoint $q\in \Gamma$, and let $D$ be an effective divisor on $\Gamma$.
There is a unique divisor $D_0 \in |D|$ which minimizes the $q$-energy, i.e. 
such that
\[ \mathcal E_q(D_0) < \mathcal E_q(E) \qquad\text{for all}\qquad E\in |D|,\, E\neq D_0.\]
\end{thm}
%\begin{proof}
%See Baker--Shokrieh \cite[Theorem A.7]{BS}.
%\end{proof}

\begin{dfn}
The {\em $q$-reduced divisor} $\reddiv{q}{D}$ is the unique divisor in $|D|$ which
minimizes the $q$-energy $\mathcal E_q$.
\end{dfn}
Note that this definition is non-standard;
the standard definition for reduced divisor is a combinatorial condition
which can be phrased in the language of chip-firing, 
see \cite[p. 4854]{A-reddiv}, \cite[Definition 2.3]{ABKS}.

\begin{eg}
In Figure \ref{fig:red}
we show a degree $4$ divisor, on the left,
and its reduced representative with respect to basepoint $q$,
on the right.
\begin{figure}[h]
\begin{minipage}{0.45\textwidth}
\centering
\begin{tikzpicture}
	\coordinate (A) at (0,-0.4);
	\coordinate (B) at (0,0);
	\coordinate (C) at (1,0);
	\coordinate (D) at (1,0.4);
	\coordinate (E) at (2,0);
	\coordinate (F) at (2,0.4);
	\coordinate (G) at (3,-0.4);
	\coordinate (H) at (3,0);

	\draw[rounded corners=5] (B) -- (A) -- +(0,-0.3) -- +(-1,-0.3) -- +(-1,0.7) -- +(0,0.7) -- cycle;
	\draw (B) -- (C);
	\draw[rounded corners=5] (D) -- +(0,0.3) -- +(1,0.3) -- (F) -- (E) -- +(0,-0.3) -- +(-1,-0.3) -- (C) -- cycle;
	\draw (E) -- (H);
	\draw[rounded corners=5] (G) -- (H) -- +(0,0.3) -- +(1,0.3) -- +(1,-0.7) -- +(0,-0.7) -- cycle;

	\foreach \c in {B,C,E,H} {
		\fill (\c) circle (2.5pt);
	}
	\draw[color=red] (3.5,0.3) node[above=0.2] {$q$} circle (3pt);
\end{tikzpicture}
\\
$D$
\end{minipage}
\begin{minipage}{0.45\textwidth}
\centering
\begin{tikzpicture}
	\coordinate (A) at (0,-0.4);
	\coordinate (B) at (0,0);
	\coordinate (C) at (1,0);
	\coordinate (D) at (1,0.4);
	\coordinate (E) at (2,0);
	\coordinate (F) at (2,0.4);
	\coordinate (G) at (3,-0.4);
	\coordinate (H) at (3,0);

	\draw[rounded corners=5] (B) -- (A) -- +(0,-0.3) -- +(-1,-0.3) -- +(-1,0.7) -- +(0,0.7) -- cycle;
	\draw (B) -- (C);
	\draw[rounded corners=5] (D) -- +(0,0.3) -- +(1,0.3) -- (F) -- (E) -- +(0,-0.3) -- +(-1,-0.3) -- (C) -- cycle;
	\draw (E) -- (H);
	\draw[rounded corners=5] (G) -- (H) -- +(0,0.3) -- +(1,0.3) -- +(1,-0.7) -- +(0,-0.7) -- cycle;

	\foreach \pos in {(1.5,0.7),(3.5,0.3),(3.5,-0.7)} {
		\fill \pos circle (2.5pt);
	}
	\draw[color=red] (3.5,0.3) node[above=0.2] {$q$} circle (3pt);
	\node[below left] at (3.5,0.3) {$2$};
\end{tikzpicture}
\\
$\reddiv{q}{D}$
\end{minipage}
\caption{A divisor (left) and its reduced divisor representative (right).}
\label{fig:red}
\end{figure}
\end{eg}

%\harry{Add statement that given a divisor, reduced divisor map is piecewise-geodesic}
For a fixed effective divisor $D$,
we may consider the family of reduced divisors $\reddiv{x}{D}$ as $x$ varies over $\Gamma$.
This association defines a map
\begin{align*}
\red_D : \Gamma &\to \Sym^d(\Gamma) \\
x &\mapsto  \reddiv{x}{D}
\end{align*}
whose image lies in the linear system $|D|$.
Amini proved that this map is continuous \cite[Theorem 3]{A-reddiv}.

\begin{thm}
\label{thm:piecewise-geodesic}
The reduced-divisor map is piecewise geodesic.
\end{thm}
\begin{proof}
We use the characterization of reduced divisors by Dhar's burning algorithm.
Dhar's burning algorithm says that a divisor $D$ is $x$-reduced if and only if a fire started at $x$ will spread to all of $\Gamma$.
If $D$ is not $x$-reduced, 
then a fire started at $x$ will spread to a open %(closed?) 
connected subset $U \subset \Gamma$,
and we obtain a divisor $D' \sim D$ which is closer to being $x$-reduced by
firing chips on the boundary $\partial U$ towards $U$.

Now suppose $D$ is $q$-reduced, and consider the family of divisors $D' = \reddiv{x}{D}$
where $x$ varies over a small segment $x \in [q, q + \epsilon]$  starting at $q$.
%Choose a point $q \in \Gamma$.
For a sufficiently small segment $[q, q + \epsilon]$,
we claim that the map $x \mapsto \reddiv{x}{D}$ is geodesic for
$x \in [q,q+\epsilon]$.
Consider starting a fire on $\Gamma$ at the point $q + \epsilon$,
with ``firefighter'' chips placed according to $\reddiv{q}{D}$.
Let $U$ denote the subset of $\Gamma$ that is on fire.
We consider three cases:

\underline{Case 1}: The set $U$ is not dense in $\Gamma$.

In this case, the boundary of $U$ forms a simple cut of $\Gamma$, and the map $x\mapsto \reddiv{x}{D}$ moves chips along this cut towards $U$.
Each segment has a chip with multiplicity one, except for the segment $[q,q+\epsilon]$ which may have higher multiplicity.

%\begin{figure}
%\centering
%\begin{tikzpicture}
%\coordinate (A) at (0,0);
%\coordinate (B) at (2,0);
%\coordinate (C) at (0,1);
%\coordinate (D) at (2,1);
%
%\draw (A) .. controls +(0.3,0.1) and (0.3,0.9) .. (C);
%\draw (A) .. controls +(-0.3,0.1) and (-0.3,0.9) .. (C);
%\draw (A) .. controls (0.2,-0.5) and (1.8,-0.5) .. (B);
%\draw (C) .. controls (0.2,1.5) and (1.8,1.5) .. (D);
%\draw (B) .. controls +(0.3,0.1) and (2.3,0.9) .. (D);
%\draw (B) .. controls +(-0.3,0.1) and (1.7,0.9) .. (D);
%
%\filldraw [black] (0.2,-0.2) circle (2pt);
%\end{tikzpicture}
%\caption{Reduced divisor map.}
%\end{figure}

\underline{Case 2}: The set $U$ is dense in $\Gamma$, and the degree of $D$ at $q$ is not less than the out-valence.

In this case, the the map $x \mapsto \reddiv{x}{D}$ moves one chip away from $q$ in each direction that is not $[q,q+\epsilon]$,
and the remaining chips from $q$ are moved in the direction of $[q,q+\epsilon]$.

\underline{Case 3}: The set $U$ is dense in $\Gamma$, and the degree of $D$ at $q$ is less than out-valence.

In this case, the map $x\mapsto \reddiv{x}{D}$ is constant for $x$ on the segment $[q,q + \epsilon]$, for sufficiently small $\epsilon$.

In each of the three cases, it is straightforward to verify that the family of divisors forms a geodesic path.
\end{proof}

\section{Weierstrass points}
\label{sec:weierstrass}
In this section we define the Weierstrass locus 
and the stable Weierstrass locus of an
arbitrary divisor $D$ on a metric graph $\Gamma$. 
We first review the notion of Weierstrass point on an algebraic curve.

\subsection{Classical Weierstrass points}
Recall that for an algebraic curve $X$ of genus $g$, the {\em ordinary Weierstrass points} are defined as follows.
The canonical divisor $K$ on $X$ determines a canonical map to projective space 
$\varphi_K : X \to \PP^{g - 1}$.
Generically a point on $\varphi_K(X)$ will have an osculating hyperplane in $\PP^{g - 1}$, 
which intersects $\varphi_K(X)$ with multiplicity $g - 1$.
For finitely many ``exceptional'' points on $\varphi_K(X)$, 
the osculating hyperplane will intersect the curve with higher multiplicity; 
the preimages of these exceptional points are the {ordinary Weierstrass points} of $X$.
(These  are also known as the {\em flex points} of the embedded curve $\varphi_K(X) \subset \PP^{g-1}$.)

This notion may be generalized by replacing $K$ 
with an arbitrary (basepoint-free) divisor.
Given a  divisor $D$ on $X$, 
%of degree at least $2g-1$, 
there is an associated map to projective space 
$\varphi_D : X \to \PP^r$,
known as the complete linear embedding defined by $D$.
%(If the degree of $D$ is at least $2g-1$ then $r = \deg D - g$ by Riemann--Roch.) 
The set of flex points of the embedded curve $\varphi_D(X)$, 
where the osculating hyperplane intersects the curve with multiplicity greater than $r$, 
are the {\em (higher) Weierstrass points} associated to the divisor $D$.
%Since we are mainly concerned with higher Weierstrass points, 
%we simply call these the Weierstrass points of $D$.
If $D$ has degree $\en \geq 2g - 1$,
the number of Weierstrass points of $D$ is $g(n - g + 1)^2$ when counted with appropriate multiplicities.

The existence of an osculating hyperplane of multiplicity greater than $r$, at the point $\varphi_D (x)\in \varphi_D (X)$, is equivalent to the existence of a non-zero global section of the line bundle  $\mathcal L(D - (r + 1)x)$,
i.e. to having $h^0(D - (r + 1)x) \geq 1$.

\subsection{Tropical Weierstrass points}
%In this section we define a tropical analogue of Weierstrass points.
On a metric graph, the Baker--Norine rank function $r(D)$ is the natural analogue of $h^0(D) - 1$ on an algebraic curve.
It would make sense to say, following \cite[Lemma 4.2]{Bak} and the previous discussion, that $x \in \Gamma$ is a Weierstrass point for $D$ if
\[
	r(D - (r + 1)x) \geq 0,
\]
where the second $r$ is the rank of $D$.
% However, for convenience and 
To avoid the notational confusion of two $r$'s, 
%we define the set of Weierstrass points of $D$ 
we give an equivalent definition of Weierstrass points on $\Gamma$
using the language of reduced divisors (see Section~\ref{sec:reduced-prop}).
\begin{dfn}
\label{dfn:wpoint}
Let $D$ be a divisor on a metric graph $\Gamma$, with rank $r = r(D)$.
A point $x\in \Gamma$ is a {\em Weierstrass point for $D$} %, or a {\em $D$-Weierstrass point}, 
if 
\[ 
\reddiv{x}{D} \geq (r + 1)x,
\]
where $\reddiv{x}{D}$ denotes the reduced divisor in $[D]$ at $x$.
%\[ [D - (r + 1)x] \geq 0.\]
The {\em Weierstrass locus} $W(D) \subset \Gamma$ of $D$ is the set of its Weierstrass points.
An {\em ordinary Weierstrass point}  is a Weierstrass point for the canonical divisor $K$.
\end{dfn}
Note that the Weierstrass locus of $D$ depends only on the divisor class $[D]$.

\begin{rmk}
\label{rmk:w-empty}
If the divisor class $[D]$ is not effective, i.e. $r(D) = -1$, 
then the set of Weierstrass points of $D$ is empty.
Thus we may restrict our attention to studying  
Weierstrass points for effective divisor classes. 
\end{rmk}

\begin{eg}
Suppose $\Gamma$ is a genus $1$ graph and $D$ is a divisor of degree $6$.
%indicated by the black dots in the figure below with multiplicities.
This $D$ has rank $5$ since it is in the nonspecial range of Riemann--Roch.
The Weierstrass locus of $D$ consists of $6$ points evenly spaced around the unique cycle of $\Gamma$.
\end{eg}
\begin{eg}
Suppose $\Gamma$ is 
a complete graph on $4$ vertices, with distinct edge lengths.
This graph has genus $3$.
Consider the canonical divisor $K$ on $\Gamma$, which is supported on the four trivalent vertices.
The Weierstrass locus of $K$ consists of $8$ distinct points on $\Gamma$,
shown in red in Figure~\ref{fig:gen3-canonical}.
\begin{figure}[h]
\centering
\begin{tikzpicture}[scale=0.8]
\coordinate (A) at (0,0);
\coordinate (B) at (0,-1);
\coordinate (C) at (1.5,1);
\coordinate (D) at (-1.5,1);

\draw (A) -- (B);
\draw (D) -- (A) -- (C);
\draw[rounded corners=20] (B) -- (-1.2,-1.6) -- (-2.5,0) -- (D);
\draw[rounded corners=20] (B) -- (1.2,-1.6) -- (2.5,0) -- (C);
\draw[rounded corners=25] (D) -- (-1.2,2) -- (1.2,2) -- (C);

\foreach \c in {A,B,C,D} {
	\fill (\c) circle (2pt);
}
\foreach \pos in {(0.3,0.2),(0.8,-1.3),(1.7,0.8),(1.2,1.5),(-0.3,0.2),(-0.8,-1.3),(-1.7,0.8),(-1.2,1.5)} {
	\fill[color=red] \pos circle (2pt);
	\draw[color=red] \pos circle (3pt);
}
\end{tikzpicture}
\caption{Metric graph with finite Weierstrass locus.}
\label{fig:gen3-canonical}
\end{figure}
\end{eg}

\begin{eg}[Wedge of circles]
Suppose $\Gamma$ is a wedge of $g$ circles,
and let $x_0$ denote the point of $\Gamma$ lying on all $g$ circles.
For a generic divisor class $[D_{\en}]$ of degree $\en$ 
(meaning generic inside of $\Pic^\en( \Gamma)$), 
the $x_0$-reduced representative of $[D_{\en}]$ consists of 
$\en-g$ chips at $x_0$ and one chip in the interior of each circle.
The Weierstrass locus $W(D_{\en})$
contains $\en-g+1$ evenly-spaced points on each circle of $\Gamma$, for a total of $g(\en-g+1)$ points.
\end{eg}

\begin{eg}[Failure of $W(D)$ to be finite]
\label{eg:3loop-chain}
Consider a graph consisting of a chain of three loops, shown in Figure \ref{fig:w-degenerate2} (left).
Suppose $D=K$ is the canonical divisor. 
By Riemann--Roch, $D$ has rank $r = 2$.
It is possible to move all 4 chips to lie on the middle loop, 
so any point in the middle loop 
has $\reddiv{x}{D} \geq 3x$.
The Weierstrass locus $W(D)$ contains the middle loop, but not the two outer loops.
\begin{figure}[h]
\begin{minipage}{0.45\textwidth}
\centering
\begin{tikzpicture}
	\coordinate (A) at (0,-0.4);
	\coordinate (B) at (0,0);
	\coordinate (C) at (1,0);
	\coordinate (D) at (1,0.4);
	\coordinate (E) at (2,0);
	\coordinate (F) at (2,0.4);
	\coordinate (G) at (3,-0.4);
	\coordinate (H) at (3,0);

	\draw[rounded corners=5] (B) -- (A) -- +(0,-0.3) -- +(-1,-0.3) -- +(-1,0.7) -- +(0,0.7) -- cycle;
	\draw[color=red,very thick] (B) -- (C);
	\draw[color=red,very thick,rounded corners=5] (D) -- +(0,0.3) -- +(1,0.3) -- (F) -- (E) -- +(0,-0.3) -- +(-1,-0.3) -- (C) -- cycle;
	\draw[color=red,very thick] (E) -- (H);
	\draw[rounded corners=5] (G) -- (H) -- +(0,0.3) -- +(1,0.3) -- +(1,-0.7) -- +(0,-0.7) -- cycle;

	\foreach \c in {B,C,E,H} {
		\fill[color=red] (\c) circle (2.0pt);
	}
	\foreach \pos in {(-1,0),(-0.5,-0.7),(3.5,-0.7),(4,0)} {
		\fill[color=red,opacity=0.5] \pos circle (1.6pt);
		\draw[color=red] \pos circle (2.4pt);
	}
\end{tikzpicture}
\end{minipage}
\begin{minipage}{0.05\textwidth}
	\qquad
\end{minipage}
\begin{minipage}{0.45\textwidth}
\centering
\begin{tikzpicture}
\coordinate (A) at (0,0);
\coordinate (B) at (1,0);
\coordinate (C) at (3,0);
\coordinate (D) at (0,2);
\coordinate (E) at (2,2);
\coordinate (F) at (3,2);

\draw (A) -- (B) -- (C);
\draw (D) -- (E) -- (F);
\draw (B) -- (E);
\draw (A) -- (D);
\draw (C) -- (F);
\draw[rounded corners=10] (A) -- +(-1,0) -- +(-1,2) -- (D);
\draw[rounded corners=10] (C) -- +(1,0) -- +(1,2) -- (F);

\foreach \c in {A,B,C,D,E,F} {
	\fill (\c) circle (2pt);
}

\draw[ultra thick,color=red] (1.4,0.8) -- (1.6,1.2);
	% \draw[ultra thick,color=red] (0,0) -- +(-0.3,0);
	% \draw[line width=2.5pt,color=red] (3,2) -- +(0.4,0);
	\foreach \pos in {
		(1.4,0.8),(1.6,1.2) %,(-0.3,0),(3.4,2),(0,0),(3,2)
	} {
	\fill[color=red] \pos circle (2.0pt);
	}
	\foreach \pos in {
		(-1,1),(0,1),(3,1),(4,1),(1,2),(2,0),
		(1.1,0.2),(1.9,1.8),(-0.2,2),(3.2,0),
		(-0.3,0),(0,0.2),(3.3,2),(3,1.8)
	} {
		\fill[color=red,opacity=0.5] \pos circle (1.6pt);
		\draw[color=red] \pos circle (2.4pt);
}
\end{tikzpicture}
\end{minipage}
\caption{Metric graphs with Weierstrass locus $W(D)$ not finite, where $D = K$.}
\label{fig:w-degenerate2}
\end{figure}
\end{eg}

\begin{rmk}
For any metric graph with a bridge edge,
it can be shown that 
the entire bridge edge is contained in the 
Weierstrass locus of the canonical divisor 
so in particular $W(K)$ is not finite.

The failure of $W(K)$ to be finite is not limited to graphs with bridge edges. 
The graph in Figure~\ref{fig:w-degenerate2} (right) is $2$-connected, but $W(K)$ contains an open segment in the middle of the central, diagonal edge. 
We omit the details.
\end{rmk}

%\todo{add example on non-finite weierstrass locus on metric graph which is 2-connected}

%\begin{eg}[Failure of $W(D)$ to be finite, v3]
%Suppose $\Gamma$ is a metric graph constructed by adding a bridge edge 
%connecting two disjoint metric graphs $\Gamma_1$, $\Gamma_2$
%of positive genus $g_1$, $g_2$ respectively.
%Let $K = K_\Gamma$ denote the canonical divisor on $\Gamma$; it is clear that
%\[ K = K_{\Gamma_1} + K_{\Gamma_2} + y_1 + y_2,\]
%where $y_1$ and $y_2$ are the endpoints of the bridge edge.
%
%By considering $\PL$ functions which vary on $\Gamma_1$, we have
%$K_{\Gamma_1} \sim_{\Gamma} (g_1-1)y_1 + E_1$
%for some effective divisor $E_1$, 
%and similary 
%$K_{\Gamma_2} \sim_{\Gamma} (g_2-1)y_2 + E_2$.
%Thus on $\Gamma$ there is a linear equivalence
%$$ K \sim g_1 y_1 + g_2 y_2 + E_1 + E_2 .$$
%Since chips may move freely on a bridge edge, 
%$g_1 y_1 + g_2  y_2 \sim (g_1+g_2) y$
%for any $y \in e$.
%This shows that $e$ is contained in the Weierstrass locus since 
%$$r(K) = g-1 = g_1 + g_2 -1 .$$
%\end{eg}

\subsection{Stable tropical Weierstrass points}
\label{sec:wstab}
In this section 
we define the stable Weierstrass locus $\Wstab(D)$
of a divisor $D$ on a metric graph.
This definition is meant to fix undesirable behavior of the ``naive'' 
Weierstrass locus $W(D)$.
In particular, $\Wstab(D)$ is always a finite set.

For the definition of stable reduced divisor $\streddiv{x}{D}$, see Section~\ref{sec:break-div}.

\begin{dfn}
\label{dfn:wstab}
Let $D$ be a divisor of degree $n$ on a metric graph $\Gamma$ of genus $g$.
If $n\geq g$, the {\em stable Weierstrass locus} $\Wstab(D) \subset \Gamma$ is the set of all points $x\in \Gamma$ such that
\begin{equation}
\label{eq:stable-weierstrass}
\streddiv{x}{D} \geq (n-g + 1)x ,
\end{equation}
where $\streddiv{x}{D}$ is the stable reduced divisor of $[D]$ at $x$ (see Definition~\ref{dfn:stable-reduced-div}).
%
%Let $D$ be a divisor of degree $\en$ on a metric graph $\Gamma$.
%If $\en\geq g$, the {\em stable Weierstrass locus} $\Wstab(D) \subset \Gamma$
%is the set of all points $x\in \Gamma$ such that
%\[ \br[D - (\en-g)x] \geq x \]
%where $\br[E]$ is the break divisor representative of 
%the divisor class $[E]$.
If $D$ has degree $\en < g$, we define $\Wstab(D)$ to be empty. % = \emptyset$.
\end{dfn}
In other words, 
$x$ is a {stable Weierstrass point} of $D$ if
\begin{equation*}
\text{there exists a break divisor }E \geq x \text{ such that}
\quad
E + (\en-g)x \in [D] . %\qquad \text{where }E_0 + x\text{ is a break divisor.}
\end{equation*}
Note that if $D$ has degree $\en=g$, then $\Wstab(D)$ is exactly the support of $ \br[D]$.

%\begin{rmk}
Compare the above definition to that of the Weierstrass locus $W(D)$, Definition~\ref{dfn:wpoint}.
In the definition of $\Wstab(G)$, if $\en\geq g$ then $\en-g$ is the rank of a
generic divisor class in $\Pic^\en(\Gamma)$.
If a divisor class $[D]$ in $\Pic^\en(\Gamma)$ has rank $r(D) = \en-g$,
then $\Wstab(D) \subset W(D)$;
otherwise, this containment may fail to hold.
In particular, we have $\Wstab(D) \subset W(D)$ 
for all divisors of degree $\en\geq 2g-1$.

\begin{eg}[Divisor with $\Wstab(D) \not\subset W(D)$]
Consider the genus $3$ metric graph shown  in Figure~\ref{fig:w-stable}.
The canonical divisor $K$ is indicated in black.
This divisor has degree $n=4$ and rank $r(K) = 2$.
The divisor is special, because $r(K) > n - g = 1$.
On the left side, the Weierstrass locus is shown in red;
the right side shows the stable Weierstrass locus.
The stable Weierstrass locus consists of the midpoint of each edge.
The sets $W(K)$ and $\Wstab(K)$ are disjoint.

\begin{figure}[h]
\begin{minipage}{0.45\textwidth}
\centering
\begin{tikzpicture}
\coordinate (A) at (0,0);
\coordinate (B) at (0,-1);
\coordinate (C) at (1.5,1);
\coordinate (D) at (-1.5,1);

\draw (A) -- (B);
\draw (D) -- (A) -- (C);
\draw[rounded corners=25] (B) -- (-1.2,-1.6) -- (-2.5,0) -- (D);
\draw[rounded corners=25] (B) -- (1.2,-1.6) -- (2.5,0) -- (C);
\draw[rounded corners=30] (D) -- (-1.2,2) -- (1.2,2) -- (C);

\foreach \c in {A,B,C,D} {
	\fill (\c) circle (2pt);
}
\foreach \pos in {(0.3,0.2),(0.8,-1.3),(1.7,0.8),(1.2,1.5),(-0.3,0.2),(-0.8,-1.3),(-1.7,0.8),(-1.2,1.5)} {
	\fill[color=red] \pos circle (2pt);
	\draw[color=red] \pos circle (3pt);
}
\end{tikzpicture}
\\
$W(K)$
\end{minipage}
\begin{minipage}{0.45\textwidth}
\centering
\begin{tikzpicture}
\coordinate (A) at (0,0);
\coordinate (B) at (0,-1);
\coordinate (C) at (1.5,1);
\coordinate (D) at (-1.5,1);

\draw (A) -- (B);
\draw (D) -- (A) -- (C);
\draw[rounded corners=25] (B) -- (-1.2,-1.6) -- (-2.5,0) -- (D);
\draw[rounded corners=25] (B) -- (1.2,-1.6) -- (2.5,0) -- (C);
\draw[rounded corners=30] (D) -- (-1.2,2) -- (1.2,2) -- (C);

\foreach \c in {A,B,C,D} {
	\fill (\c) circle (2pt);
}
\foreach \pos in {(0.75,0.5),(1.8,-0.8),(0,2),(0,-0.5),(-0.75,0.5),(-1.8,-0.8)} {
	\fill[color=red] \pos circle (2pt);
	\draw[color=red] \pos circle (3pt);
}
\end{tikzpicture}
\\
$\Wstab(K)$
\end{minipage}

\caption{Divisor with Weierstrass locus, left, and stable Weierstrass locus, right.}
\label{fig:w-stable}
\end{figure}
\end{eg}

\section{Finiteness of Weierstrass locus}
\label{sec:finite}
In this section we show that the Weierstrass locus 
of a generic divisor class $[D]$ on a metric graph
is a finite set whose cardinality is $\#W(D) = g(n-g+1)$.
We do so by studying  the stable Weierstrass locus
$\Wstab(D)$, defined in Section~\ref{sec:wstab}.

\subsection{Setup}
\label{sec:wstab-setup}
Our main technical tool is
to consider the ABKS decomposition of $\Pic^g(\Gamma)$
(see Section~\ref{sec:break-div})
and the topology of certain branched covering spaces.

As the divisor class $[D]$ varies over $\Pic^\en(\Gamma)$,
we realize the stable Weierstrass loci $\Wstab(D)$
as the fibers of a surjective map $X \to \Pic^\en(\Gamma)$.
We are able to study the cardinality of $\Wstab(D)$
by  imposing a nice topology on $X$ and analyzing
topological properties of the map $X\to \Pic^\en(\Gamma)$.

Recall that $\Br^g(\Gamma)$ denotes the space of break divisors on $\Gamma$,
viewed as a subspace of $\Sym^g(\Gamma)$.
\begin{dfn}
Let $\widetilde{\Br}^g(\Gamma)$ denote the space
\begin{equation*}
\widetilde\Br^g(\Gamma) = \{ (x,E) \in \Gamma\times\Sym^{g-1}(\Gamma) : 
                            x + E \text{ is a break divisor} \} .
\end{equation*}
This defines a closed subset of the compact Hausdorff space $\Gamma\times \Sym^{g-1}(\Gamma)$,
so $\widetilde\Br^g(\Gamma)$ is compact and Hausdorff.
\end{dfn}
\begin{rmk}
We may think of $\widetilde\Br^g(\Gamma)$ as the space of ``pointed break divisors''
on $\Gamma$, i.e. 
$\widetilde \Br^g(\Gamma)$ is homeomorphic to
$\{ (x,D) \in \Gamma\times \Br^g(\Gamma)
\text{ such that }x\leq D \}$.
\end{rmk}
Let $\sigma : \widetilde\Br^g(\Gamma) \to \Br^g(\Gamma)$
denote the ``summation'' map $(x,E)\mapsto x+E$,
and let $\sigma_m : \widetilde \Br^g(\Gamma) \to \Pic^{m+g-1}(\Gamma)$
denote the ``summation with multiplicity'' map defined by 
\begin{equation*}
\sigma_m : (x,E) \mapsto [mx + E] .
\end{equation*}
Let $\pi_1 : \widetilde\Br^g(\Gamma) \to \Gamma$
 denote projection to the first factor, 
i.e. $\pi_1(x,E) = x$.
\begin{lem}
\label{lem:wstab}
Suppose $[D] \in \Pic^{m + g - 1}(\Gamma)$,
and let $\sigma_m$ and $\pi_1$ be defined as above.
\begin{enumerate}[(a)]
\item 
The stable Weierstrass locus $\Wstab(D)$ is equal to
$\pi_1(\sigma_m^{-1}[D])$.

\item The cardinality $\#\Wstab(D) = \# \sigma_m^{-1}[D]$.
\end{enumerate}
\end{lem}
\begin{proof}
(a) This follows from the definition of the stable Weierstrass locus.
%for a divisor $D$ of degree $m+g-1$ we have
%$\Wstab(D) = \pi_1 ( \sigma_m^{-1}[D] )$.

(b) The claim is that $\pi_1$ is injective on the preimage $\sigma_m^{-1} [D] $.
To see this, consider two points $(x, E)$ and $(x', E') \in \widetilde{\Br}^g(\Gamma)$ 
in the same fiber $\sigma_m^{-1}[D]$.
This means that $[mx + E] = [mx' + E'] = [D]$.
Suppose $\pi_1(x,E) = \pi_1(x',E')$, i.e. that $x = x'$.
Then
\[ [D - (m-1)x] = [x + E] = [x + E'] \in \Pic^g(\Gamma). \]
Since both $(x+E)$ and $(x+E')$ are break divisors,
the uniqueness of break divisor representatives (Theorem~\ref{thm:break-div})
implies that $E = E'$.
This shows that the restriction of $\pi_1$  to $\sigma_m^{-1}[D]$
is injective, as desired.
\end{proof}

Let $(G,\ell)$ be a combinatorial model for $\Gamma$,
which induces a decomposition of break divisors $\Br^g(\Gamma)$
into a union  of cells 
\begin{equation}
\label{eq:abks}
\Br^g(\Gamma) = \bigcup_{T\in \mathcal T(G)} C_T
\end{equation}
indexed by spanning trees of $G$,
where the interior of each cell $C_T$ is homeomorphic to an open hypercube.
(See Section~\ref{sec:break-div} or  \cite{ABKS}.)
Note that $\Br^g(\Gamma)$ is homeomorphic to $\Pic^g(\Gamma)$.
The ABKS decomposition \eqref{eq:abks} of $\Br^g(\Gamma)$
induces a decomposition 
\begin{equation}
\label{eq:abks-pointed}
\widetilde\Br^g(\Gamma) = \bigcup_{T\in \mathcal T(G)}
\left( \bigcup_{e \not\in E(T)}\widetilde C_{T,e} \right)
\end{equation}
where the second union is over edges $e$ of $G$ not contained in 
the spanning tree $T$.
There are $g$ such edges for any $T$.
Namely,
\[ \widetilde C_{T,e} = \{ (x,E) \in \widetilde \Br^g(\Gamma) : x + E \in C_T,\, x \in e \}\]
The map $\widetilde \Br^g(\Gamma) \to \Br^g(\Gamma)$
sends the cell $\widetilde C_{T,e}$ surjectively to $C_T$.
On the interior $C_T^\circ$ of each cell,
each fiber of  $\widetilde \Br^g(\Gamma) \to \Br^g(\Gamma)$ contains exactly $g$ points.

If $\kappa(G) = \#\mathcal T(G)$ denotes the number of spanning trees of $G$,
 the ABKS decomposition \eqref{eq:abks-pointed}
decomposes $\widetilde \Br^g(\Gamma)$
into a union of $g \cdot \kappa(G)$ cells.

\begin{eg}
In Figure~\ref{fig:abks-pointed}, we show the decomposition of $\widetilde\Br^2(\Gamma)$
into six cells $\widetilde C_{T,e}$,
where $\Gamma$ is a theta graph.
This graph has genus $g = 2$
and $\kappa(G) = 3$ spanning trees.
In this case 
$\Br^2(\Gamma) \cong \Pic^2(\Gamma) \cong \RR^2 / \ZZ^2$ is a genus $1$ surface
(cf. Example~\ref{eg:abks}, Theorem~\ref{thm:abel-jacobi}),
and
$\widetilde \Br^2(\Gamma)$ is a surface of genus $2$.
The map $\widetilde\Br^2(\Gamma) \to \Br^2(\Gamma)$
is a branched double cover ramified at two points,
corresponding to the two break divisors which consist of 
two chips at a trivalent vertex of $\Gamma$. %nontrivial multiplicity.

\begin{figure}[h]
\centering
\begin{tikzpicture}[scale=1.0]
	\coordinate (O) at (0,0);
	\coordinate (Q) at (-1.3,1);
	\coordinate (R) at (1.3,2);
	\coordinate (S) at (0,3);
	\coordinate (Qd) at (1.3,-1);
	\coordinate (Rd) at (-1.3,-2);
	\coordinate (Sd) at (0,-3);
	\coordinate (Ol) at ($(O) + (-4,0)$);
	\coordinate (Ql) at ($(Q) + (-4,0)$);
	\coordinate (Sl) at ($(S) + (-4,0)$);
	\coordinate (Rdl) at ($(Rd) + (-4,0)$);
	\coordinate (Sdl) at ($(Sd) + (-4,0)$);
	\coordinate (Or) at ($(O) + (4,0)$);
	\coordinate (Rr) at ($(R) + (4,0)$);
	\coordinate (Sr) at ($(S) + (4,0)$);
	\coordinate (Qdr) at ($(Qd) + (4,0)$);
	\coordinate (Sdr) at ($(Sd) + (4,0)$);

	\draw (O) -- (R) -- (S) -- (Q) -- cycle;
	\draw (O) -- (Rd) -- (Sd) -- (Qd) -- cycle;
	\draw (Q) -- (Ql) -- (Ol) -- (Rdl) -- (Rd);
	\draw (R) -- (Rr) -- (Or) -- (Qdr) -- (Qd);
	\draw (Ol) -- (Or);
	% thick edges
	\begin{scope}[line width=2.0pt]% [ultra thick]
		\draw (Rdl) -- (Rd) -- (Sd);
		\draw (Ql) -- (Q) -- (S);
		\draw (Rr) -- (Or) -- (Qdr);
		\draw (Ol) -- (O) -- (R);
		\draw (O) -- (Qd);
		\end{scope}
	\begin{scope}[decoration={
		markings,
		mark=at position 0.5 with {\arrow{Latex[scale width=2]}}
	}]
		\draw[postaction={decorate}] (Or) -- (Rr);
		\draw[postaction={decorate}] (Q) -- (S);
	\end{scope}
	\begin{scope}[decoration={
		markings,
		mark=at position 0.6 with {\arrow{Latex[open,scale=2]}}
	}]
		\draw[postaction={decorate}] (Sd) -- (Rd);
		\draw[postaction={decorate}] (Qdr) -- (Or);
	\end{scope}
	\begin{scope}[decoration={
		markings,
		mark=at position 0.5 with {\arrow{>[scale=1.5]>[scale=1.5]}}
	}]
		\draw[postaction={decorate}] (Ql) -- (Q);
		\draw[postaction={decorate}] (Rdl) -- (Rd);
	\end{scope}
	\begin{scope}[decoration={
		markings,
		mark=at position 0.5 with {\arrow{>[scale=1.2]}}
	}]
		\draw[postaction={decorate}] (R) -- (Rr);
		\draw[postaction={decorate}] (Qd) -- (Qdr);
	\end{scope}
	\begin{scope}[decoration={
		markings,
		mark=at position 0.5 with {\arrow{Latex[scale width=1.5]}}
	}]
		\draw[postaction={decorate}] (Rdl) -- (Ol);
		\draw[postaction={decorate}] (Sd) -- (Qd);
	\end{scope}
	\begin{scope}[decoration={
		markings,
		mark=at position 0.5 with {\arrow{Latex[open,scale width=1.5]}}
	}]
		\draw[postaction={decorate}] (Ol) -- (Ql);
		\draw[postaction={decorate}] (R) -- (S);
	\end{scope}

	\begin{scope}[shift={(-0.1,1.1)}, scale=0.4]
		\coordinate (A) at (0,0);
		\coordinate (C) at (0,2.0);
		\coordinate (E1) at (-1,1);
		\coordinate (E2) at (0,1);
		\coordinate (E3) at (1.5,1);
	
		\draw (A) -- (C);
		\draw[rounded corners=6] (A) -- ++(-1,0) -- ++(0,2.0) -- (C);
		\draw[rounded corners=6] (A) -- ++(1.5,0) -- ++(0,2.0) -- (C);
	
		\foreach \P in {E1,E2} {
			\fill (\P) circle (4pt);
		}
		\draw (E2) circle (6pt);
	\end{scope}
	\begin{scope}[shift={(2.2,0.5)}, scale=0.4]
		\coordinate (A) at (0,0);
		\coordinate (C) at (0,2.0);
		\coordinate (E1) at (-1,1);
		\coordinate (E2) at (0,1);
		\coordinate (E3) at (1.5,1);
	
		\draw (A) -- (C);
		\draw[rounded corners=6] (A) -- ++(-1,0) -- ++(0,2.0) -- (C);
		\draw[rounded corners=6] (A) -- ++(1.5,0) -- ++(0,2.0) -- (C);
	
		\foreach \P in {E1,E3} {
			\fill (\P) circle (4pt);
		}
		\draw (E3) circle (6pt);
	\end{scope}
	\begin{scope}[shift={(2.5,-0.9)}, scale=0.4]
		\coordinate (A) at (0,0);
		\coordinate (C) at (0,2.0);
		\coordinate (E1) at (-1,1);
		\coordinate (E2) at (0,1);
		\coordinate (E3) at (1.5,1);
	
		\draw (A) -- (C);
		\draw[rounded corners=6] (A) -- ++(-1,0) -- ++(0,2.0) -- (C);
		\draw[rounded corners=6] (A) -- ++(1.5,0) -- ++(0,2.0) -- (C);
	
		\foreach \P in {E2,E3} {
			\fill (\P) circle (4pt);
		}
		\draw (E3) circle (6pt);
	\end{scope}
	\begin{scope}[shift={(-2.5,0.1)}, scale=0.4]
		\coordinate (A) at (0,0);
		\coordinate (C) at (0,2.0);
		\coordinate (E1) at (-1,1);
		\coordinate (E2) at (0,1);
		\coordinate (E3) at (1.5,1);
	
		\draw (A) -- (C);
		\draw[rounded corners=6] (A) -- ++(-1,0) -- ++(0,2.0) -- (C);
		\draw[rounded corners=6] (A) -- ++(1.5,0) -- ++(0,2.0) -- (C);
	
		\foreach \P in {E2,E3} {
			\fill (\P) circle (4pt);
		}
		\draw (E2) circle (6pt);
	\end{scope}
	\begin{scope}[shift={(-2.8,-1.4)}, scale=0.4]
		\coordinate (A) at (0,0);
		\coordinate (C) at (0,2.0);
		\coordinate (E1) at (-1,1);
		\coordinate (E2) at (0,1);
		\coordinate (E3) at (1.5,1);
	
		\draw (A) -- (C);
		\draw[rounded corners=6] (A) -- ++(-1,0) -- ++(0,2.0) -- (C);
		\draw[rounded corners=6] (A) -- ++(1.5,0) -- ++(0,2.0) -- (C);
	
		\foreach \P in {E1,E3} {
			\fill (\P) circle (4pt);
		}
		\draw (E1) circle (6pt);
	\end{scope}
	\begin{scope}[shift={(-0.1,-1.8)}, scale=0.4]
		\coordinate (A) at (0,0);
		\coordinate (C) at (0,2.0);
		\coordinate (E1) at (-1,1);
		\coordinate (E2) at (0,1);
		\coordinate (E3) at (1.5,1);
	
		\draw (A) -- (C);
		\draw[rounded corners=6] (A) -- ++(-1,0) -- ++(0,2.0) -- (C);
		\draw[rounded corners=6] (A) -- ++(1.5,0) -- ++(0,2.0) -- (C);
	
		\foreach \P in {E1,E2} {
			\fill (\P) circle (4pt);
		}
		\draw (E1) circle (6pt);
	\end{scope}
\end{tikzpicture}
\caption{ABKS decomposition of $\widetilde\Br^2(\Gamma)$, for the metric graph in Figure~\ref{fig:abks}.}
\label{fig:abks-pointed}
\end{figure}

In Figure~\ref{fig:abks-pointed},
each cell $\widetilde C_{T,e}$ shows a representative break divisor $x+E$
where the point $x\in e$ %in the break divisor $x + E$ 
is marked with an extra outline.
Boundary edges of $\widetilde C_{T,e}$
which have $x$ on an endpoint of $e$ are marked in bold.
Edges on the boundary are glued to the parallel boundary edge which 
has the same weighting (bold or unbold).
\end{eg}

\subsection{Point-set topology} 

\begin{dfn}
\label{dfn:branched-cover}
Let $M$ and $N$ be compact Hausdorff spaces,
and let $N$ be path-connected.
We say $p: M \to N$ is a {\em branched covering map} if 
\begin{enumerate}[(i)]
\item 
\label{it:bc-i}
$p$ is continuous and surjective
\item \label{it:bc-ii}
$p$ is an open map (the image of an open set is open)
\item \label{it:bc-iii}
$p^{-1}(y)$ is finite for each $y\in N$ 
\end{enumerate}
and there exists a closed subset $R\subset N$ 
such that 
\begin{enumerate}[(i)]
\setcounter{enumi}{3}
\item \label{it:bc-iv}
 $N\backslash R$ is path-connected
\item
\label{it:bc-v}
$R$ has empty interior in $N$
\item
\label{it:bc-vi}
the restriction of $p$ to 
$M \backslash p^{-1}(R) \to N\backslash R$
is a topological covering map.
\end{enumerate}
The subspace $R$ is a {\em ramification locus} of $p$, %(if chosen minimally?),
and the preimage $p^{-1}(R)$ is a {\em branch locus}.
(Note that properties (ii) and (v) imply $p^{-1}(R)$
has empty interior in $M$.)
\end{dfn}
%\harry{conditions (iv) and (v) imply $N$ is connected; 
%condition (ii) implies $p$ is surjective}
It is straightforward to verify that the map $\widetilde \Br^g(\Gamma) \to \Br^g(\Gamma)$
from Section~\ref{sec:wstab-setup} is a branched covering.
We show below, in Proposition~\ref{prop:w-count}, that in fact 
each $\sigma_m : \widetilde\Br^g(\Gamma) \to\Pic^{m+g-1}(\Gamma)$,
for $m\geq 1$, is a branched covering.

Recall that a map is {\em proper} if the preimage of a compact set is compact.
Recall that a map $f : X \to Y$ is a {\em local homeomorphism}
if, for any $x \in X$ there is an open neighborhood $U$ containing $ x$
such that $f(U)$ is open in $Y$ and the restriction $U \to f(U)$
is a homeomorphism.
A covering map is always a local homeomorphism, but the converse is not true.

The following lemma will be used to check the last condition \eqref{it:bc-vi}
in Definition~\ref{dfn:branched-cover},
that the restriction $M \backslash p^{-1}(R) \to N\backslash R$
is a covering map.
\begin{lem}
\label{lem:cover-local}
Suppose $p: X \to Y$ is a local homeomorphism
between locally compact, Hausdorff spaces.
If $p$ is proper and surjective,
then $p$ is a covering map.
\end{lem}
This is a standard exercise in point-set topology;
see e.g. \cite[Lemma 2]{Ho}. 

\begin{lem}
\label{lem:cover-deg}
Suppose $p : M\to N$ is a branched covering
with ramification locus $R \subset N$ such that the restriction 
$p : M \backslash p^{-1}(R) \to N\backslash R$
is a covering map of degree $d$.
Then for any $y \in N$, the preimage $p^{-1}(y)$
has cardinality at most $d$.
\end{lem}
Note: the restriction of $p$ to $M \backslash p^{-1}(R) \to N\backslash R$
has constant degree $d$ because 
in the definition of branched cover, $N\backslash R$ is assumed to be path connected.
\begin{proof}
Let $y\in R$ be a point in the ramification locus,
and let $x_1, \ldots, x_k$ be the points in the preimage $p^{-1}(y)$.
Since $M$ is Hausdorff, we may choose open neighborhoods $U_1, \ldots, U_k$ with $x_i\in U_i$
which are disjoint, $U_i\cap U_j = \emptyset$.
Let $C = M \backslash (U_1 \cup \cdots \cup U_k)$ be the complement of these neighborhoods,
which is closed in $M$.
Since $M$ is compact and $N$ is Hausdorff, the image $p(C)$ is closed in $N$. 
Thus $V = N \backslash p(C)$ is open and nonempty since $y\in V$.
Note that by construction 
$p^{-1}(V)  = M \backslash p^{-1}( p(C))\subset M \backslash C = U_1 \cup \cdots \cup U_k$.

Let $U'_i$ be the intersection of $p^{-1}(V)$ with $U_i$,
which is open and nonempty because $x_i\in U_i'$.
Since the $U_i$ were chosen to be disjoint, 
$p^{-1}(V) = U_1' \sqcup \cdots \sqcup U_k'$.

Note that $p$ is an open map (by definition of branched cover), so the intersection
$p(U_1') \cap \cdots \cap p(U_k')$ is an open neighborhood of $y$ in $N$.
Since $R$ has empty interior in $N$, 
% \harry{or $R$ is nowhere dense}, 
we can choose some point 
\begin{equation*}
z \in \left(p(U_1') \cap \cdots \cap p(U_k')\right) \backslash R  \subset V \backslash R.
\end{equation*}
By the assumption that $M\backslash p^{-1}(R) \to N\backslash R$ is a degree $d$ covering map,
the preimage $p^{-1}(z)$ contains $d$ points 
$w_1, \ldots, w_d$.
Since $z\in V$ by construction,
each $w_i \in p^{-1}(V) = U_1' \sqcup \cdots \sqcup U_k'$
so $w_i$ lies within $U_j'$ for some unique $j \in \{ 1,\ldots, k\}$.
This relation defines a map  $\pi: \{1,\ldots d\} \to \{1, \ldots, k\}$.
Moreover, the map $\pi$ is surjective because  $z\in p(U_j')$
for each $j \in \{1,\ldots,k\}$.
This proves that $k \leq d$, so
the preimage $p^{-1}(y)$ has cardinality at most $d$ as desired.
\end{proof}

\subsection{Proofs}
\begin{prop}
\label{prop:w-finite}
For any divisor $D$, the stable Weierstrass locus $\Wstab(D)$
is a finite subset of $\Gamma$.
\end{prop}
\begin{proof}
If $D$ has degree $n < g$, 
the stable Weierstrass locus is defined to be empty.
Thus we assume below that $D$ has degree $n \geq g$.

Recall that 
$\widetilde \Br^g(\Gamma) = \{(x,E) \in \Gamma \times \Sym^{g-1}(\Gamma) 
 : \text{ $x+E$ is a break divisor} \}$
and that
$\sigma_m : \widetilde \Br^g(\Gamma) \to \Pic^{m+g-1}(\Gamma)$
is defined by 
\begin{equation*}
\sigma_m : (x,E) \mapsto [mx + E] .
\end{equation*}
Recall that $\pi_1$ % : \widetilde\Br^g(\Gamma) \to \Gamma$
denotes the projection  $\pi_1(x,E) = x$.
(See Section~\ref{sec:wstab-setup}.)
By Lemma~\ref{lem:wstab},
for a divisor $D$ of degree $m+g-1$ we have
$\Wstab(D) = \pi_1 ( \sigma_m^{-1}[D] )$.
Hence it suffices to show that the preimage
$\sigma_m^{-1}[D]$ is a finite set.

Let $(G,\ell)$ be a combinatorial model for $\Gamma$,
which induces the ABKS decomposition
$
\Br^g(\Gamma) = \bigcup_{T\in \mathcal T(G)} C_T
$,
where the cells $C_T$  are indexed by spanning trees of $G$.
The ABKS decomposition of $\Br^g(\Gamma)$
induces a decomposition 
\begin{equation*}
\widetilde\Br^g(\Gamma) = \bigcup_{T\in \mathcal T(G)}
\left( \bigcup_{e \not\in E(T)}\widetilde C_{T,e} \right) .
\end{equation*}
%The map $\widetilde \Br^g(\Gamma) \to \Br^g(\Gamma)$
%sends the cell $\widetilde C_{T,e}$ surjectively to $C_T$.
Let $\sigma_m^{(T,e)} : \widetilde C_{T,e} \to \Pic^{m+g-1}(\Gamma)$ 
denote the restriction of $\sigma_m$ to $\widetilde C_{T,e}$.

%\harry{make this a proposition / lemma}

\ul{Claim}: The preimage of $[D]$ under 
$\sigma_m^{(T,e)} : \widetilde C_{T,e} \to \Pic^{m+g-1}(\Gamma)$
is finite.

\noindent This Claim implies that the preimage $\sigma_m^{-1}[D]$
is a finite set, 
since $\widetilde \Br^g(\Gamma)$ is covered by
finitely many $\widetilde C_{T,e}$.

\textit{Proof of Claim}:
The map 
$\sigma_m^{T,e} : \widetilde C_{T,e} \to \Pic^{m+g-1}(\Gamma)$
is locally defined by a linear map, 
which we show is full rank.
For a spanning tree $T = G \backslash \{e, e_2, \ldots, e_g\}$,
there is a natural surjective parametrization 
$ \prod_{i=1}^g [0, \ell(e_i)] \to \widetilde C_{T,e}$.

Let $f_m^{T,e}$ denote the lift of 
 $\prod_{i=1}^g [0,\ell(e_i)]\to \widetilde C_{T,e} \xrightarrow{} \Pic^{m+g-1}(\Gamma)$ 
  to the universal cover 
$\RR^g  \to \Pic^{m+g-1}(\Gamma)$.
\[ \begin{tikzcd}
 \prod_{i=1}^g [0,\ell(e_i)] \ar[d] \ar[r,dashed,"f_m^{T,e}"] & \RR^g \ar[d,"\pi"] \\
\widetilde C_{T,e} \ar[r,"\sigma_m^{T,e}"]  & \Pic^{m+g-1}(\Gamma)
\end{tikzcd} \]
When $m = 1$,
coordinates may be chosen on $\RR^g$ such that $f_1^{T,e}$
is represented by the identity matrix.
Using these same coordinates on $\RR^g$
(up to a translation from $\Pic^g$ to $\Pic^{m+g-1}$), for $m \geq 1$
the definition $\sigma_m(x,E) = [mx + E]$
implies that $f^m_{T,e}$
is represented by the diagonal matrix
\begin{equation*}
% \text{the matrix representing $f^m_{T,e}$ is} \quad
 \begin{pmatrix}
m & & \\
 & 1 & \\
 & & \ddots \\
 & & & 1
\end{pmatrix} .
\end{equation*}
This shows that $f_m^{T,e}$ is locally injective, 
which implies $\sigma_m^{T,e}$ is locally injective as well.
Thus for any $[D] \in \Pic^{m+g-1}(\Gamma)$,
the preimage under $\sigma_m^{T,e}$  is a discrete subset of $\widetilde C_{T,e}$.
Since $\widetilde C_{T,e}$ is compact, 
the preimage of $[D]$ is finite as claimed.
\end{proof}

In the following proposition, ``generic''
means the statement holds for $[D] \in \Pic^n(\Gamma)$
outside of a nowhere-dense exceptional set.
\begin{prop}
\label{prop:w-count}
For any divisor class $[D]$ of degree $\en\geq g$,
we have
\[ \#\Wstab(D) \leq g(\en - g + 1) .\]
For a generic divisor class $[D]$ of degree $\en\geq g$,
the stable Weierstrass locus $\Wstab(D)$ has cardinality
$ \#\Wstab(D) = g(\en - g + 1) $.
\end{prop}
\begin{proof}
Let $\widetilde{\Br}^g(\Gamma)$,
 $\sigma_m : \widetilde \Br^g(\Gamma) \to \Pic^{m+g-1}(\Gamma)$,
 and $\pi_1 : \widetilde \Br^g(\Gamma) \to \Gamma$
be defined as in Section~\ref{sec:wstab-setup}.
Recall that for a divisor $D$ of degree $m+g-1$, we have
$ \# \Wstab(D) = \#( \sigma_m^{-1}[D])$
by Lemma~\ref{lem:wstab}.
Thus it suffices to show that 
$\sigma_m: \widetilde{\Br}^g(\Gamma) \to \Pic^{m+g-1}(\Gamma)$
is a branched covering map of degree $gm$, for any $m\geq 1$.
From this, 
Lemma~\ref{lem:cover-deg} implies the inequality $\#\Wstab(D) \leq gm$
and
Definition~\ref{dfn:branched-cover}
implies that equality holds for $[D]$ outside of the ramification locus.

(If $D$ has degree $\en = m+g-1$, then $gm = g(\en-g+1)$.)

\ul{Claim 1}: 
The map $\sigma_m: \widetilde\Br^g(\Gamma) \to \Pic^{m+g-1}(\Gamma)$ 
is open, for any $m\geq 1$.

\textit{Proof of Claim 1}: 
As above, let $(G,\ell)$ be a combinatorial model for $\Gamma$,
and
\begin{equation*}
\widetilde \Br^g(\Gamma) = \bigcup_{T\in \mathcal T(G)} \bigcup_{e\not \in E(T)} \widetilde C_{T,e}
\end{equation*}
the
 induced ABKS decomposition.
(See Section~\ref{sec:wstab-setup}.)
The map $\sigma_m$ is naturally a piecewise affine map
with domains of linearity $\widetilde C_{T,e}$. 
% \subset \widetilde \Br^g(\Gamma)$.
%defined in Section~\ref{sec:wstab-setup}.

To show that $\sigma_m$ is open,
it suffices to check that for any $(x_0, E_0) \in \widetilde\Br^g(\Gamma)$,
the image of a neighborhood contains points in all tangent directions around 
$\sigma_m(x_0,E_0) \in \Pic^{m+g-1}(\Gamma)$. 
% (which is locally homeomorphic to $\RR^g$).
To check this, we observe how $\sigma_m$ restricts 
to each domain of linearity $\widetilde C_{T,e}$
containing $(x_0, E_0)$.
We will show that the behavior of $\sigma_m$ on tangent directions
does not depend on the integer $m$.

For a point $(x_0, E_0)$ in $\widetilde C_{T,e}$,
let $\cone( \sigma_m^{T,e}(x_0,E_0) )$ denote the positive cone in $\RR^g$ 
spanned by 
\[ \sigma_m(x,E) - \sigma_m(x_0,E_0) \quad 
 \text{for $(x,E)$ in a neighborhood of $(x_0, E_0)$ in $\widetilde C_{T,e}$}.\]
(Here we identify $\RR^g$ with the tangent space of $\Pic^{0}(\Gamma)$  at the identity.)
Since $\sigma_m$ %($ = \sigma_m^{T,e}$) 
is affine on $\widetilde C_{T,e}$,
this cone does not depend on the neighborhood chosen.
Since $m\geq 1$,
the positive span of 
\[ \sigma_m(x,E) - \sigma_m(x_0,E_0)  = m[x - x_0] + [E - E_0]
\quad \text{for $(x,E)$ in } \widetilde C_{T,e}\]
is equal to the positive span of
\[ \sigma_1(x+E) - \sigma_1(x_0 + E_0) = [x-x_0] + [E - E_0]
\quad \text{for $(x, E)$ in } \widetilde C_{T,e} ,\]
so $\cone( \sigma_m^{T,e}(x_0,E_0) ) = \cone( \sigma_1^{T,e}(x_0,E_0) )$.
This holds for all cells  $\widetilde C_{(T,e)}$ containing $(x_0,E_0)$.

Hence to show that $\sigma_m$ is open,
it suffices to show that 
$\sigma_1 : \widetilde\Br^g(\Gamma) \to \Pic^g(\Gamma)$
is open.
This is clear from the construction of $\widetilde\Br^g(\Gamma)$
as a branched cover $\widetilde \Br^g(\Gamma) \to \Br^g(\Gamma)$,
and from Theorem~\ref{thm:break-div}
which states that $\Br^g(\Gamma) \to \Pic^g(\Gamma)$ is a homeomorphism.

\ul{Claim 2}: 
The map $\sigma_m: \widetilde\Br^g(\Gamma) \to \Pic^{m+g-1}(\Gamma)$ 
is a branched cover, for any $m\geq 1$.

\textit{Proof of Claim 2}: 
In the definition of branched cover, Definition~\ref{dfn:branched-cover},
condition \eqref{it:bc-ii} was verified by Claim 1
and condition \eqref{it:bc-iii} was verified by Proposition~\ref{prop:w-finite}.
Condition \eqref{it:bc-i} holds ($\sigma_m$ is surjective)
because $\sigma_m$ is an open map from a compact space to a connected, Hausdorff space.

We first identify a ramification locus $R$ for $\sigma_m$,
and then apply Lemma~\ref{lem:cover-local}
to show that the restriction of $\sigma_m$ 
away from $R$
%to 
%$$\sigma_m^\circ : \widetilde\Br^g(\Gamma) \backslash \sigma_m^{-1}(R) \to \Pic^{m+g-1}(\Gamma) \backslash R $$
is a covering map.

Let
$
\Br^g(\Gamma) = \bigcup_{T\in \mathcal T(G)} C_T
$
be the ABKS decomposition 
induced by a combinatorial model $\Gamma = (G,\ell)$
%indexed by spanning trees of $G$
%where the interior of each cell $C_T$ is homeomorphic to an open hypercube.
(see Section~\ref{sec:break-div}).
Let $\Ztwo \subset \Br^g(\Gamma)$ denote 
the union of faces of $C_T$ of codimension at least $2$,
and let $\Utwo = \Br^g(\Gamma) \backslash \Ztwo$.
In other words,
\begin{equation*}
\Utwo = \bigcup_{T\in \mathcal T(G)} \{\text{interior $C_T^\circ$ of $C_T$}\}
  \cup\{\text{interiors of facets of $\partial C_T$} \} .
\end{equation*} 
More concretely in terms of break divisors,
given a set of edges $e_1, \ldots, e_g$ in $G$ whose complement is a spanning tree,
$\Utwo$ contains break divisors which are a sum of $g$ points 
taken from the interior of each $e_1,e_2,\ldots,e_g$,
and divisors which are a sum of one endpoint of $e_1$ and a point in the interior of each $e_2,\ldots,e_g$.
We assume our  combinatorial model $(G,\ell)$ is chosen to have no loops, 
so that each cell $ C_T$ in the ABKS decomposition has $2g$ distinct boundary facets.

Note that for a break divisor $E$, 
\begin{equation}
\label{eq:2}
\text{if } E\in \Utwo,\text{ the support of $E$ consists of $g$ distinct points}.
\end{equation}

We let $\widetilde \Ztwo$ and $\widetilde \Utwo$
denote the preimages of $\Ztwo$ and $\Utwo$
under $\sigma: \widetilde\Br^g(\Gamma) \to \Br^g(\Gamma)$.
Note that with respect to the ABKS decomposition 
$$\widetilde \Br^g(\Gamma) = \bigcup_{T\in \mathcal T(G)} \bigcup_{e \not\in E(T)} \widetilde C_{T,e},$$
$\widetilde \Ztwo$
is the union of codimension $2$ faces of $\widetilde C_{T,e}$,
and $\widetilde \Utwo = \widetilde\Br^g(\Gamma) \backslash \widetilde \Ztwo$.
Thus $\widetilde \Ztwo$ is a closed subset of codimension $2$ 
and $\widetilde \Utwo$ is a dense open subset of $\widetilde \Br^g(\Gamma)$.

Next, let $R = R_m = \sigma_m(\widetilde \Ztwo)$.
We will show that $R$ is a valid ramification locus for the branched cover $\sigma_m$.
The conditions \eqref{it:bc-iv} and \eqref{it:bc-v} hold
because $R$ is a codimension $2$ submanifold of the connected manifold $\Pic^{m+g-1}(\Gamma)$.
It remains to check condition \eqref{it:bc-vi},
that the restriction 
\begin{equation}
\label{eq:m-cover}
 \sigma_m|_{\widetilde \Br^g(\Gamma) \backslash \sigma_m^{-1}(R)} : 
\widetilde \Br^g(\Gamma) \backslash \sigma_m^{-1}(R) \to \Pic^{m+g-1}(\Gamma) \backslash R
\end{equation}
away from ramification is a covering map.
To check this condition, we apply Lemma~\ref{lem:cover-local}.
It is clear that the domain and codomain of \eqref{eq:m-cover}
are locally compact Hausdorff spaces.\footnote{
The domain is locally compact and Hausdorff because it is an open subspace of 
$\widetilde\Br^g(\Gamma)$ which
is a finite CW complex, hence compact and Hausdorff.
The same holds for the codomain,
as an open subspace of  $\Pic^{m+g-1}(\Gamma)\cong \RR^g/\ZZ^g$.}
The map in \eqref{eq:m-cover} is surjective by construction;
it is proper because $\sigma_m$ is a map from a compact space
to a Hausdorff space, hence proper.
%To satisfy the hypotheses of Lemma~\ref{lem:cover-local},
It remains to check that \eqref{eq:m-cover} is a local homeomorphism,
which we leave for the next claim.
Note that the domain of \eqref{eq:m-cover} is contained in $\widetilde\Utwo$:
\begin{equation*}
\widetilde\Br^g(\Gamma) \backslash \sigma_m^{-1}(R)
= \widetilde\Br^g(\Gamma) \backslash \sigma_m^{-1}( \sigma_m(\widetilde \Ztwo) )
\subset \widetilde\Br^g(\Gamma) \backslash \Ztwo
= \widetilde \Utwo .
\end{equation*}
Assuming Claim 3, 
Lemma~\ref{lem:cover-local} implies that 
$\sigma_m$ is a covering map 
away from the ramification locus $R$, which completes the proof of Claim 2.

\ul{Claim 3}:
The restriction of $\sigma_m$ to $\widetilde \Utwo \to \Pic^{m+g-1}(\Gamma)$
is a local homeomorphism, for any $m\geq 1$.

\textit{Proof of Claim 3}:
First consider $m = 1$.
Observation \eqref{eq:2} implies that 
\begin{equation}
\label{eq:3}
\text{the restriction $\sigma_1|_{\widetilde \Utwo}  : \widetilde \Utwo \to \Utwo$
is a (unbranched) covering of degree $g$.}
\end{equation}
Since $\Utwo \subset \Pic^g(\Gamma)$ is open,
it follows that $\sigma_1: \widetilde \Utwo \to \Pic^g(\Gamma)$ is a local homeomorphism.

Recall that $\widetilde \Utwo$ is the union of the interior of $\widetilde C_{T,e}$
and the interiors of facets of $\partial \widetilde C_{T,e}$, over all $(T,e)$.
In the interior of $\widetilde C_{T,e}$, 
$\sigma_m$ can be expressed as a full-rank linear map so it is a local homeomorphism.
Now consider how $\sigma_m$ acts near the interior of a facet of $\partial \widetilde C_{T,e}$.
We claim that each facet is shared by exactly two cells.

Suppose $T = G \backslash \{ e = e_1, e_2, \ldots, e_g\}$.
There are $2g$ facets of the boundary $\partial \widetilde C_{T,e}$,
indexed by choosing an edge $e_j$
and choosing one of its two endpoints.
For a fixed index $j$  in $\{1,\ldots,g\}$
and $v(e_j)$  a fixed endpoint of $e_j$,
 the corresponding facet of $\partial \widetilde C_{T,e}$
consists of pairs $(x,E) \in \widetilde \Br^g(\Gamma)$
of the form
\begin{equation}
\label{eq:CT-facet}
\widetilde F_{(T,e)}^{(j,v)} =  \{ (x = x_1,E = x_2 + \cdots + x_g) : 
\begin{array}{l}
x_j = v(e_j), \\
 x_i \in e_i^\circ \text{ for }i = 1,\ldots g,\, i\neq j \} . 
 \end{array}
\end{equation}

Let $G_j = T \cup e_j$. 
% denote the union of the spanning tree $T$ with the edge $e_j$.
Since $e_j \not \in T$, the graph $G_j$ contains a unique cycle, which must contain
$v(e_j) \in e_j$.
Let $e_j'$ be the unique edge $\neq e_j$ in this cycle which also borders $v(e_j)$,
and let $T' = G_j \backslash e_j' = (T \cup e_j) \backslash e_j'$.
Then $\widetilde C_{T',e'}$ is the only other cell containing the facet \eqref{eq:CT-facet},
where $e' = e_1'$ if $j = 1$, and $e' = e$ otherwise.
The facet \eqref{eq:CT-facet} is then the relative interior of 
$\widetilde C_{T,e} \cap \widetilde C_{T',e'}$

As before, let $f_m^{T,e}$ denote the lift of %$\sigma_m^{T,e}: $
$\widetilde C_{T,e} \to \Pic^{m+g-1}(\Gamma)$ in the diagram
\[ \begin{tikzcd}
 \prod_{i=1}^g [0,\ell(e_i)] \ar[d] \ar[r,dashed,"f_m^{T,e}"] & \RR^g \ar[d,"\pi"] 
 & \prod_{i=1}^g [0, \ell(e_i')] \ar[l, dashed, "f_m^{T',e'}",swap] \ar[d] \\
\widetilde C_{T,e} \ar[r]  & \Pic^{m+g-1}(\Gamma) 
 & \widetilde C_{T',e'} \ar[l] 
\end{tikzcd} \]
and define $f_m^{T',e'}$ analogously.

We may choose coordinates (depending on $T$) on %the universal cover
$\RR^g$
such that 
\begin{equation*}
\label{eq:matrix-form}
\text{the matrix representing $f_m^{T,e}$ is} \quad
 \begin{pmatrix}
m & & \\
 & 1 & \\
 & & \ddots \\
 & & & 1
\end{pmatrix} .
\end{equation*}
In these same coordinates,
the matrix representing $f_m^{T',e'}$ is
\begin{equation*}
%\text{the matrix representing $f_m^{T',e'}$ is} \quad
 \begin{pmatrix}
-m & & \\
* & 1 & \\
* & & \ddots \\
* & & & 1
\end{pmatrix} 
\text{ if } j = 1, \quad\text{ or }\quad
 \begin{pmatrix}
m & & * & \\
& \ddots & * &\\ 
& & -1 & \\
& & * & \ddots 
\end{pmatrix} 
\text{ if } j \in \{ 2,\ldots, g\}.
\end{equation*}
(Recall that $j $ is the index specifying which edge $e_j \in G\backslash T$
has a break divisor chip on one of its endpoints;
$e_j$ is the unique edge in $T' \backslash T$.)
This shows that $\sigma_m$ is a local homeomorphism 
in a  neighborhood of the chosen facet of $\partial \widetilde C_{T,e}$.

%The positive constant $m$ does not change the homeomorphism type 
%of the image of a sufficiently small neighborhood in $\widetilde U_2$.

\ul{Claim 4}: The branched cover 
$\sigma_m : \widetilde\Br^g(\Gamma) \to \Pic^{m+g-1}(\Gamma)$
has degree $gm$.

\textit{Proof of Claim 4}:
When $m = 1$, 
it is clear that $\sigma_1: \widetilde \Br^g(\Gamma) \to \Pic^g(\Gamma) \cong \Br^g(\Gamma)$
is a degree $g$ branched cover.
When $m> 1$, we note that $\sigma_m$ differs from $\sigma_1$
by a scaling factor of $m$,
i.e. on a sufficiently small neighborhood $U \subset \widetilde{\Br}(\Gamma)$,
the Haar measure of $\sigma_m(U)$ is $m$-times as large as  
the Haar measure of $\sigma_1(U)$.
(The space $\Pic^{m+g-1}(\Gamma)$ carries a Haar measure
since it is a torsor for the compact topological group $\Pic^0(\Gamma)$.)
This implies that the degree of $\sigma_m$ as a branched cover
must be $m$ times the degree of $\sigma_1$,
so $\sigma_m$ must have degree $gm$ as desired.
\end{proof}

\begin{customthm}{A}
Let $\Gamma$ be a compact, connected metric graph of genus $g$. 
\begin{enumerate}[(a)]

\item For a generic divisor class of degree $\en\geq g$, 
the  Weierstrass locus $W(D)$ is finite
with cardinality $\#W(D) = g(\en-g+1)$.
For a generic divisor class of degree $\en<g$, $W(D)$ is empty.

\item For an arbitrary divisor class of degree $\en\geq g$,
the stable Weierstrass locus 
$\Wstab(D)$ is finite with cardinality
\[ \# \Wstab(D) \leq g(\en-g+1),\]
and equality holds for a generic divisor class.

\end{enumerate}
\end{customthm}
\begin{proof}
Part (b) is a restatement of Proposition \ref{prop:w-count}.

For part (a), 
first suppose $n < g$.
The space 
$\Pic^n(\Gamma)$ has dimension $g$,
while the subspace of effective divisor classes %in $\Pic^n(\Gamma)$
has dimension at most $n$.
Thus a generic divisor class in $\Pic^n(\Gamma)$
is not effective, assuming  $n < g $.
By Remark~\ref{rmk:w-empty}, the Weierstrass locus is empty
for a non-effective divisor class.

Now suppose $n \geq g$.
To prove (a), it suffices to show that $W(D) = \Wstab(D)$
for a generic divisor class, since then part (b) applies.
To compare  $W(D)$ with  $\Wstab(D)$,
we construct a map $X\to \Pic^\en(\Gamma)$
whose fiber over $[D]$ is the Weierstrass locus $W(D)$;
this parallels our construction in Section~\ref{sec:wstab-setup}
for $\Wstab(D)$.

For $ m \geq 1$, let 
$s_m : \Gamma \times \Sym^{g-1}(\Gamma) \to \Pic^{m+g-1}(\Gamma)$
denote the map
\[ s_m(x,E) = [ mx + E] .\]
Let $\pi_1 : \Gamma \times \Sym^{g-1}(\Gamma) \to \Gamma$
denote projection to the first factor.

The Riemann--Roch formula, Theorem~\ref{thm:riemann-roch},
implies that  a generic divisor class $[D]\in\Pic^{m+g-1}(\Gamma)$
has rank $r(D) = (m+g-1) -g = m-1$.
For such a divisor,
%the Weierstrass locus 
\[ W(D) = \{ x \in \Gamma : [D - mx] \geq 0\} = \pi_1 ( s_{m}^{-1} [D]).\] 
Recall that $\Wstab(D) = \pi_1(\sigma_m^{-1}[D])$,
where $\sigma_m$ is defined to be the restriction of $s_m$
to the subset $\widetilde \Br^g(\Gamma) \subset \Gamma \times \Sym^{g-1}(\Gamma)$;
note that 
\begin{equation}
\label{eq:w-comparison2}
 \sigma_m^{-1} [D] = s_m^{-1} [D] \cap \widetilde \Br^g(\Gamma) \subset  s_m^{-1} [D] .
\end{equation}
Under the genericity assumption on $[D]$, we have
\begin{equation*}
\label{eq:w-comparison}
%%% W(D) = \pi_1 ( s_m^{-1} [D]) \supset \pi_1 ( \sigma_m^{-1} [D]) = \Wstab(D),
 \Wstab(D) = \pi_1 ( \sigma_m^{-1} [D]) \subset \pi_1 ( s_m^{-1} [D]) = W(D).
\end{equation*}
Using part (b), this observation implies that a generic Weierstrass locus
$W(D)$ contains at least $g(n-g+1)$ points.

We consider when $W(D)$ can be strictly larger than $\Wstab(D)$.
By \eqref{eq:w-comparison2}, this happens only if $s_m^{-1}[D]$ 
is not contained in $\widetilde{\Br}(\Gamma)$;
equivalently, only if $[D]$ lies in the image of
$ (\Gamma \times \Sym^{g-1}(\Gamma)) \backslash \widetilde{\Br}(\Gamma)$
under $s_m$.

%\harry{move following claim to proposition / lemma}

\ul{Claim}: The image $s_m(\, (\Gamma \times \Sym^{g-1}(\Gamma)) \backslash \widetilde{\Br}(\Gamma) \,)$
has dimension $g-1$ in $\Pic^{m+g-1}(\Gamma)$.

It is clear that $s_m$ is piecewise affine on $\Gamma \times \Sym^{g-1}(\Gamma)$,
with domains of linearity  indexed by $g$-tuples of edges
$(e_1; e_2 ,\ldots, e_g)$,
up to reordering the edges $e_2, \ldots, e_g$.
(Here we choose an arbitrary combinatorial model $(G,\ell)$ for $\Gamma$.)
The edges $e_i$ are not necessarily distinct.

If the edges $(e_1; e_2,\ldots, e_g)$
form the complement of a spanning tree $T$ in $G$,
then the corresponding domain is in $\widetilde \Br^g(\Gamma)$;
namely, it is the cell $\widetilde C_{T,e_1}$ in the notation of Section~\ref{sec:wstab-setup}.
Conversely, if the edges $(e_1; e_2,\ldots, e_g)$
are not the complement of a spanning tree in $G$,
then either some edge is repeated or the edges contain a cut set of $G$.
In either case, 
the fibers of  $s_m : \Gamma \times \Sym^{g-1}(\Gamma) \to \Pic^{m+g-1}(\Gamma)$
have dimension at least $1$ over the interior of the corresponding domain
(see \cite[Proposition 13]{HMY}),
so the image of this domain
under $s_{m}$ has dimension at most $g-1$.
This proves the claim.

The claim implies that for a generic divisor class $[D]$,
the preimage $s_m^{-1}[D]$ is contained in $\widetilde \Br^g(\Gamma)$.
By \eqref{eq:w-comparison2} this implies $W(D) = \Wstab(D)$, as desired.
\end{proof}

\section{Distribution of Weierstrass points}
\label{sec:equidistribution}
In this section we prove  Theorem~\ref{thm:equidistribution}.
We show that for a degree-increasing sequence of generic divisors on a metric graph,
the Weierstrass points become distributed 
with respect to the Zhang canonical measure 
(defined in Section \ref{subsec:resistance}).
We also give a quantitative version of this distribution result,
Theorem~\ref{thm:quant-equidistribution}.

Our proofs of Theorems~\ref{thm:equidistribution} and \ref{thm:quant-equidistribution}
work unchanged when $W(D)$ is replaced by the stable Weierstrass locus $\Wstab(D)$.
%makes use of the fact that 
%$W(D) = \Wstab(D)$ for a generic divisor class.

\subsection{Examples}
First we consider some low genus examples of Weierstrass points 
converging to a limiting distribution. 
\begin{eg}[Genus zero metric graph] 
Let $\Gamma$ be a genus zero metric graph.
For any divisor $D_{\en}$, 
the associated Weierstrass locus $W(D_{\en}) $ is empty so $\delta_{\en} = 0$.
All edges are bridges, so the canonical measure is $\mu = 0$.
\end{eg}

\begin{eg}[Genus one metric graph] 
Let $\Gamma$ be a genus one metric graph 
which consists of a loop of length $L$.
For a divisor $D_{\en}$ of degree $\en$, 
the Weierstrass locus $W_\en = W(D_\en)$ consists of $\en$ evenly-spaced points (``torsion points'') around the loop. 
The distance between adjacent points is ${L}/{\en}$, 
so on a segment $e$ of length $\length{e}$ the number of Weierstrass points is bounded by
\[ \frac{\length{e}}{L/\en} - 1 \leq \#(W_{\en}\cap e) \leq \frac{\length{e}}{L/\en} + 1 .\]
%Normalizing by $\frac1{\en}$,
This means the associated discrete measure 
$\delta_{\en} = \frac1n \sum_{x\in W_\en} \delta_x$ satisfies
\[ \delta_\en(e) = \frac{\#(W_{\en}\cap e)}{\en} \qquad\Rightarrow\qquad
\frac{\length{e}}{L} - \frac1{\en} \leq \delta_{\en}(e)    % = \frac{\#(W_{\en}\cap e)}{\en} 
\leq \frac{\length{e}}{L} + \frac1{\en} .\]
Hence $ \delta_{\en}(e)  \to \frac{\length{e}}{L} = \mu(e)$
as $\en\to \infty$.
\end{eg}

\subsection{Proofs}
We now address the limiting distribution of Weierstrass points 
$W(D_n)$ as $n\to \infty$ 
in the case of an arbitrary metric graph $\Gamma$.

\begin{lem}
\label{lem:wfinite}
Suppose the Weierstrass locus $W(D)$ is finite. 
Let $r = r(D)$.
\begin{enumerate}[(a)]
\item If $x$ is in the interior of a segment, $\reddiv{x}{D}$ contains at most $r+1$ chips at $x$.

\item If $x$ is in the interior of a segment $e \subset \Gamma$,
$\reddiv{x}{D}$ contains at most $r+1$ chips on $e$ (including its endpoints).

\end{enumerate}
\end{lem}
\begin{proof}
(a) Suppose $\reddiv{x}{D}$ contains $r+2$ chips at $x$.
 Then for sufficiently small $\epsilon > 0$
we can move $r+1$ of these chips together for a distance $\epsilon$ in one direction,
while moving $1$ chip a distance $(r+1)\epsilon$ in the other direction.
This gives an $\epsilon$-length interval in $W(D)$, a contradiction.

(b)
Suppose $\reddiv{x}{D}$ contains $r+2$ chips on the closed segment $e$.
Note that at least $r$ of these chips must be at $x$, in the interior of $e$.
By chip-firing, we may move all $r+2$ chips to a single point $x' $ in the interior
of $ e$. Then part (a) applies.
\end{proof}

%\todo{decide where to cite this result from previous section }
%\begin{thm}
%The reduced divisor map is locally geodesic.
%\end{thm}

\begin{customthm}{B}
%\label{thm:equidistribution}
Let $\{ D_{\en} : \en\geq 1\} $ be a sequence of  divisors on $\Gamma$
with $\deg D_{\en} = \en$. 
Let $W_{\en}$ be the Weierstrass locus of $D_{\en}$.
Suppose each $W_\en$ is a finite set, and let
\[ \delta_{\en} = \frac1{\en} \sum_{x\in W_{\en}} \delta_x\]
denote the normalized discrete measure on $\Gamma$ associated to $W_{\en}$. 
Then as $\en\to \infty$, the measures $\delta_{\en}$ converge weakly 
to the Zhang canonical measure $\mu$ on $\Gamma$.
\end{customthm}

Recall that by definition of weak convergence, 
Theorem~\ref{thm:equidistribution} says that for any continuous function $f : \Gamma \to \RR$, 
as $\en \to \infty$ we have convergence 
\[ \frac1{\en} \sum_{x\in W_{\en}}f(x) 
= \int_\Gamma f(x) \delta_{\en}(dx) 
\quad\to \quad
\int_\Gamma f(x) \mu(dx) .\]

\begin{proof}[Proof of Theorem~\ref{thm:equidistribution}]
To show weak convergence of measures on $\Gamma$ it suffices to 
show convergence when integrated against step functions.
Hence it suffices to integrate the measures against the 
indicator function of an arbitrary segment of $\Gamma$.

Let $e$ be a segment in the metric graph $\Gamma$ of length $\length{e}$, 
with endpoints $\startv$ and $\tailv$.
Let $W_{\en} \cap e$ 
denote the set of Weierstrass points of $D_\en$ lying on the segment $e$.
It suffices to show that 
\begin{equation}
\label{eq:convergence-e}
 \lim_{\en\to\infty} \frac{\#(W_{\en} \cap e) }{\en} =  \mu(e).
\end{equation}
% where $\mu$ is the canonical measure on $\Gamma$.
Recall that by Proposition \ref{prop:canonical-formula},
\[ \mu(e) = \frac{\length{e}}{\length{e} + \lengtheff{\Gamma \backslash e}}\]
where 
%$\length{e}$  denotes the length of the segment $e$ and 
$\lengtheff{\Gamma\backslash e}$  denotes the effective resistance 
between the endpoints of $e$ after the interior of $e$ is removed from $\Gamma$.
(If $\Gamma \backslash e$ is disconnected, 
 $\lengtheff{\Gamma \backslash e} = +\infty$ and $\mu(e) = 0$.)
We prove \eqref{eq:convergence-e} by relating each side to  
the slope of a piecewise linear function on $\Gamma$.

For the right-hand side of \eqref{eq:convergence-e}, 
consider the voltage function $\potent{t}{s} : \Gamma \to \RR$ 
% which gives the 
% electric potential on $\Gamma$ viewed as a resistor network,
(see Section \ref{sec:voltage}).
The voltage drop in $\Gamma$ between endpoints of $e$ is
the effective resistance
\[ \potent{\tailv}{\startv}(\startv) - \potent{\tailv}{\startv}(\tailv) 
= r(\startv,\tailv) = \frac{\length{e} \lengtheff{\Gamma\backslash e}}{\length{e} + \lengtheff{\Gamma\backslash e}},\]
by the parallel rule for effective resistance.
Thus we have
\begin{align} 
\label{eq:j-slope}
%\text{slope of }\potent{\tailv}{\startv} \text{ on }e = 
\frac{\potent{t}{s}(s) - \potent{t}{s}(t)}{\length{e}} 
&= \frac{ \lengtheff{\Gamma\backslash e}}{\length{e} + \lengtheff{\Gamma\backslash e}} 
= 1-\frac{\length{e}}{\length{e} + \lengtheff{\Gamma\backslash e}} 
= 1 - \mu(e).
\end{align}
(Recall that this slope can be interpreted as the current flowing along the segment $e$ 
when a unit current is forced from $\startv$ to $\tailv$,
see Section~\ref{sec:voltage}.)
% since 
% $\text{current} = \frac{\text{voltage drop}}{\text{resistance}} $.)

To connect $\potent{t}{s}$ to the left-hand side of \eqref{eq:convergence-e}, 
we consider a sequence of piecewise linear functions %in $\PL_\ZZ(\Gamma)$ 
which are discrete approximations of $\potent{t}{s}$,
as in Proposition~\ref{prop:voltage-approx},
and show that certain slopes in these functions 
are related to the number of Weierstrass points. 

Let $f_{\en}$ be the piecewise $\ZZ$-linear function on $\Gamma$ satisfying
\[ \Divisor(f_{\en}) = \reddiv{t}{D_{\en}} - \reddiv{s}{D_{\en}}
\qquad\text{and}\qquad f_{\en}(t) = 0. \]
(Recall that $\reddiv{x}{D}$ denotes the $x$-reduced divisor linearly equivalent to $D$.)
By Proposition \ref{prop:voltage-approx},  as $\en\to \infty$ we have uniform convergence 
\begin{equation}
\label{eq:f-to-j}
 \frac1{\en} f_{\en} \to \potent{\tailv}{\startv}.
 \end{equation}
Thus to show \eqref{eq:convergence-e} using \eqref{eq:j-slope} and \eqref{eq:f-to-j}, 
it suffices to show that
\begin{equation}
\label{eq:approx-limit}
\lim_{\en\to\infty} \frac1{\en} \left( \frac{f_{\en}(s) - f_{\en}(t)}{\length{e}}   \right)
= 1 - \lim_{\en\to\infty}\frac{\# (W_{\en}\cap e)}{\en}.
\end{equation}

We first give an intuitive explanation for \eqref{eq:approx-limit}:
the slope of the function $f_{\en}$ on a directed segment is equal to 
the net flow of chips across the segment,
as we move from $\reddiv{\startv}{D_{\en}}$ to $\reddiv{\tailv}{D_{\en}}$ along any path in the linear system $|D_{\en}|$.
If we follow $\reddiv{x}{D_{\en}}$ as $x$ varies from  $\startv$ to  $\tailv$, 
we have $\en-g$ chips moving in the ``forward'' direction of $e$ (following $x$)
and some number of chips moving in the reverse direction one-by-one.
%\footnote{Here we need the hypothesis that $W_n$ is finite.}
The number of ``reverse-moving'' chips is equal to $\#(W_{\en}\cap e)$,
since $x$ is in $W_n$ exactly when $\reddiv{x}{D_{\en}}$ has an ``extra'' chip at $x$,
i.e. when the $\en-g$ chips on $x$ collide with a reverse-moving chip.
Thus the net number of chips moving across the segment $e$ is equal to 
$(\en-g) - \#(W_{\en} \cap e)$,
up to some bounded error due to boundary behavior.
%since the first and last reverse chips may move only partially along $e$. 
This yields \eqref{eq:approx-limit}
after dividing by $\en$ and taking $\en \to \infty$.

Now we give a rigorous argument.
Let $w_1, w_2, \ldots, w_m$ denote the Weierstrass points on $e$, 
ordered from $\startv$ to $\tailv$, 
so that $m = \#(W_\en \cap e)$.
Here we use the hypothesis that $W_n$ is finite.
(Note that $m = m_n$ depends on $n$ and $e$.)

We partition the segment $e = [s,t]$ into subintervals
$[s,w_1], [w_1,w_2], \ldots, [w_m,t]$.
(It is possible that the intervals $[s,w_1]$ and $[w_m,t]$ are degenerate.)
%Recall that $\length{e} = \ell(e)$  is the length of the segment $e$.
Let $\length{[w_i,w_{i+1}]}$ denote the length of the segment 
$[w_i,w_{i+1}] \subset e$.
We have
\begin{equation*}
\length{e} = \length{[s,w_1]} + \length{[w_1,w_2]} + \cdots + \length{[w_{m-1},w_m]} + \length{[w_m,t]} .
\end{equation*}

For each $i = 1, 2, \ldots, m-1$,
let $g_\en^{(i)}$ denote the function in $\PL_\ZZ(\Gamma)$ satisfying
\[
\Divisor( g_\en^{(i)} ) = \reddiv{w_{i+1}}{D_\en} - \reddiv{w_i}{D_\en},
\]
and let $g_\en^{(0)}$ and $g_\en^{(m)}$
denote functions satisfying
\begin{equation*}
\Divisor( g_\en^{(0)}) = \reddiv{w_1}{D_\en} - \reddiv{s}{D_\en},
\qquad\text{and}\qquad
\Divisor( g_\en^{(m)} ) = \reddiv{t}{D_\en} - \reddiv{w_m}{D_\en}.
\end{equation*}
By adding an appropriate constant, we may assume that
$g_\en^{(i)}(t) = 0$ for each $i = 0, 1, \ldots, m$.
By telescoping of poles and zeros, we have %by construction that
\[ \Divisor( f_n ) = \Divisor( g_n^{(0)} ) + \Divisor( g_n^{(1)} ) + \cdots + \Divisor( g_n^{(m)}). \]
With the additional constraint that $f_n(t) = \sum_i g_n^{(i)}(t)=0$, this implies that
\begin{equation}
\label{eq:f-gsum}
 f_n = g_n^{(0)} + g_n^{(1)} + \cdots + g_n^{(m)} .
\end{equation}
Thus we can compute $f_n(s) - f_n(t)$ by summing
$ \sum_{i=0}^m \left( g^{(i)}(s) - g^{(i)}(t) \right) $.

To analyze the slopes of $g^{(i)}$ on segment $e$,
we make use of Lemma~\ref{lem:wfinite}.
This information is sufficient to deduce all slopes over $e$.
We may assume %without loss of generality
 that $r(D_\en)  = \en - g$,
since this holds for $\en \geq 2g-1$.

For $i = 1, 2, \ldots, m-1$, 
the function $g_n^{(i)}$ has slope $-(n-g)$ on the interval $[w_i, w_{i+1}]$,
and slope $1$ on $e $ outside of this interval. % [w_i, w_{i+1}]$.
See Figure~\ref{fig:w-mid}.
\begin{figure}[h]
\centering
\begin{tikzpicture}[xscale=0.8]
	\coordinate (O) at (0,0);
	\coordinate (A) at (2,0.2);
	\coordinate (B) at (3,-1.2);
	\coordinate (C) at (10,-0.5);
	\coordinate (Od) at (0,-2);
	\coordinate (Ad) at (2,-2);
	\coordinate (Bd) at (3,-2);
	\coordinate (Cd) at (10,-2);
	
	\draw (O) -- node[above] {$1$} (A) 
		-- node[above right] {$-(n - g)$} (B) 
		-- node[above] {$1$} (C);
	\draw (Od) -- (Ad) -- (Bd) -- (Cd);
	\draw[dotted] (O) -- +(-0.5,0);
	\draw[dotted] (Od) -- +(-0.5,0);
	\draw[dotted] (C) -- +(0.5,0);
	\draw[dotted] (Cd) -- +(0.5,0);
	
	\fill (Od) circle (2pt);
	\fill (Ad) circle (2pt);
	\fill (Bd) circle (2pt);
	\fill (Cd) circle (2pt);
	
	\node[below] at (Od) {$s$};
	\node[below] at (Ad) {$w_i$};
	\node[below] at (Bd) {$w_{i + 1}$};
	\node[below] at (Cd) {$t$};
	
	\node[left=0.8] at (Od) {segment $e$};
	\node[left=0.7] at (O) {$g_n^{(i)}$};
\end{tikzpicture}
\caption{Function $g_{\en}^{(i)}$ having 
    zeros $\reddiv{w_{i+1}}{D_\en}$ and poles $\reddiv{w_i}{D_\en}$,
    with slopes are indicated above each affine part.}
\label{fig:w-mid}
\end{figure}

Thus we have
\begin{align}
\label{eq:w-mid}
g_n^{(i)}(s) - g_n^{(i)}(t) &= (\en - g ) \length{[w_i,w_{i+1}]} - \length{[s,w_i]} - \length{[w_{i+1},t]} \nonumber \\
% &= (\en - g ) \length{[w_i,w_{i+1}]} - (\length{e} - \length{[w_i,w_{i+1}]}) \nonumber \\
 &= (\en - g + 1) \length{[w_i,w_{i+1}]} - \length{e} .
\end{align}

For $i = 0$ and $i=m$, to write an expression for $g_n^{(i)}(x) - g_n^{(i)}(t)$
we need to set additional notation.
If $\reddiv{\startv}{D_\en}$ has a chip in the interior of $e$,
let   $y$ be the position of this chip (which is unique
by Lemma~\ref{lem:wfinite});
otherwise, let $y = \tailv$.
Similarly, let $z$ be the position of the unique chip of 
$\reddiv{\tailv}{D_\en}$ in the interior of $e$ if it exists;
otherwise let $z = \startv$.
We have
\begin{align}
\label{eq:w-start}
g_n^{(0)}(s) - g_n^{(0)}(t) &= (\en - g ) \length{[s,w_{1}]} - \length{[w_1,y]} \nonumber  \\
 &= (\en - g + 1) \length{[s,w_1]} - \length{[s,y]} 
\end{align}
and
\begin{align}
g_n^{(m)}(s) - g_n^{(m)}(t) &= (\en - g ) \length{[w_{m},t]} - \length{[z,w_m]}  \nonumber \\
\tag{\ref{eq:w-start}'}
 &= (\en - g + 1) \length{[w_m,t]} - \length{[z,t]} . 
\end{align}
\begin{figure}[h]
\centering
\begin{tikzpicture}[xscale=0.8]
	\coordinate (O) at (0,2.2);
	\coordinate (A) at (1,0.8);
	\coordinate (B) at (8,1.5);
	\coordinate (C) at (10,1.5);
	\coordinate (Od) at (0,0);
	\coordinate (Ad) at (1,0);
	\coordinate (Bd) at (8,0);
	\coordinate (Cd) at (10,0);
	
	\draw (O) -- node[above right] {$-(n - g)$} (A) 
		-- node[above] {$1$} (B) 
		-- node[above] {$0$} (C);
	\draw (Od) -- (Ad) -- (Bd) -- (Cd);
	\draw[dotted] (O) -- +(-0.5,0);
	\draw[dotted] (Od) -- +(-0.5,0);
	\draw[dotted] (C) -- +(0.5,0);
	\draw[dotted] (Cd) -- +(0.5,0);
	
	\fill (Od) circle (2pt);
	\fill (Ad) circle (2pt);
	\fill (Bd) circle (2pt);
	\fill (Cd) circle (2pt);
	
	\node[below] at (Od) {$s$};
	\node[below] at (Ad) {$w_1$};
	\node[below] at (Bd) {$y$};
	\node[below] at (Cd) {$t$};
	
	\node[left=0.8] at (Od) {segment $e$};
	\node[left=0.7] at (O) {$g_n^{(0)}$};
\end{tikzpicture}
\caption{Function $g_{\en}^{(0)}$ 
having zeros $\reddiv{w_1}{D_\en}$ and poles $\reddiv{s}{D_\en}$.}
\label{fig:w-start}
\end{figure}

Thus adding the expressions \eqref{eq:w-mid} and \eqref{eq:w-start} together,
by \eqref{eq:f-gsum} we have
\begin{align*}
f_\en(s) - f_\en(t) &= (\en-g+1)\big( \length{[s,w_1]}+\length{[w_1,w_{2}]} + \cdots + \length{[w_{m-1},w_m]} + \length{[w_m,t]} \big) \\
& \qquad - \length{[s,y]} - (m-1)\length{e} - \length{[z,t]} \\
 &= (\en -g + 1) \length{e} - (m-1)\length{e} - \length{[s,y]} - \length{[z,t]} \\
 &= (\en -g -m + 2) \length{e}  - \length{[s,y]} - \length{[z,t]} \\
 &= (\en - g - m) \length{e} + (\length{e}- \length{[s,y]}) + (\length{e} - \length{[z,t]}) \\
 &= (\en - g - m)\length{e} + \length{[y,t]} + \length{[s,z]} .
\end{align*}
Since $0 \leq \length{[y,t]} + \length{[s,z]} \leq 2\length{e}$
and $m = \#(W_\en \cap e)$,
this shows that
\begin{equation*}
n - g - \#(W_\en \cap e) \leq \frac{f_n(s) - f_n(t)}{\length{e}} \leq n - g + 2 - \#(W_\en \cap e).
\end{equation*}
Dividing by $\en$ and taking the limit $\en \to \infty$ yields \eqref{eq:approx-limit}
as desired.
\end{proof}

\begin{thm}
Consider the setup of Theorem~\ref{thm:equidistribution}.
\begin{enumerate}[(a)]
\item 
Suppose each $[D_\en]$ is generic in $\Pic^\en(\Gamma)$.
Then each $W_n$
is finite and we have weak convergence $\delta_\en \to \mu$.

\item Let $\Wstab_\en = \Wstab(D_\en)$ be the stable Weierstrass locus,
and define $\delta^\text{st}_\en$ analogously to $\delta_\en$.
For any divisors $\{ D_{\en} : \en \geq 1\}$
we have weak convergence $\delta^\text{st}_\en \to \mu$.
\end{enumerate}
\end{thm}
\begin{proof}
(a) This is part of Theorem~\ref{thm:w-finite}.

(b) We may follow the same argument used in Theorem~\ref{thm:equidistribution},
except in place of $\reddiv{x}{D_\en}$
we consider the ``stable reduced divisor''
\[ \reddivst{x}{D_\en} := (n-g)x + \br[D_\en - (n-g)x] .\]
With this change in the definitions of $f_\en$ and $g_\en^{(i)}$,
 equations \eqref{eq:w-mid} and \eqref{eq:w-start} still hold,
as does the convergence \eqref{eq:f-to-j}.
\end{proof}

\begin{customthm}{C}[Quantitative distribution of $W(D)$]
Let $\Gamma$ be a metric graph of genus $g$,
let $D_{\en}$ be a  divisor class of degree $\en > g$ 
and let $W_{\en}$ denote the Weierstrass locus of $D_{\en}$.
Suppose $W_\en$ is finite.
Let $\mu$ denote the Zhang canonical measure on $\Gamma$.
\begin{enumerate}[(a)]
\item 
\label{it:w-quantitative}
For any segment $e$ in $\Gamma$,
\begin{equation*}
\en \mu(e) - 2g \leq \# (W_{\en}\cap e) \leq \en \mu(e) + g + 2.
\end{equation*}

\item 
\label{it:w-nonempty} 
If $e$ is a segment of $\Gamma$ with   $\mu(e) > \frac{2g}{\en}$,
then $e$ contains at least one Weierstrass point of $D_{\en}$.

\item 
\label{it:w-integration}
For a fixed continuous function $f : \Gamma \to \RR$,
\begin{equation*}
\frac1{\en}\sum_{x \in W_{\en}} f(x) = \int_\Gamma f(x) \mu(dx) + O\left(\frac1{\en}\right).
\end{equation*}
%(The big-$O$ constant may depend on $f$.)

\end{enumerate}
\end{customthm}

\begin{proof}
It is clear that part \eqref{it:w-nonempty} follows from part \eqref{it:w-quantitative}, 
since $\#(W_{\en} \cap e)$ must be an integer.
Part \eqref{it:w-integration} is a straightforward extension of \eqref{it:w-quantitative}.

We now prove part (a).
Let $f_{\en}$ be the piecewise linear function satisfying 
$\Divisor(f_{\en}) = \reddiv{\tailv}{D_{\en}} - \reddiv{\startv}{D_{\en}}$
and $f_{\en}(\tailv) = 0$,
 where $\startv$ and $\tailv$ are the endpoints of $e$.
By Proposition \ref{prop:quant-voltage-approx}, 
we have
\[ |(f_{\en} - (\en-g) \potent{\tailv}{\startv})'(x)| \leq g\]
so
\[ |f_{\en}'(x)| \leq (\en - g)|j'(x)| + g.\]
Recall that for $x$ on the segment $e$,
$ |j'(x)| = 1 - \mu(e)$.
Thus we have the bound
\[ | f_{\en}'(x)| \leq \en - \en\mu(e) + \mu(e) g .\]
Moreover the proof of Theorem \ref{thm:equidistribution} shows that
\[ \en - g - \#(W_{\en}\cap e) \leq |f_{\en}'(x)|. \]
Combining these inequalities gives
\[ \en\mu(e) - (1 + \mu(e)) g  \leq \# (W_{\en}\cap e).\]
Finally, the inequality $\mu(e) \leq 1$ from Corollary \ref{cor:canonical-bound} yields
the lower bound in (a).

We similarly obtain the upper bound
\[ \#(W_{\en} \cap e) \leq \en\mu(e) + g + 2  \]
by combining the inequalities
\[ \en - \en\mu(e)-(2 - \mu(e)) g \leq |f_{\en}'(x)| \quad\text{and}\quad 
 |f_{\en}'(x)| \leq \en-g - \#(W_{\en}\cap e) + 2\]
 and $\mu(e) \geq 0$ from Corollary \ref{cor:canonical-bound}.
\end{proof}

\appendix

\section{Theta intersections}

Throughout this section, we assume that the divisor class $[D]$
is {\em (Riemann--Roch) nonspecial}, 
meaning that its rank satisfies 
\begin{equation*}
r(D) = \begin{cases}
\deg(D) - g & \text{if } \deg(D) \geq g, \\
 -1 & \text{otherwise}.
\end{cases}
\end{equation*}
A generic divisor class in $\Pic^n(\Gamma)$ is nonspecial.
If $n\geq 2g-1$, all divisors in $\Pic^n(\Gamma)$ are nonspecial.

In this appendix we give an alternate description of 
the Weierstrass locus $W(D)$ as the intersection of two 
polyhedral subcomplexes of complementary dimension in $\Pic^{g-1}(\Gamma)$.
This allows us to give an alternate proof 
%of Theorem \ref{thm:w-finite} (a)
that $W(D)$ is  finite
for a generic divisor class $[D]$.
In this perspective, the stable Weierstrass locus $\Wstab(D)$
naturally appears as the stable tropical intersection of these two subsets.

\subsection{Intersection with \texorpdfstring{$\Theta$}{Theta}}
\label{subsec:theta-intersection}
Recall that the {\em theta divisor} $\Theta \subset \Pic^{g-1}(\Gamma)$ 
is the space of degree $g - 1$ divisor classes which have an effective representative;
\[ \Theta = \{ [D]\in \Pic^{g-1}(\Gamma) : [D]\geq 0\}.\]

Given a divisor $D$ of degree $n \geq g$, 
let $\Phi_D: \Gamma \to \Pic^{g-1}(\Gamma)$ denote the map
\[ 
	\Phi_D: x \mapsto [D - (\en - g + 1)x] .
  \]
If $D$ has degree $n < g$ let $\Phi_D: x \mapsto [D]$ be the constant map.
Note that the map $\Phi_D$ depends only on the divisor class $[D]$.
If  $D$ is nonspecial,
%i.e. has degree at least $2g-1$ (so that $r(D) = \en-g$),
the Weierstrass locus of $D$ is equal to the  intersection $\Phi_D(\Gamma) \cap \Theta$,
 pulled back to $\Gamma$ from $\Pic^{g-1}(\Gamma)$.
\begin{prop}
\label{prop:theta-intersection}
Let $D$ be a divisor of degree $n\geq g$,
and let $\Phi_D : \Gamma \to \Pic^{g-1}(\Gamma)$ be the map
$\Phi_D(x) = [D - (n-g+1)x]$.
If $D$ is a nonspecial,
\begin{equation*}
 W(D) = \Phi_D^{-1} ( \Phi_D(\Gamma) \cap \Theta) .
\end{equation*}
\end{prop}
\begin{proof}
This follows from the definition of Weierstrass locus, 
if $D$ has rank $n-g$.
\end{proof}

\begin{prop}
\label{prop:phi-injective}
Suppose $\Gamma$ is a bridgeless metric graph.
If $D$ has degree $ \en \geq g$, 
the map $\Phi_D : \Gamma \to \Pic^{g-1}(\Gamma)$ is locally injective
(i.e. an immersion).
\end{prop}
\begin{proof}
The map $\Phi_D$ may be expressed as a composition of three maps
\[ \Phi_D : \Gamma \xrightarrow{\alpha} \Pic^1( \Gamma ) \xrightarrow{\beta} 
 \Pic^{\en-g+1}(\Gamma) \xrightarrow{\gamma} \Pic^{g-1}( \Gamma),\]
where $\alpha$ sends $x \mapsto [x]$, $\beta$ sends $[E] \mapsto [(\en-g+1)E]$, and $\gamma$ sends $[E] \mapsto [D-E]$.
The map $\gamma = \gamma_D$ is a homeomorphism.
The map $\beta$ is a $(\en - g+1)^g$-fold covering map, so it is a local homeomorphism if $\en \geq g$.
Thus it suffices to verify that the first map $\alpha$ is locally injective. 

This follows from the Abel--Jacobi theorem for metric graphs,
see e.g. Baker--Faber \cite[Theorem 4.1 (3)(4)]{BF2}.
Note that $\Pic^1(\Gamma)$ is (non-canonically) isomorphic to the Jacobian 
$\Jac(\Gamma) = \Pic^0(\Gamma)$ 
by choosing a basepoint $x_0$ to subtract.
\end{proof}

If $\Gamma$ contains bridge segments, 
let $\Gamma_\mathrm{/(br)}$ denote the metric graph obtained from $\Gamma$
by contracting all bridges.
Let $S_\mathrm{(br)} \subset \Gamma_\mathrm{/(br)}$ denote the 
set of points which were bridges in $\Gamma$.

\begin{lem}
\label{lem:bridge-contract}
Let $\pi : \Gamma \to \Gamma_\mathrm{/(br)}$
denote the canonical map contracting all bridge segments of $\Gamma$,
which induces $\pi_* : \Pic^\en(\Gamma) \to \Pic^\en(\Gamma_\mathrm{/(br)})$
for all $\en$.
For any divisor $D$ on $\Gamma$,
\[W(D) = \pi^{-1} W(\pi_*(D)) .\]
\end{lem}
\begin{proof}
On $\Gamma$ the linear equivalence map
$x \mapsto [x]$ factors through $\pi: \Gamma \to \Gamma_\mathrm{/(br)}$; i.e. 
we have a commuting diagram
\[ \begin{tikzcd}
 \Gamma \ar[d,swap,"{[x]}"] \ar[r,"\pi"] & \Gamma_\mathrm{/(br)} \ar[d,"{[x]}"] \\
\Pic^1(\Gamma) \ar[r,"\sim"]  & \Pic^{1}(\Gamma_\mathrm{/(br)}) .
\end{tikzcd} \]
Using this, the result is clear from the definition of $W(D)$.
\end{proof}

\begin{lem}
\label{lem:w-exception}
Suppose $S \subset \Gamma$ is a finite set of points in a metric graph $\Gamma$.
For a generic divisor class $[D]$,
the intersection $W(D) \cap S$ is empty.
\end{lem}
\begin{proof}
It suffices to consider when $S = \{s\}$ contains one point.
Assuming $D$ is nonspecial, 
which holds for generic $[D]\in \Pic^\en(\Gamma)$,
we have $s\in W(D)$ if and only if 
\[ [D - (n-g+1)s] \text{ is effective}
\quad\Leftrightarrow\quad
[D] = [(n-g+1)s + E] \text{ for some } [E] \in \Theta .\]
Since $\Theta$ has dimension $g-1$,
the space 
$\{ [D] = [(n-g+1)s + E] : [E] \in \Theta\}$
also has dimension $g-1$.
Hence a generic class $[D]$ has $s \not\in W(D)$.
\end{proof}

\begin{thm}
For a generic divisor class $[D]$ in $\Pic^\en(\Gamma)$, 
the Weierstrass locus $W(D)$ is finite.
\end{thm}

\begin{proof}
If $\en < g$, then a generic divisor class  in $\Pic^\en(\Gamma)$ is not effective
because the image of $\Sym^\en(\Gamma) \to \Pic^\en(\Gamma)$ 
has dimension at most $n$, while $\Pic^\en(\Gamma)$ has dimension $g$.
For a non-effective divisor class $[D]$, the Weierstrass locus $W(D)$ is empty.

Now suppose $\en\geq g$.
By Riemann--Roch, a generic divisor class in $\Pic^\en(\Gamma)$ has rank $r(D) = \en-g$.
(By the above paragraph,  $r(K-D) = -1$ generically.)
Thus, it suffices to show that $W(D)$ is finite for 
a generic nonspecial %\footnote{meaning $r(D) = \en-g$} 
divisor class.

\ul{Case 1: $\Gamma$ is bridgeless}.
As above, let 
$\Phi_D : \Gamma \to \Pic^{g-1}(\Gamma)$ be the map 
$ \Phi_D(x) =  [D-(\en-g+1)x]$.
Recall that the Weierstrass locus $W(D)$ is equal to
\[ W(D) = \Phi_D^{-1}(\Phi_D(\Gamma) \cap \Theta) \subset \Gamma \]
where 
$\Theta = \{ [E] \in \Pic^{g-1}(\Gamma) : [E]\geq 0\}$ 
is the theta divisor.
Note that as $[D]$ varies, the image $\Phi_D(\Gamma)$ varies by translation inside $\Pic^{g-1}(\Gamma)$.

Recall that 
$\Theta$ is a $(g-1)$-dimensional polyhedral complex with finitely many facets,
and 
$\Phi_D(\Gamma)$ is a 1-dimensional polyhedral complex with finitely many segments.
This implies that the space of translations which cause $\Phi_D(\Gamma)$ 
to intersect $\Theta$ non-transversally has dimension at most $g-1$.
Hence for a generic divisor class $[D]$,
the intersection $\Phi_D(\Gamma)\cap \Theta $ is transverse.

Suppose all intersections in $\Phi_D(\Gamma) \cap \Theta$ are transverse,
and occur in the interiors of the respective segment and facet.
Recall that $\Phi_D$ is locally injective by Proposition~\ref{prop:phi-injective}.
If $\Phi_D$ sends $x \in \Gamma$ to a transverse intersection,
then $x$ must have some neighborhood $U  \subset \Gamma$ 
such that $\Phi_D(U \backslash \{x\})$ is disjoint from $\Theta$.
This means that 
$W(D) = \Phi_D^{-1}(\Phi_D(\Gamma) \cap \Theta) $ is a discrete subset of $\Gamma$.
Because $\Gamma$ is compact, this implies $W(D)$ is finite.

\ul{Case 2: $\Gamma$ has bridge segments}.
Let $\pi : \Gamma \to \Gamma_\mathrm{/(br)}$ 
denote the map contracting all bridge segments of $\Gamma$.
Let $S_\mathrm{(br)} \subset \Gamma_\mathrm{/(br)}$
denote the image of all bridges, which is a finite subset of $\Gamma_\mathrm{/(br)}$.
Note that $\pi$ restricts to an injection 
away from $\pi^{-1} S_\mathrm{(br)}$.

By Lemma~\ref{lem:w-exception},
a generic divisor class $[D] \in \Pic^n(\Gamma_\mathrm{/(br)})$
has $W(D)$ disjoint from $S_\mathrm{(br)}$.
Since $\pi$ induces a homeomorphism $\pi_*: \Pic^n(\Gamma) \to \Pic^n(\Gamma_\mathrm{/(br)})$,
this implies that a generic class $[D] \in \Pic^n(\Gamma)$
has $W(\pi_*[D])$ disjoint from $S_\mathrm{(br)}$.
The result then follows from Lemma~\ref{lem:bridge-contract}
and Case 1.
\end{proof}

\subsection{Stable Weierstrass locus}
In this section we describe the relation of the current setup, involving the 
theta divisor $\Theta$, 
and the stable Weierstrass locus defined in Section~\ref{sec:wstab}.
\begin{prop}
Suppose $\Gamma$ is a bridgeless metric graph of genus $g$.
Let $D$ be a divisor of degree $g$, and let $\Phi_D : \Gamma \to \Pic^{g-1}(\Gamma)$
send $\Phi_D(x) = [D - x]$.
Then the break divisor $\br[D]$ is equal to
\[ \br[D] = \Phi_D^{-1}( \Phi_D(\Gamma) \cap^\mathrm{st} \Theta)\]
where $\Theta$ is the theta divisor and 
$\cap^\mathrm{st}$ denotes stable tropical intersection.
	
	%\footnote{
(The stable tropical intersection may have multiplicities,
so we interpret the preimage to be a multiset  in $\Gamma$
carrying the same multiplicities.)
	%}
\end{prop}
\begin{proof}
Let us denote $\br^*[D] :=  \Phi_D^{-1}( \Phi_D(\Gamma) \cap^\mathrm{st} \Theta)$.
For a generic divisor class $[D]\in \Pic^g(\Gamma)$,
the intersection $\Phi_D(\Gamma) \cap\Theta$ is transverse so 
\[ \br^*[D] = \{ x \in \Gamma : [D-x] \geq 0 \} ,\]
i.e. $\br^*[D]$ contains the support of any effective representative of $[D]$.
Generically, the class $[D]$ contains a single effective representative so
$\br^*: \Pic^g(\Gamma)\to \Sym^g(\Gamma)$ defines a generic section of the linear equivalence map
$\Sym^g(\Gamma) \to \Pic^g(\Gamma)$.

By general properties of stable tropical intersection,
the map $\br^* : \Pic^g(\Gamma) \to \Sym^g(\Gamma)$
is continuous.
But by Theorem~\ref{thm:break-div},
the break divisor map
$\br$ is the unique continuous section of $\Sym^g(\Gamma) \to \Pic^g(\Gamma)$
so we must have $\br^*[D] = \br[D]$.
\end{proof}

\section{Tropicalizing Weierstrass points}
\label{sec:tropicalization}
In this appendix, we describe how the Weierstrass locus for a tropical curve can be related to the Weierstrass locus for an algebraic curve.
The key result is Baker's Specialization Lemma \cite[Lemma 2.8]{Bak}; here we use a more general version given by Jensen--Payne \cite{JP} in the language of Berkovich analytic spaces.
% The results of this section are not needed for any later sections of the paper.

Throughout this section, let $\KK$ denote an algebraically closed field equipped with a nontrivial non-Archimedean valuation $v: \KK^\times \to \RR$;
we assume $\KK$ is complete with respect to~$v$.

\begin{thm}[Specialization Lemma {\cite[Lemma 2.4]{JP}}]
\label{thm:specialization}
Suppose $X$ is a smooth projective algebraic curve over $\KK$.
Let $\Gamma$ be a skeleton on the Berkovich analytification $\Xan$,
let $\rho : \Xan \to \Gamma$ be the retraction to the skeleton
and let $\rho_* : \Div(X) \to \Div(\Gamma)$
denote the induced map on divisors.
Then for any divisor $D \in \Div(X)$,
\begin{equation*}
r_X(D) \leq r_\Gamma(\rho_*(D)) .
\end{equation*}
\end{thm}
Here $r_X$ denotes the dimension of a complete linear system $|D|$ on $X$,
and $r_\Gamma$ denotes the Baker--Norine rank on $\Gamma$
(see Section~\ref{sec:rank}).
The map $\rho_* : \Div(\Xan) \to \Div(\Gamma)$ is called the {\em specialization map} by Baker~\cite{Bak}.

\begin{thm}
\label{thm:w-tropicalize}
Consider the setup of Theorem~\ref{thm:specialization}.
%Suppose $X$ is a smooth projective algebraic curve over  $\KK$.
%Let $\Gamma$ be a skeleton of the Berkovich analytifcation $\Xan$,
%and let $\rho : \Div(X) \to \Div(\Gamma)$ denote the specialization map.
For any divisor $D\in \Div(X)$ such that $\rho_*(D) \in \Div(\Gamma)$
is Riemann--Roch nonspecial, we have
\begin{equation*}
 \rho_*( W_X(D)) \subseteq  W_\Gamma(\rho_*(D)) .
\end{equation*}
\end{thm}
\begin{proof}
The map $\rho_*$ respects degree;
let $n = \deg(D)  = \deg(\rho_*(D))$.
Recall that $\rho_*(D)$ is {\em nonspecial} means that
\[ r_\Gamma(\rho_*(D)) = \max\{ n- g,\, -1\} .\]
In this case, Theorem~\ref{thm:specialization} implies 
$r_X(D) \leq \max\{ n-g,\, -1\}$
while Riemann--Roch implies
$ r_X(D) \geq \max\{ n-g,\, -1\}$ for any divisor.
Thus $r_X(D) = r_\Gamma(\rho_*( D))$.

Let $r$ denote the rank in either sense.
If $x \in W_X(D)$,  we have
\[  r_X(D - (r+1)x) \geq 0 .\]
By Theorem~\ref{thm:specialization} and linearity of $\rho_*$, this implies
\[ r_\Gamma(\rho_*(D - (r+1)x )) = r_\Gamma( \rho_*(D) - (r+1) \rho_*(x) ) \geq 0.\]
This means $\rho_*(x) \in W_\Gamma(\rho_*(D))$ as claimed.
\end{proof}
The conclusion of Theorem~\ref{thm:w-tropicalize}
also holds for $D = K_X$ the canonical divisor, and $\rho_*(K_X) \sim K_\Gamma$.
This was observed by Baker in \cite[Corollary 4.9]{Bak}.

As mentioned in the introduction, the result in Theorem~\ref{thm:w-finite} on the number of tropical Weierstrass points and the classical count of Weierstrass points on curves suggest the following conjecture.

\begin{conj}
	In the setup of Theorem~\ref{thm:specialization}, suppose $\Xan$ and $\Gamma$ have genus $g$, and divisors $D$ and $\rho_*(D)$ have degree $n \geq g$ and are nonspecial.
	Suppose that the Weierstrass points in $W_X(D)$ all have weight one,
	and the Weierstrass locus $W_\Gamma(\rho_*(D))$ contains $g(n - g + 1)$ distinct points.
	Then the specialization map
	\[
		\rho: W_X(D) \to W_\Gamma(\rho_*(D))
	\]
	on Weierstrass points is $(n - g + 1)$-to-$1$.
\end{conj}

\section*{Acknowledgements}
I am grateful to David Speyer for introducing me to tropical geometry and 
guiding my interest in the subject.
I thank Farbod Shokrieh and Matthew Baker for helpful suggestions and comments.
In particular, 
the definition of the stable Weierstrass locus
was a result of  Matthew Baker's advice
that  break divisors should play an essential role.
This work was partially supported by NSF grant DMS-1600223
and a Rackham Predoctoral Fellowship.

\bibliography{weierstrass-ref} 
\bibliographystyle{abbrv}

\end{document}